\documentclass[letterpaper,11pt,reqno]{amsart}
\usepackage[margin=1in]{geometry}
\usepackage{hyperref}
\hypersetup{
     colorlinks   = true,
     citecolor    = black,
     linkcolor    = blue
}

\usepackage{graphicx,amsmath,amssymb,amsthm,paralist,color,tikz-cd}
\usepackage{mathrsfs}
\usepackage{mathtools}
\usepackage{cite}
\usepackage{setspace}

\usepackage{amscd}

\usepackage{bbm} 

\newtheorem{theorem}{Theorem}[section]
\newtheorem{lemma}[theorem]{Lemma}

\newtheorem{theoremAlph}{Theorem}

\newtheorem{question}[theorem]{Question}

\newtheorem{proposition}[theorem]{Proposition}

\newtheorem{corollary}[theorem]{Corollary}

\theoremstyle{definition}
\newtheorem{definition}[theorem]{Definition}
\newtheorem*{definition-nono}{Definition}

\newtheorem{assumption}[theorem]{Assumption}

\newtheorem{example}[theorem]{Example}

\newtheorem{remark}[theorem]{Remark}

\newcommand{\N}{\mathbb{N}}
\newcommand{\Z}{\mathbb{Z}}
\newcommand{\Q}{\mathbb{Q}}
\newcommand{\R}{\mathbb{R}}
\newcommand{\C}{\mathbb{C}}

\newcommand{\mc}{\mathcal}

\newcommand{\mf}{\mathfrak}
\newcommand{\mbf}{\mathbf}
\newcommand{\mrm}{\mathrm}

\renewcommand{\a}{\alpha}
\renewcommand{\b}{\beta}
\newcommand{\g}{\gamma}
\newcommand{\G}{\Gamma}
\renewcommand{\d}{\delta}
\newcommand{\e}{\varepsilon}
\renewcommand{\l}{\lambda}
\renewcommand{\L}{\Lambda}
\newcommand{\w}{\omega}
\newcommand{\s}{\sigma}
\newcommand{\vp}{\varphi}

\renewcommand{\k}{\kappa}

\newcommand{\set}[1]{\left\{#1\right\}}
\renewcommand{\r}{\rightarrow}

\def\multiset#1#2{\ensuremath{\left(\kern-.3em\left(\genfrac{}{}{0pt}{}{#1}{#2}\right)\kern-.3em\right)}}
\newcommand{\norm}[1]{\left\lVert#1\right\rVert}

\numberwithin{equation}{section}

\title[Height Functions and Expanding Curves]{Bounded and Divergent Trajectories And Expanding Curves on Homogeneous Spaces}
\author{OSAMA KHALIL}
\address{Department of Mathematics, Ohio State University, Columbus, OH}
\email{khalil.37@osu.edu}


\date{}

\begin{document}

\begin{abstract}
Suppose $g_t$ is a $1$-parameter $\mathrm{Ad}$-diagonalizable subgroup of a Lie group $G$ and $\Gamma < G$ is a lattice.
We study the dimension of bounded and divergent orbits of $g_t$ emanating from a class of curves lying on leaves of the unstable foliation of $g_t$ on the homogeneous space $G/\G$.
We obtain sharp upper bounds on the Hausdorff dimension of divergent on average orbits and show that the set of bounded orbits is winning in the sense of Schmidt (and, hence, has full dimension).
The class of curves we study is roughly characterized by being tangent to copies of $\mrm{SL}(2,\R)$ inside $G$, which are not contained in a proper parabolic subgroup of $G$.

 We describe applications of our results to problems in Diophantine approximation by number fields and intrinsic Diophantine approximation on spheres. Our methods also yield the following result for lines in the space of square systems of linear forms: suppose $\varphi(s) = sY + Z$ where $Y\in \mathrm{GL}(n,\mathbb{R})$ and $Z\in M_{n,n}(\mathbb{R})$. Then, the dimension of the set of points $s$ such that $\varphi(s)$ is singular is at most $1/2$ while badly approximable points have Hausdorff dimension equal to $1$.
\end{abstract}

\maketitle
{\hypersetup{linkcolor=black}
\tableofcontents
}
    \section{Introduction}

\subsection{Summary of the results}

    The purpose of this article is to study the Hausdorff dimension of bounded and divergent orbits of diagonalizable flows emanating from curves on homogeneous spaces.
    The motivation for studying these problems comes from the theory of Diophantine approximation. 
    The class of curves we study is roughly characterized by being tangent to maximal representations of $\mrm{SL}(2,\R)$ into the ambient Lie group $G$. These are representations whose images are not contained in a proper parabolic subgroup $G$.
    See Definition~\ref{defn: deformations} for a precise description.
    In this setting, we provide a sharp upper bound on the dimension of divergent on average trajectories (Definition~\ref{defn: divergence on average}) and show that bounded orbits are winning for a Schmidt game on intervals of the real line (see Section~\ref{section: schmidt games} for detailed definitions).
    Moreover, we establish, in a quantitative form, the non-divergence of push-forwards of shrinking curve segments (cf. Proposition~\ref{propn: non-divergence of shrinking curves}).
    
    For concreteness, we state our results in the introduction in the examples which are most relevant to applications in Diophantine approximation, deferring the more general statements to Theorems~\ref{thrm: Hdim and non-divergence},~\ref{thrm: CH for maximal representations}, and~\ref{thrm: products of so(n,1)}.
    These concrete examples include homogeneous spaces of products of real rank $1$ Lie groups (Theorem~\ref{thrm: rank 1 DOA}), a more general class of curves on homogeneous spaces of products of $\mrm{SO}(n,1)$ (Theorem~\ref{thrm: products of sl2} and~\ref{thrm: products of so(n,1)}), and actions of $\mrm{SL}(2,\R)$ on any homogeneous space of finite volume (Theorem~\ref{thrm: sl2 actions}).
    Curves on more general arithmetic homogeneous spaces are studied in Section~\ref{section: CH higher rank}.

	In Section~\ref{section: applications}, we present applications of our results to problems in intrinsic Diophantine approximation on spheres (Corollary~\ref{cor: dioph approx on Sn}), Diophantine approximation by number fields (Corollary~\ref{cor: dioph approx in num fields}), and Diophantine properties of lines in the space of square systems of linear forms $M_{n,n}$ (Corollary~\ref{cor: dioph approx of linear forms}).

\subsection{Historical context}
    
    To the best of our knowledge, the problem of the dimension of divergent orbits starting from curves has not been previously addressed in the literature.
	Among the motivations for studying this problem is a well-known deep conjecture, due to Wirsing, concerning the approximability of transcendental numbers by algebraic numbers of bounded degree. By the work of Bugeaud and Laurent, the Hausdorff dimension of the set of counterexamples to Wirsing's conjecture in degree $n$ is bounded above by the dimension of singular vectors in $M_{1,n}\cong\R^n$ lying on the Veronese curve $\set{(\xi,\xi^2,\dots,\xi^n):\xi\in\R}$ \cite{BugeaudLaurent}.
    
    To place our results in context, we briefly survey the history of the subject.
    In~\cite{Dani-Bounded,Dani-Bounded2}, Dani studied the problem of bounded orbits in two settings: orbits of diagonalizable flows on homogeneous spaces of rank $1$ Lie groups and orbits in $\mrm{SL}(m+n,\R)/\mrm{SL}(m+n,\Z)$ of the form $g_t u_Y \G$, where, for $t\in \R$ and $Y\in M_{m,n}$ an $m\times n$ real matrix,
    \begin{equation} \label{Dani's thrms}
    g_t = \mrm{diag}(e^{t/m},\dots,e^{t/m},e^{-t/n},\dots,e^{-t/n}),\quad u_Y = \begin{pmatrix}
    		\mathrm{I}_m & Y \\ \mathbf{0} & \mathrm{I}_n \end{pmatrix}.
    \end{equation}
    We refer to $g_t$ as a diagonal element with weight $(1/m,\dots,1/m,1/n,\dots,1/n)$.
    It is shown that bounded orbits of diagonalizable flows on rank $1$ homogeneous spaces have full Hausdorff dimension.
    It is also shown that orbits of the form $(g_t u_Y \G)_{t\geqslant 0}$ are bounded if and only if $Y$ is badly approximable, i.e., there exists $\d>0$ such that for all $(\mbf{p},\mbf{q}) \in \Z^m\times \Z^n$, with $\mbf{q} \neq 0$,
    \begin{equation*}
    	\norm{Y \mbf{q} -\mbf{p} }^n \norm{\mbf{q}}^m > \d.
    \end{equation*}
	Using the results of Schmidt on badly approximable systems of linear forms~\cite{Schmidt-BadLinearSystems}, this implies that bounded orbits for $g_t$ as in~\eqref{Dani's thrms} have full dimension.
    These results were generalized in~\cite{KleinbockWeiss-SchmidtSLn+m,KleinbockWeiss-SchmidtExpanding} where bounded orbits of \emph{non-quasiunipotent} flows were shown to have full dimension.
    
    All of these results were obtained by showing that bounded orbits are winning for variants of a game invented by Schmidt in~\cite{Schmidt-Games}.
    The winning property is much stronger than having full Hausdorff dimension since it is stable under countable intersections and implies thickness, i.e., the intersection of a winning set with any non-empty open set has full dimension. We refer the reader to~\cite{KleinbockWeiss-SchmidtSLn+m} for more details on Schmidt's original game as well as a new variant introduced by the authors.
    More recently, far reaching generalizations of these results were obtained in~\cite{BadziahinEtal-SchmidtConjecture}, in particular settling an old conjecture of Schmidt on the intersection of sets of weighted badly approximable vectors with different weights.
    
    Dani also studied the existence and classification of divergent orbits of diagonalizable flows on homogeneous spaces in~\cite{Dani-Divergent}.
    Among the results obtained by Dani is the fact that divergent orbits on non-compact homogeneous spaces of a rank $1$ Lie group $G$ are \emph{degenerate}, i.e., can be detected using the behavior of finitely many vectors in some fixed representation of $G$.
    In particular, the set of divergent orbits consists of a countable collection of immersed submanifolds in $G/\G$.
    This result also holds for quotients of Lie groups by arithmetic lattices of rational rank $1$.
    By contrast, quotients by higher rank arithmetic lattices always admit non-degenerate divergent orbits~\cite{Dani-Divergent,Weiss-Divergent}.
    
    In a landmark paper, the precise Hausdorff dimension of divergent orbits under the flow induced by $g_t$ in~\eqref{Dani's thrms} was calculated when $(m,n)=(2,1)$ in~\cite{Cheung-SingularPairs}. This result was extended in~\cite{CheungChevallier} to the case when $\min(m,n)=1$.
    These results build on earlier ideas of Cheung in~\cite{Cheung-Sl2products} where the Hausdorff dimension of divergent orbits in $\mrm{SL}(2,\R)^n/\mrm{SL}(2,\Z)^n$ for $n\geq 2$ under the flow induced by a diagonal matrix in each coordinate was determined to be $3n-1/2$.
    In~\cite{KKLM-SingSystems}, a sharp upper bound on the dimension of divergent orbits for general $m$ and $n$ was obtained by different methods.
    The proof in~\cite{KKLM-SingSystems} relies on the powerful technique of systems of integral inequalities introduced in~\cite{EskinMargulisMozes} in the context of quantifying Margulis' work on the Oppenheim conjecture.

    Parallel to these developments and motivated by problems in Diophantine approximation, the study of the evolution of curves on homogeneous spaces under diagonal flows attracted a lot of interest.
    In~\cite{KleinbockMargulis}, Kleinbock and Margulis showed that the push-forward of certain ``non-degenerate" smooth curves in the group $\set{u_Y:Y \in M_{1,n}}$ by diagonal elements similar to $g_t$ in~\eqref{Dani's thrms} do not diverge in $\mrm{SL}(n+1,\R)/\mrm{SL}(n+1,\Z)$.
    This allowed them to settle a conjecture due to Baker and Sprind\v{z}uk showing that the Lebesgue measure of very well approximable vectors belonging to such curves is $0$.
    This result has been generalized in numerous directions, cf.~\cite{KleinbockLindenstraussWeiss,BeresnevichKleinbock-WeakNonPlanar,Aka-ExtremalCriterion} for notable examples.
    
    In~\cite{Shah-Inventiones,Shah-JAMS}, using Ratner's theorems and the linearization technique, Shah extended the results of Kleinbock and Margulis by showing that the push-forwards of the parameter measure on these curves, in fact, become equidistributed towards the Haar measure on $G/\G$.
    These results build on earlier work of Shah in~\cite{Shah-Duke1,Shah-Duke2} where the push-forwards of certain smooth curves on the unit tangent bundle of hyperbolic manifolds by the geodesic flow were shown to be equidistributed towards the Haar measure.
    
    On the other hand, the problem of determining the Hausdorff dimension of bounded and divergent orbits restricted to curves as above is far less understood.
    In a breakthrough article, Beresnevich showed in~\cite{Beresnevich-BadCurves} that the Hausdorff dimension of finite intersections of weighted badly approximable vectors on non-degenerate analytic curves in $M_{1,n}$ is full.
	By means of Dani's correspondence, this implies that bounded orbits of diagonal elements similar to $g_t$ in~\eqref{Dani's thrms} with more general weights than $(1,1/n,\dots,1/n)$ starting from points on curves on the group $\set{u_Y:Y \in M_{1,n}}$ is equal to $1$.
    We refer the reader to~\cite{Beresnevich-BadCurves} for more on the history of this problem and to~\cite{AnBeresnevichVelani-WinningPlanarCurves} where these bounded orbits were shown to be in fact winning in the sense of Schmidt for planar curves.
    The dimension of bounded orbits starting from curves on other homogeneous spaces was studied in~\cite{Aravinda} in rank $1$ homogeneous spaces and in~\cite{Einsiedler-CurvesNumfields} in quotients of $\mrm{SL}(2,\R)^r\times \mrm{SL}(2,\C)^s$ by irreducible lattices.



\section{Main Results}

\subsection{Preliminary Notions}
Before stating our main results, we need to introduce necessary definitions and notation.
	Given a real Lie group $G$, we denote by $\mf{g}$ its Lie algebra.
    For a $1$-parameter subgroup $g_t$ of $G$, we say $g_t$ is $\mrm{Ad}$-diagonalizable over $\R$ if $\mf{g}$ decomposes over $\R$ under the Adjoint action of $g_t$ into eigenspaces.
    \[ \mf{g} = \bigoplus_{\a \in \R} \mf{g}_\a, \qquad \mf{g}_\a = \set{Z\in \mf{g}: \mrm{Ad}(g_t)(Z) = e^{\a t}Z }. \]
	We remark that the decomposition above is only an eigenspace decomposition with respect to $\mrm{Ad}(g_t)$, not a decomposition into root spaces.
    Suppose that $G$ acts on a metric space $X$.
	Our goal is to study the Hausdorff dimension of certain orbits of $g_t$ on $X$ with prescribed recurrence properties.
    For that purpose, let us make precise the recurrence notions we shall be interested in.

\begin{definition-nono}
For a flow $g_t :X\r X$ on a metric space $X$ and $y\in X$, we say the (forward) orbit $g_t y$ is \textbf{divergent on average}, if for any compact set $Q \subset X$, one has
    \begin{equation} \label{defn: divergence on average}
    	\lim_{T\r\infty} \frac{1}{T} \int_0^T \chi_Q(g_ty)\;dt =0
    \end{equation}
    where $\chi_Q$ denotes the indicator function of $Q$. We say the orbit $g_t y$ is \textbf{bounded} if $\overline{\set{g_ty :t>0}}$ is compact.
    The orbit $g_t y$ is said to have \textbf{linear growth} if for some base point $y_0$, we have
    \begin{equation}
    	\limsup_{t\r\infty}  \frac{d(g_ty, y_0)}{t}  >0,
    \end{equation}
    where $d(\cdot,\cdot)$ is the metric on $X$.
\end{definition-nono}

	Finally, recall that a subset $A$ of a metric space is \textbf{thick} if the intersection of $A$ with every non-empty open set has full Hausdorff dimension.


\subsection{Homogeneous Spaces of Products of Rank One Lie Groups}
	Our first result is in the setting of homogeneous spaces of Lie groups of the form $G = G_1 \times \cdots \times G_k$, where each $G_i$ is a real rank one Lie group.
    To state the result, we need some preparation.
    
    Suppose $\G$ is any lattice in $G$.
    Then, we can write $\G = \G_1 \times \cdots \times \G_l$ (up to finite index), where each $\G_j$ is an irreducible lattice in a sub-product of $G$, which we denote by $H_j$.    
   By Margulis' arithmeticity theorem, if for some $1\leq j\leq l$, $H_j$ is a product of more than $1$ factor (i.e. $\mrm{rank}_\R (H_j)>1$), then there exists a rational structure on $H_j$ in which $\G_j$ is arithmetic, i.e. $\G_j$ is commensurable with $H_j(\Z)$.

    We say that a $1$-parameter subgroup $g_t$ of $G$ is \textbf{split} if the projection of $g_t$ onto each higher rank factor $H_j$ is $\mrm{Ad}$-diagonalizable over $\Q$ with respect to the $\Q$-structure in which $\G_j$ is arithmetic.
    The following maps into $\mf{g}$ are the main object of study in this setting.
    \begin{definition-nono} 
    For a compact interval $B\subset \R$ and an $\mrm{Ad}$-diagonalizable subgroup $g_t$, we say a differentiable map $\vp: B \r \mf{g}$ is $\mathit{g_t}$\textbf{-admissible} if the image of $\vp$ is contained in a single eigenspace $\mf{g}_\a$ for some $\a>0$ and $[\vp,\dot{\vp}] \equiv 0$ on $B$. For every $s$, we denote by $u(\vp(s))$ the image of $\vp(s)$ in $G$ under the exponential map.
    \end{definition-nono}

    Denote by $\mf{g}_i$ the Lie algebra of $G_i$. The following is the first main theorem of this article.

    \begin{theoremAlph} \label{thrm: rank 1 DOA}
    Suppose $G = G_1 \times \cdots \times G_k$ , where each $G_i$ is a simple Lie group of real rank $1$ and finite center and $\G$ is any lattice in $G$. For each $1\leq i\leq k$, let $g_t^{(i)}$ be a non-trivial $1$-parameter subgroup of $G_i$ which is $\mrm{Ad}$-diagonalizable over $\R$, and suppose $ \vp_i: B \r \mf{g}_i $ is a $g_t^{(i)}$-admissible $C^2$-map.
    Let $g_t = (g_t^{(i)})_{1\leq i\leq k}$ and $\vp = \oplus_{i=1}^k \vp_i$. Assume that $g_t$ is split and that $\vp$ is $g_t$-admissible.
    Define the following set.
    \[ Z = \set{ s\in B: \dot{\vp}_i(s) = 0 \text{ for some } 1\leq i\leq k }. \]
	Then, for every $x_0\in X=G/\G$, the following hold.
        \begin{enumerate}[(i)]
        \item The Hausdorff dimension of the set of points $s\in B\backslash Z $ for which the orbit $(g_tu(\vp(s))x_0)$ is divergent on average as $t\r\infty$ is at most $1/2$.        
        \item For any compact interval $V \subseteq B\backslash Z$, the set of points $s\in V$ for which the orbit $(g_tu(\vp(s))x_0)_{t\geqslant 0}$ is bounded in $X$ is winning for a Schmidt game on $V$ induced by $g_t$.
        In particular, this set is thick in $B\backslash Z$. 
        \item For almost every $s\in B \backslash Z $, any weak-$\ast$ limit of the measures $\frac{1}{T}\int_0^T \d_{g_t u(\vp(s))x_0}ds$ is a probability measure on $X$.
        \item The set of points $s\in B\backslash Z$ for which the forward orbit $(g_tu(\vp(s))x_0)_{t\geqslant 0}$ has linear growth has Lebesgue measure $0$.
        \end{enumerate}
    \end{theoremAlph}
    
    \begin{remark}
    In studying divergent orbits, it is necessary for our methods that the diagonalizable flow we consider in Theorem~\ref{thrm: rank 1 DOA} expands the curve by the same amount in every coordinate. In Theorems~\ref{thrm: products of sl2} and~\ref{thrm: products of so(n,1)} below, we relax this assumption, where we allow some of the coordinate flows $g_t^{i}$ to be trivial.
    It is an interesting question as to whether similar results hold for more general diagonal flows.
    \end{remark}
    
    We refer the reader to Section~\ref{section: schmidt games} for details on Schmidt games and a more precise form of part $(ii)$ of Theorem~\ref{thrm: rank 1 DOA}.
    Number theoretic corollaries of Theorem~\ref{thrm: rank 1 DOA} concerning intrinsic Diophantine approximation on spheres are discussed in Section~\ref{section: diophantine Sn}.
    
    We note that the assumption in Theorem~\ref{thrm: rank 1 DOA} that $\vp = \oplus_{i=1}^k \vp_i$ is $g_t$-admissible amounts to ensuring that the eigenspace of $\mrm{Ad}(g_1^{(i)})$ containing the image of $\vp_i$ corresponds to the same eigenvalue for each $i$.
    Moreover, the restriction to the points in $B\backslash Z$ is natural since it is possible for the map $\vp$ to map a sub-interval of $B$ onto a point whose orbit is divergent.

    \begin{remark}
    The proof of Theorem~\ref{thrm: rank 1 DOA} is reduced to the case when $\G$ is an irreducible lattice in $G$. When $\mrm{rank}_\R G >1$, $\G$ is an arithmetic lattice by Margulis' arithmeticity theorem.
    In that case, Theorem~\ref{thrm: rank 1 DOA} is a special case of a more general result we obtain for quotients of semisimple algebraic Lie groups by arithmetic lattices, Theorem~\ref{thrm: CH for maximal representations}.
    \end{remark}

     In~\cite{Aravinda}, in the setting of rank one locally symmetric spaces, it is shown that bounded orbits under the geodesic flow restricted to non-constant $C^1$-maps on the unit tangent sphere around a point is winning in the sense of Schmidt. The methods in~\cite{Aravinda} rely on the geometry of rank $1$ locally symmetric spaces. Our proof is completely different and remains valid in more generality. Theorems~\ref{thrm: products of sl2} and~\ref{thrm: sl2 actions} below are other instances where our methods also apply. We refer the reader to Theorems~\ref{thrm: Hdim and non-divergence} and~\ref{thrm: schmidt games} where we show an analogous statement to Theorem~\ref{thrm: rank 1 DOA} in the abstract setting of Lie group actions on metric spaces satisfying certain recurrence hypotheses.
     
     \begin{remark}
    If we assume the image of a coordinate function $\vp_i$ is contained in an abelian subspace of $\mf{g}_i$, we can weaken the regularity condition on $\vp_i$ to be $C^{1+\e}$ for some $\e>0$.
      In particular, Theorem~\ref{thrm: rank 1 DOA} holds for $C^{1+\e}$-maps when $G_i\cong \mrm{SO}(d_i,1)$ for each $1\leq i\leq k$.
    \end{remark}

    Using a result in~\cite{KadyrovPohl}, we deduce a lower bound on the dimension of the divergent on average orbits considered in Theorem~\ref{thrm: rank 1 DOA} in a special case which agrees with the upper bound we obtain.
    We further discuss the sharpness of this bound, as well as the bounds obtained in the results below, in Section~\ref{section: conclusion}.
    
    \begin{corollary} \label{cor: lower bound cor}
    In the notation of Theorem~\ref{thrm: rank 1 DOA}, suppose $G/\G = (\mrm{SL}(2,\R)/\G_1) \times (G'/\G')$, where $\G_1$ is a non-cocompact lattice in $\mrm{SL}(2,\R)$.
    Assume further that $\vp_1$ is non-constant.
	Then, for every $x_0\in G/\G$, the Hausdorff dimension of the set of points $s\in B\backslash Z$ such that the orbit $(g_t u(\vp(s)) x_0)_{t\geqslant 0}$ is divergent on average is exactly $1/2$.
    \end{corollary}
	
\subsection{Non-maximal Curves and Restrictions of Scalars of SL(2)}
	In Theorem~\ref{thrm: rank 1 DOA}, every coordinate of the map $\vp$ is assumed to be non-constant.
    However, our methods apply in more general situations. This is the content of our next result in the setting where $G =  \mrm{SL}(2,\R)^r \times \mrm{SL}(2,\C)^s$ for some $r,s\in \N$.
    The motivation for studying these problems in this particular setting comes from questions in Diophantine approximation with number fields.

    For $g\in G$, we denote by $U^+(g)$ the expanding horospherical subgroup of $G$ associated with $g$ and by $\mrm{Lie}(U^+(g))$ its Lie algebra.
    We also use $u(z)$ to denote $\exp(z)$ for $z\in \mrm{Lie}(U^+(g))$.
    For $t\in \R$ and $ \mbf{x}=(\mbf{x}_i) \in \R^r\times \C^s$, let
          \[  a_t = \left( \begin{pmatrix} e^t & 0 \\ 0 & e^{-t} \end{pmatrix} \right)_{1\leqslant i \leqslant r+s},
    \qquad  
    u(\mathbf{x})= \left( \begin{pmatrix} 1& \mathbf{x}_i \\ 0 & 1 \end{pmatrix} \right)_{1\leqslant i \leqslant r+s}. \]
    Note that $U^+(a_1) = \set{ u(\mbf{x}):  \mbf{x} \in \R^r\times \C^s}$ and for all $g\in G$,
    $U^+(ga_1 g^{-1}) = g U^+(a_1) g^{-1}$.

   For each $k$, let us write $G_k = \mrm{SL}(2,\R)^{r_k} \times \mrm{SL}(2,\C)^{s_k}$.
   Thus, we can make the following identifications.
   \[  \mrm{Lie}(U)^+(a_1) \cong \R^r\times \C^s \cong 
   		\bigoplus_{k=1}^l \R^{r_k}\oplus \C^{s_k}. \]
   
   Given a map $\psi=(\psi_i):B \r \R^a\times \C^b$ such that $\psi \not\equiv 0$, where $B\subset\R$, the \textbf{characteristic} of $\psi$, denoted by $\mrm{char}(\psi)$ is defined to be
   \begin{align}\label{eqn: characteristic}
   \mrm{char}(\psi) =   \frac{ \# \set{1\leqslant i \leqslant a: \psi_i \equiv  0 } 
   			+ 2 \cdot \# \set{a< i \leqslant a+b: \psi_i \equiv  0 }}
            {\# \set{1\leqslant i \leqslant a: \psi_i \not\equiv  0 } 
   			+ 2 \cdot \# \set{a< i \leqslant a+b: \psi_i \not\equiv  0 }}.
   \end{align}
    We can now state our main result in this setting.
    
    \begin{theoremAlph} \label{thrm: products of sl2}
    	Suppose $G = G_1 \times \cdots \times G_l$ is as above, $\G = \G_1 \times \cdots \times \G_l$ such that $\G_k$ is an irreducible lattice in $G_k$, and
        $g_t$ is a split $1$-parameter subgroup which is conjugate to $a_t$.
        For $1\leqslant k \leqslant l$, let $\vp_k : B \r  \R^{r_k}\oplus \C^{s_k}$ be a $C^{1+\e}$-map for some $\e>0$ and let $\vp = \oplus_k \vp_k: B \r \mrm{Lie}(U^+(g_1)) \cong \bigoplus_{k=1}^l \R^{r_k}\oplus \C^{s_k}$.
        Denote by $(\vp_k)_i$ the $i^{th}$ coordinate of $\vp_k$ and let       
        \begin{equation*}
        Z = \set{s\in B:   (\dot{\vp}_k)_i(s) = 0, (\dot{\vp}_k)_i \not\equiv 0 \textrm{ for some }k,i}.
        \end{equation*}
        Assume that $\vp$ is not a constant map.
        Then, for every $x_0 \in X=G/\G$, the Hausdorff dimension of the set of points $s\in B\backslash Z$ for which the forward trajectory $ (g_tu(\vp(s))x_0)_{t\geqslant 0}$ is divergent on average is at most
        \begin{equation*}
        \frac{1}{2} + \frac{1}{2} \max_{1\leq k\leq l}
        \mrm{char}(\dot{\vp}_k).
        \end{equation*}
        Moreover, if the above quantity is strictly less than $1$, then parts $(ii)-(iv)$ of Theorem~\ref{thrm: rank 1 DOA} also hold in this setting.
	\end{theoremAlph}

    We remark that the upper bound in Theorem~\ref{thrm: products of sl2} is strictly less than $1$ if and only if
    \begin{equation} \label{intro max condition}
    	\# \set{1\leqslant i \leqslant r_k: (\dot{\vp}_k)_i \not\equiv  0 } 
   			+ 2 \cdot \# \set{r_k< i \leqslant r_k+s_k: (\dot{\vp}_k)_i \not\equiv  0 } > \frac{r_k+2s_k}{2},
    \end{equation}
    for all $1\leq k \leq l$.
    
    \begin{remark}
    An analogue of Theorem~\ref{thrm: products of sl2} holds for other products of real rank $1$ Lie groups.
    The upper bound formula for the dimension of divergent orbits will depend on the factors in the product, but the rest of the proof goes through verbatim.
    We refer the reader to Theorem~\ref{thrm: products of so(n,1)} for a result for products of copies of $\mrm{SO}(n,1)$.
    \end{remark}
    
    The bounded orbits in Theorem~\ref{thrm: products of sl2} were shown to be winning in the sense of Schmidt in~\cite{Einsiedler-CurvesNumfields} for $C^1$-curves $\vp$ satisfying~\eqref{intro max condition} and $\G$ an irreducible lattice. Our methods are rather different in flavor and apply to a wider class of examples.
    Moreover, equidistribution of translates by $g_t$ of submanifolds of $U^+(g_1)$ of small codimension and satisfying certain curvature conditions was established in~\cite{Ubis}.

    Applications of Theorem~\ref{thrm: products of sl2} to Diophantine approximation by number fields are discussed in Section~\ref{section: diophantine number fields}.

\subsection{SL(2,R) Actions on Homogeneous Spaces}
	
    The motivation for our next result comes from problems in Diophantine approximation of square systems of linear forms.
    In particular, Theorem~\ref{thrm: sl2 actions} below is used to study the Hausdorff dimension of singular and badly approximable square systems of linear forms belonging to a straight line with an invertible slope (Corollary~\ref{cor: dioph approx of linear forms}).

\begin{theoremAlph} \label{thrm: sl2 actions}
    Let $B \subset \R$ be an interval and suppose $L$ is a semisimple algebraic Lie group defined over $\Q$, $\G$ an arithmetic lattice in $L$, and $\rho: \mrm{SL}(2,\R) \r L $ a non-trivial representation. Let
          \[ g_t = \rho \left( \begin{pmatrix}e^t &0\\ 0 & e^{-t}    \end{pmatrix}  \right), \qquad u(\vp(s)) = \rho  \left( \begin{pmatrix}1 &s\\ 0 & 1    \end{pmatrix}  \right), s\in B. \]
	Then, for every $x_0 \in X= L/\G$, $(i)-(iv)$ of Theorem~\ref{thrm: rank 1 DOA} hold in this setting.
	\end{theoremAlph}

    \begin{remark}
    An analogue of Theorem~\ref{thrm: sl2 actions} is known for the action of $\mrm{SL}(2,\R)$ on strata of abelian differentials. The $1/2$ upper bound on the dimension of divergent orbits was established by Masur in~\cite{Masur-NUE}.
    This was recently extended in~\cite{AAEKMU} to show that this upper bound in fact holds for divergent on average orbits.
    Moreover, Kleinbock and Weiss showed that bounded orbits in that setting have full Hausdorff dimension in~\cite{KleinbockWeiss-BoundedModuli}.
    The winning property of bounded orbits was later obtained in~\cite{ChaikaCheungMasur-WinningModuli}.
    The proof of Theorem~\ref{thrm: sl2 actions} uses the method of height functions and integral inequalities and is valid for $\mrm{SL}(2,\R)$ actions on general metric spaces satisfying the hypotheses of Theorem~\ref{thrm: schmidt games}.
    In particular, the work of Eskin and Masur in~\cite{EskinMasur-HeightStrata} establishes these hypotheses in the setting of $\mrm{SL}(2,\R)$ actions on strata of abelian differentials.
    \end{remark}


\subsection{Paper Organization and Overview of Proofs}


	In Section~\ref{section: applications}, we discuss applications of our main results to problems in Diophantine approximation.
    In Section~\ref{section: abstract setup}, we prove a general result for Lie group actions on metric spaces which implies the upper bound on the dimension of divergent on average orbits as well as the almost sure non-divergence result of Theorems~\ref{thrm: rank 1 DOA} (parts $(i)$ and $(iii)$),~\ref{thrm: products of sl2} and~\ref{thrm: sl2 actions} as soon as the assumptions are verified. 
    
    The winning property of bounded trajectories is also obtained for general Lie group actions in Section~\ref{section: schmidt games}, where we discuss Schmidt's game in detail.
    Finally, part $(iv)$ of the above theorems concerning growth of orbits is established under these abstract hypotheses in Section~\ref{section: shrinking nondivergence} where the quantitative non-divergence of expanding translates of shrinking curve segments is established.
    
    These general results assume the existence of a certain ``height function" encoding recurrence of orbits in the form of an integral inequality (Eq.~\eqref{eqn: CH}) roughly asserting that the average height of the push-forward of a curve tends to decrease.
    This idea was introduced in~\cite{EskinMargulisMozes} and has been used in numerous other contexts since.
    Our restriction on the class of curves is to insure that such an inequality holds uniformly and - more importantly - in a form that we can iterate.
    
    The construction of these functions along with establishing their main properties is carried out in Sections \S~\ref{section: height function rank 1}, \S~\ref{section: height fun higher rank}-\ref{section: sl2 products} and \S~\ref{section: height linear forms}. The proofs of Theorems~\ref{thrm: rank 1 DOA},~\ref{thrm: products of sl2}, and~\ref{thrm: sl2 actions} are given in Sections~\ref{section: proof of rank 1 thrm},~\ref{section: proof of sl2 products}, and~\ref{section: proof of sl2 actions}.
    Corollary~\ref{cor: lower bound cor} is established in Section~\ref{section: conclusion}.


    \section{Applications to Diophantine Approximation}
\label{section: applications}
	
    In this section, we state number theoretic consequences of our main results, particularly to Diophantine approximation problems.

\subsection{Diophantine Approximation on Spheres} \label{section: diophantine Sn}

     Intrinsic Diophantine approximation on $\mathbb{S}^n$ refers to approximating vectors in $\mathbb{S}^n$ using elements of the set $\mc{Q} = \Q^{n+1}\cap \mathbb{S}^n$, as opposed to approximation by elements of all of $\Q^{n+1}$.
     Given a function $\phi: \N \r (0,\infty)$, we say that $\mbf{x}\in \mathbb{S}^n$ is \textbf{intrinsically} $\phi$-approximable if there exist infinitely many $(\mathbf{p},q) \in \Z^{n+1}\times \N$ such that $\mbf{p}/q \in \mathbb{S}^n$ and
     \begin{equation} \label{eqn: phi approximable}
     	\norm{ \mbf{x} - \frac{\mathbf{p}}{q} } < \phi(q).
     \end{equation}
     Following~\cite{KleinbockMerrill}, we denote by $A(\phi,\mathbb{S}^n)$ the set of $\phi$-approximable points and for $\tau >0$, we let $    \phi_\tau(x) = x^{-\tau}$.
     An analogue of Dirichlet's classical theorem was obtained in~\cite[Theorem 1.1]{KleinbockMerrill}
     showing that $A(C_n \phi_1,\mathbb{S}^n)=\mathbb{S}^n$ for some constant $C_n>0$.    
      Moreover, it is shown that badly approximable points on $\mathbb{S}^n$ exist in this setting~\cite[Theorem 1.2]{KleinbockMerrill}.
   	  We say $\mbf{x} \in \mathbb{S}^n$ is \textbf{badly approximable} if there exists a constant $\epsilon(\mbf{x})>0$ such that 
      $\mbf{x} \notin A(\epsilon(\mbf{x}) \phi_1,\mathbb{S}^n)$.
      The analogue of Khinchin's theorem was established in~\cite[Theorem 1.3]{KleinbockMerrill}.
      
    We say that $\mbf{x} \in \mathbb{S}^n$ is \textbf{intrinsically singular on average} if for all $\epsilon>0$, the following holds.
    \begin{equation}\label{defn: singular on average on spheres}
    \lim_{N\r\infty} \frac{1}{N} \# \set{ 1\leqslant \ell\leqslant N:  \norm{ \mbf{x} - \frac{\mathbf{p}}{q} } < \epsilon 2^{-\ell}, 0<|q| \leqslant 2^\ell } = 1.
    \end{equation}
    
    In~\cite{KleinbockMerrill}, these Diophantine properties were connected to the dynamics of a diagonalizable flow $g_t$ on $\mrm{SO}(n+1,1)/\G$, where $\G$ is an arithmetic lattice.
    This is done by associating to each $\mbf{x} \in \mathbb{S}^n$, an element $u(Z_\mbf{x})$ in the expanding horospherical subgroup of $g_t$.
    Then, they show that $\mbf{x} \in \mathbb{S}^n$ is badly approximable if and only if the orbit $g_t u(Z_\mbf{x})\G$ is bounded in $G/\G$.
    In~\cite[Theorem 1.5]{KleinbockMerrill}, the property of being $\phi$-approximable was connected to excursions of the orbit $g_t u(Z_\mbf{x})\G$ into cusp neighborhoods parametrized by $\phi$.
    Using this correspondence with dynamics, one can show that $\mbf{x}$ is intrinsically singular on average if and only if the orbit $g_t u(Z_\mbf{x})\G$ is divergent on average in $G/\G$.    
    This correspondence when combined with Theorem~\ref{thrm: rank 1 DOA} imply the following corollary.

    \begin{corollary} \label{cor: dioph approx on Sn}
    Suppose $B\subset \R$ is a compact interval and $\vp: B \r \mathbb{S}^n$ is a $C^{1+\e}$-map for some $\e>0$ such that $\dot{\vp}$ does not vanish on $B$. Then, the following hold.
        \begin{enumerate}
        \item The Hausdorff dimension of the set of points $s\in B$ such that $\vp(s)$ is intrinsically singular on average  is at most $1/2$.
          \item The set of points $s\in B$ for which $\vp(s)$ is intrinsically badly approximable is winning for a Schmidt game on $B$. In particular, this set is thick in $B$.
          \item For every $\g >0$, the set of points $s\in B$ for which $\vp(s) \in  A(\phi_{1+\g},\mathbb{S}^n) $ has Lebesgue measure $0$. 
        \end{enumerate}
    \end{corollary}
    

    
    \subsection{Diophantine Approximation by Number Fields}
    \label{section: diophantine number fields}
    Our next application concerns a generalization of the classical notion of Diophantine approximation of a real number by rationals to approximation by elements in a number field.
    Suppose $K$ is a finite extension of $\Q$ of degree $d$ and let $\mc{O}_K$ denote its ring of integers.
    Denote by $\Sigma$ the set of Galois embeddings of $K$ into $\R$ and $\C$, where we choose one of the two complex conjugate embeddings.
    Let $r$ (resp. $s$) denote the number of real (resp. complex) embeddings in $\Sigma$ so that $d = r+2s$.
    Denote by $K_\Sigma = \R^r \times \C^s$ and let $\Delta: K \r K_\Sigma$ be the embedding defined by
    \[ \Delta(x) = (\s(x))_{\s\in \Sigma}. \]
    
    Let $G = \mrm{SL}(2,\R)^r \times \mrm{SL}(2,\C)^s$.
    The map $\Delta$ extends to an embedding of $\mrm{SL}(2,\mc{O}_K)$ into $G$ and we let $\G=\Delta(\mrm{SL}(2,\mc{O}_K))$.
    Then, $\G$ is a non-uniform irreducible lattice in $G$ and there exists a rational structure on $G$ so that $\G$ is an arithmetic lattice of $\Q$-rank $1$.
    Define the following elements of $G$.
    \begin{equation} \label{eqn: g_t and u(x)}
    g_t = \left( \begin{pmatrix} e^t & 0 \\ 0 & e^{-t} \end{pmatrix} \right)_{\s\in\Sigma},
    \qquad  
    u(\mathbf{x})= \left( \begin{pmatrix} 1& \mathbf{x}_\s \\ 0 & 1 \end{pmatrix} \right)_{\s\in\Sigma}.
    \end{equation}

    We say $\mbf{x} = (x_\s)_{\s\in\Sigma} \in K_\Sigma$ is $K$-\textbf{badly approximable} if there exists $\epsilon(\mbf{x}) >0$ so that for all $p,q\in \mc{O}_K$ with $q\neq 0$,
    \begin{equation*}
    	\max_{\s\in \Sigma} \set{ | \s(p) + x_\s \s(q)|  } \max_{\s\in \Sigma} \set{|\s(q)|} \geqslant \epsilon(\mbf{x}).
    \end{equation*}
    We say $\mbf{x}$ is $K$-\textbf{very well approximable} if for some $\g>0$, there exist infinitely many non-zero pairs $(p,q)\in \mc{O}_K^2$ such that
    \begin{equation*}
    	\max_{\s\in \Sigma} \set{ | \s(p) + x_\s \s(q)|  } \max_{\s\in \Sigma} \set{|\s(q)|^{1+\g}} <1.
    \end{equation*}
    Finally, say $\mbf{x}$ is $K$-\textbf{singular on average} if for all $\epsilon >0$, the following holds.
    \begin{equation}\label{defn: singular on average num fields}
    \lim_{N\r\infty} \frac{1}{N} \# \set{ 1\leqslant \ell\leqslant N:  \max_{\s\in \Sigma} \set{ | \s(p) + x_\s \s(q)|  }   < \epsilon 2^{-\ell}, 0<\max_{\s\in \Sigma} \set{|\s(q)|} \leqslant 2^\ell } = 1.
    \end{equation}
    
    Analogues of Dirichlet's theorem as well as the existence of badly approximable vectors have been established in this setting.
    Moreover, it is shown in~\cite{Einsiedler-CurvesNumfields} that $\mbf{x}$ is $K$-badly approximable if and only if the orbit $g_t u(\mbf{x})\G$ is bounded in $G/\G$.
    The same correspondence implies that $\mbf{x}$ is $K$-\textbf{singular on average} if and only if the orbit $g_t u(\mbf{x})\G$ is divergent on average in $G/\G$.
    Finally, we note that the group $g_t$ above is split in this case and, in particular, Theorem~\ref{thrm: products of sl2} applies and gives the following corollary.

    \begin{corollary}\label{cor: dioph approx in num fields}
    Suppose $B\subset \R$ is a compact interval and $\vp=(\vp_\s)_{\s\in\Sigma} : B \r \R^r \times \C^s$ is a $C^{1+\e}$-map for some $\e>0$ such that for each $\s$, either $\dot{\vp}_\s \equiv 0$ or $\dot{\vp}_\s$ has finitely many zeros.
    Assume further that
    \begin{equation} \label{cor max condition}
    	 \# \set{\s\in\Sigma: \dot{\vp}_\s \not\equiv  0 , \s \textrm{ is real}  } 
   			+ 2 \cdot \# \set{\s\in\Sigma: \dot{\vp}_\s \not\equiv  0 , \s \textrm{ is complex}} > \frac{r+2s}{2}.
    \end{equation}
    Then, the following hold.
    \begin{enumerate}
      \item The Hausdorff dimension of the set of points $s\in B$ for which $\vp(s)$ is $K$-singular on average is at most
      \begin{equation*}
      	\frac{1}{2}  + \frac{1}{2} \frac{ \# \set{1\leqslant i \leqslant r: \dot{\vp}_i \equiv  0 } 
   			+ 2 \cdot \# \set{r< i \leqslant r+s: \dot{\vp}_i \equiv  0 }}
            {\# \set{1\leqslant i \leqslant r: \dot{\vp}_i \not\equiv  0 } 
   			+ 2 \cdot \# \set{r< i \leqslant r+s: \dot{\vp}_i \not\equiv  0 }}.
      \end{equation*}
        \item The set of points $s\in B$ for which $\vp(s)$ is $K$-badly approximable is winning for a Schmidt game on $B$. In particular, this set is thick in $B$.
        \item The set of points $s\in B$ for which $\vp(s)$ is $K$-very well approximable has Lebesgue measure $0$.
    \end{enumerate}
    \end{corollary}
    
    As stated in the introduction, the winning property of badly approximable vectors in Corollary~\ref{cor: dioph approx in num fields} was obtained before in~\cite{Einsiedler-CurvesNumfields} by different methods.
    
    \subsection{Square Systems of Linear Forms}

    Our next corollary is an application of Theorem~\ref{thrm: sl2 actions} to the study of the Diophantine properties of square matrices regarded as systems of linear forms.
    In particular, we are interested in the dimension of badly approximable and singular matrices and the measure of very well approximable matrices belonging to a straight line in $M_{n,n}(\R)$.
    We first recall the precise definitions of these notions.
    We say a matrix $Y \in M_{n,n}(\R)$ is \textbf{badly approximable} if there exists $\epsilon(Y)>0$ for all $(\mathbf{p}, \mathbf{q}) \in \Z^m\times \Z^{n}$ with $\mbf{q} \neq 0$:
    \[  \norm{ \mathbf{p} + Y \cdot \mathbf{q} } \norm{\mathbf{q}} \geqslant \epsilon(Y), \]
    where for $v =(v_1,\dots,v_n)\in \R^n$, $\norm{v} = \max |v_i|$.
    We say $Y$ is \textbf{singular} if for every $\e>0$, there exists $N_0 \in \N$ so that for all $N \geqslant N_0$, the following inequalities hold for some $(\mathbf{p},\mathbf{q}) \in \Z^n\times \Z^{n}$.
    \[
    	\begin{cases}
    		\norm{\mathbf{p} + Y \mathbf{q}  } \leqslant \e/N, \\
            0<\norm{\mathbf{q}} \leqslant  N.
    	\end{cases}
    \]
    Finally, $Y$ is \textbf{very well approximable} (VWA) if there exists $\e >0$ and infinitely many $\mbf{q} \in \Z^n$ such that
    \begin{equation*}
    \norm{  Y\mbf{q} - \mbf{p} } < \norm{\mbf{q}}^{-1 - \e} \text{ for some }
    \mbf{p} \in \Z^n.
    \end{equation*}

    These Diophantine properties can be studied through dynamics on the space of unimodular lattices in $\R^{2n}$ as follows. 
    Let $G= \mrm{SL}(2n,\R)$, $\G= \mrm{SL}(2n,\Z)$, and $X = G/\G$.
    For $t\in  \R$ and $Y \in M_{n,n}(\R)$, define the following elements of $G$.
    \begin{equation} \label{linear forms g_t}
    	g_t = \begin{pmatrix}
    		e^t\mathrm{I}_n & \mathbf{0} \\  \mathbf{0} & e^{-t} \mathrm{I}_n
    	\end{pmatrix}, \qquad
    	u_Y = \begin{pmatrix}
    		\mathrm{I}_n & Y \\ \mathbf{0} & \mathrm{I}_n
    	\end{pmatrix},        
    \end{equation}
    where $\mathrm{I}_n$ denotes the identity matrix.
    As discussed in the introduction, Dani showed that $Y$ is badly approximable if and only if the forward orbit $g_t u_Y \G$ is bounded in $X$.
    Similarly, $Y$ is singular if and only if the forward orbit $g_t u_Y \G$ is divergent.
    Finally, by~\cite[Proposition 3.1(a)]{KleinbockMargulisWang}, $Y$ is VWA if and only if 
    \[  \limsup_{t\r\infty} \frac{d_X(g_tu_Y\G, x_0)}{t} >0, \]
    where $d_X(\cdot,\cdot)$ is the Riemannian metric on $X$ induced by the right invariant metric on $G$ and $x_0$ is any base point in $X$.
    
    Using this correspondence with dynamics, Theorem~\ref{thrm: sl2 actions} has the following corollary.
    
    \begin{corollary}\label{cor: dioph approx of linear forms}
    Suppose $\vp: B \r M_{n,n}(\R)$ is defined by $\vp(s) = sY +Z$ for some $Y\in \mrm{GL}(n,\R)$ and $Z\in M_{n,n}(\R)$. Then, the following hold.
    \begin{enumerate}
      \item The Hausdorff dimension of the set of points $s\in B$ for which $\vp(s)$ is singular is at most $1/2$. 
      \item The set of points $s\in B$ for which $\vp(s)$ is badly approximable is winning for a Schmidt game on the real line. In particular, this set is thick.
         \item The Lebesgue measure of the set of points $s\in B$ for which $\vp(s)$ is very well approximable is $0$.
    \end{enumerate}
    \end{corollary}
    
    In this setting, the homomorphism $\rho: \mrm{SL}(2,\R)\r G$ used to obtain Corollary~\ref{cor: dioph approx of linear forms} from Theorem~\ref{thrm: sl2 actions} is defined as follows.
    \[ \rho \left( \begin{pmatrix}e^t &0\\ 0 & e^{-t}    \end{pmatrix}  \right)=  g_t, \quad  \rho  \left( \begin{pmatrix}1 &s\\ 0 & 1    \end{pmatrix}  \right)=u_{sY}, \quad \rho  \left( \begin{pmatrix}1 &0 \\ s & 1    \end{pmatrix}  \right)= \begin{pmatrix}
    		\mathrm{I}_n & \mathbf{0} \\ sY^{-1} & \mathrm{I}_n\end{pmatrix}.          \]
    Finally, one applies Theorem~\ref{thrm: sl2 actions} to the base point $x_0=u_Z \G$.

	\section{The Contraction Hypothesis and Divergent Trajectories}
\label{section: abstract setup}
	In this section, we prove an abstract recurrence result for diagonalizable trajectories starting from admissible curves in actions of Lie groups on metric spaces.
    Theorem~\ref{thrm: Hdim and non-divergence} is the main result of this section establishing, in particular, a bound on the dimension of divergent orbits.
    In later sections, we verify the hypotheses of this theorem in the settings of the results stated in the introduction.

	\subsection{The Contraction Hypothesis for Lie Group Actions}

	Suppose $G$ is a connected real Lie group with Lie algebra $\mf{g}$.
    Consider a non-trivial $1$-parameter subgroup $A = \set{g_t:t\in \R}$ which is $\mrm{Ad}$-diagonalizable over $\R$.
    Then, $\mf{g}$ decomposes under the adjoint action of $g_t$ into eigenspaces
    \[ \mf{g} = \bigoplus_{\a\in A^\ast} \mf{g}_\a.  \]
    where $A^\ast$ denotes the group of additive homomorphisms $\a:A \r \R$.
    In particular, for every $t\in \R$, $\a\in A^\ast$ and $Y\in \mf{g}_\a$, we have
    \begin{equation} \label{defn: eigenvalue}
    \mrm{Ad}(g_t)(Y) = e^{\a(t)}Y.
    \end{equation}

    We are interested in studying $g_t$-admissible curves $\vp$ as defined in the introduction. Note that the vanishing set $Z$ in the statements of the main theorems is a closed set. Since all the results stated in the introduction concerning measure and Hausdorff dimension are local, we assume without loss of generality that the curves we study are defined on a compact interval where $Z = \emptyset$. We make a further simplification requiring that $\vp$ commutes with itself. The case $[\vp,\dot{\vp}] \equiv 0$ of Theorem~\ref{thrm: rank 1 DOA} requires very minor modifications to our proofs.
    The following definition makes these reductions more precise for purposes of reference in the later parts of the article.
    \begin{definition} \label{defn: admissible curves}
    A map $\vp: [-1,1] \r \mf{g}$ is $\mathbf{g_t}$\textbf{-admissible} if the following holds:
    \begin{enumerate}
    	\item $\vp$ is $C^{1+\g}$ for some $\g >0$, i.e. it is continuously differentiable and the H\"{o}lder exponent of its derivative $\dot{\vp}$ is $\g$.
        \item The image of $\vp$ is contained in a a subspace $V$ of a single eigenspace $\mf{g}_\a$ for some $\a$ such that $\a(t) >0 $ for $t>0$ and $[V,V] =0$.
        \item The derivative of $\vp$ does not vanish on $[-1,1]$.
    \end{enumerate}
    \end{definition}
    Note that we only require the span of the image of $\vp$ to be an abelian subalgebra. In particular, the ambient eigenspace $\mf{g}_\a$ need not be an abelian subspace.

    The following is the key recurrence property for the action which underlies the results stated in the introduction.
	\begin{definition} [The Contraction Hypothesis]
      \label{defn: height functions}
      Suppose $X$ is a metric space equipped with a $G$-action.
      A $g_t$-admissible curve $\vp:[-1,1] \r \mf{g}_\a \subset \mf{g}$ is said to satisfy the $\b$-\textbf{contraction hypothesis} on $X$ if there exists a proper function $f:X\r (0,\infty]$ satisfying the following properties:
      \begin{enumerate}
      \item The set $Z =  \set{f=\infty}$ is $G$-invariant and $f$ is bounded on compact subsets of $X\backslash Z$.
        \item $f $ is uniformly log Lipschitz with respect to the $G$ action. That is for every bounded neighborhood $\mc{O}$ of identity in $G$, there exists a constant $C_\mc{O}\geq 1$ such that for $g\in \mc{O}$ and all $x\in X$,
      \begin{equation}\label{defn: log lipschitz}
          C_\mc{O}^{-1} f(x) \leqslant f(gx) \leqslant C_\mc{O} f(x).
      \end{equation}
		\item There exists $\tilde{c} \geq 1$ such that the following holds:
        for all $t >0$, there exists $\tilde{b} = \tilde{b}(t) >0$ such that for all $x \in X$ and all $s\in [-1,1]$, 
        \begin{equation} \label{eqn: CH}
        	\frac{1}{2} \int_{-1}^1 f(g_t u(r\dot{\vp}(s)) x) \;dr 
            		\leqslant \tilde{c} e^{-\b \a(t)} f(x) + \tilde{b},
        \end{equation}
        where $u(Y) = \exp(Y)$ for $Y\in \mf{g}_\a$.
         \item For all $M\geq 1$, the sets $\overline{\set{x\in X: f(x) \leq M}}$, denoted by $X_{\leq M}$, are compact.
      \end{enumerate}
   The function $f$ will be referred to as a \textbf{height function}.
    \end{definition}
    
    The notion of height functions was introduced to homogeneous dynamics in~\cite{EskinMargulisMozes}.
    It was used in~\cite{KKLM-SingSystems} to obtain sharp upper bounds on the dimension of singular systems of linear forms.
    We note that allowing height functions to assume the value $\infty$ has proven useful in several important applications~\cite{BQ-RandomWalkRecurrence,EMM}.

    The following is the main result of this section.
    \begin{theorem} \label{thrm: Hdim and non-divergence}
    	Let $G$ be a real Lie group and $X$ be a metric space equipped with a $G$-action. Suppose $g_t$ is an $\mrm{Ad}$-diagonalizable one-parameter subgroup of $G$ and $\vp$ is a $g_t$-admissible curve satisfying the $\b$-contraction hypothesis on $X$.
        Then, for all $x\in X\backslash \set{f=\infty}$, the following hold.
        \begin{enumerate}
        \item The Hausdorff dimension of the set of $s\in [-1,1]$ for which the trajectory $ g_tu(\vp(s))x$ is divergent on average is at most $1-\b$.
        
        \item For Lebesgue almost every $s$, any weak-$\ast$ limit of the measures $\frac{1}{T}\int_0^T \d_{g_t u(\vp(s))x}\;dt$ is a probability measure on $X$.
        \end{enumerate}
    \end{theorem}
    Throughout this section, we fix a metric space $X$ equipped with a proper continuous $G$-action and we fix a $g_t$-admissible curve $\vp$ satisfying the $\b$-contraction hypothesis on $X$.
    
   We remark that if an orbit $\set{ g_t x: t\geq 0}$ is divergent on average for some $x$ with $f(x)<\infty$, then for all $M>0$,
   \[ \frac{1}{T} \int_0^T \chi_{M}(g_t x)\;dt \r 0, \]
    where $\chi_M$ is the indicator function of $X_{\leq M} =\set{y\in X: f(y) \leq M}$.

    The main applications of our results are to $G$ actions on homogeneous spaces of $G$ of the form $X=G/\G$ where $\G$ is a lattice in $G$.
    The following lemma shows that the $\b$-contraction hypothesis is a property of the commensurability class of $\G$ and will allow us to reduce the task of establishing the $\b$-contraction hypothesis to irreducible lattices in the case $G$ is semisimple.
    
    Recall that two topological spaces $X_1$ and $X_2$ are \emph{commensurable} if they have homeomorphic finite-sheeted covering spaces.
    
    \begin{lemma} \label{lemma: stability of CH}
    Suppose $\vp$ is a $g_t$-admissible curve satisfying the $\b$-contraction hypothesis for the $G$-action on a metric space $X$ and for some $\b>0$. Denote by $\mf{g}_\a$ the $\mrm{Ad}(g_t)$-eigenspace of $\mf{g}$ containing the image of $\vp$. Then, the following hold.
    \begin{enumerate}
    \item Suppose $X'$ is a metric space which is commensurable to $X$ with a common finite cover $Y$.
    Assume that $Y$ and $X'$ are equipped with an action of $G$ which is equivariant with respect to the covering maps $Y\r X$ and $Y\r X'$.
    Then, $\vp$ satisfies the $\b$-contraction hypothesis for the $G$-action on $X'$ and for the same $\b$.
    \item Suppose $G'$ is a Lie group with Lie algebra $\mf{g}'$ and $g_t'$ is a $1$-parameter $\R$-diagonalizable subgroup of $G'$.
    Suppose $\vp'$ is a $g_t'$-admissible curve satisfying the $\b'$-contraction hypothesis for the $G'$-action on a metric space $X'$.
    Let $\mf{g}_{\a'}'$ be the $\mrm{Ad}(g_t')$-eigenspace of $\mf{g}'$ containing the image of $\vp'$.
    Assume further that $\a(t) = \a'(t)$ for all $t\in \R$. 
    Then, $\vp \oplus \vp':[-1,1]\r \mf{g}\oplus \mf{g}'$ is $(g_t,g_t')$-admissible and satisfies the $(\min(\b,\b'))$-contraction hypothesis for the $G\times G'$-action on $X\times X'$.
    \end{enumerate}
    \end{lemma}
    
    \begin{proof}
    (1) 
    Denote by $p:Y\r X$ and $p':Y\r X'$ the covering maps.
    Let $f$ the height function on $X$.
    Define a height function $f'$ on $X'$ by
    \begin{equation*}
        f'(x') = \sum_{y: p'(y)=x'} f(p(y)).
    \end{equation*}
    Since the above sum runs over finitely many points, whose cardinality is equal to the sheetedness of the cover $Y\r X'$, then one verifies that $f'$ satisfies all the properties in Definition~\ref{defn: height functions}.
    
    (2) Denote by $f$ and $f'$ the height functions on $X$ and $X'$ respectively.
    Then, one defines a function $f+f'$ on $X\times X'$ by $(f+f')(x,x')= f(x)+f'(x')$.
    Then, $f+f'$ provides the desired height function on $X\times X'$.
    \end{proof}

	\subsection{Approximation by horocycles and the Markov property}	
    The following elementary lemma allows us to obtain an integral estimate over curves via integral estimates over tangents while simultaneously providing us with a mechanism for iterating such integral estimates.
    This iteration mechanism will play the same role as the Markov property in the context of random walks.
    
    Recall that $\g >0$ denotes the H\"{o}lder exponent of the derivative of $\vp$.
  \begin{lemma} \label{lemma: approximate curve with tangent}
  	There exists a constant $C_1 >1$, such that for all $x\in X$,
    natural numbers $n $ with $n \geq 1/\g$, $t>0$ and all subintervals $J \subset [-1,1]$ of radius at least $e^{-\a(nt)} $, one has    
    \begin{equation} 
    	\int_J f (g_{(n+1)t} u(\vp(s))x)\;ds \leqslant 
        	C_1 \int_J \int_{-1}^1 f (g_t u(r\dot{\vp}(s)) g_{nt} u(\vp(s))x) \;dr\;ds.
    \end{equation}
  \end{lemma}
  
  \begin{proof}
  	First, we note that for all $r \in [-1,1]$, we have
    \begin{equation} \label{eqn: cover J by translates}
    	J \subseteq J \pm r e^{-\a(nt)} := (J + re^{-\a(nt)}) \cup (J-re^{-\a(nt)}).
    \end{equation}
    Using positivity of $f$,~\eqref{eqn: cover J by translates} and a change of variable, we get
    \begin{align*}
    	&\int_J f (g_{(n+1)t} u(\vp(s))x)\;ds \\
        	&=\int_{0}^1 \int_J f (g_{(n+1)t} u(\vp(s))x)\;ds \;dr 
            \leqslant \int_{0}^1 \int_{J \pm r e^{-\a(nt)}} f (g_{(n+1)t} u(\vp(s))x)\;ds \;dr \\
            &= \int_{-1}^1 \int_{J + r e^{-\a(nt)}} f (g_{(n+1)t} u(\vp(s))x)\;ds \;dr 
            = \int_{-1}^1 \int_{J } f (g_{(n+1)t} u(\vp(s+ r e^{-\a(nt)}))x)\;ds \;dr. 
    \end{align*}
    Then, Fubini's theorem and the fact that $\vp$ is $C^{1+\g}$ imply the following.
    \begin{align*}
    	\int_J f (g_{(n+1)t} u(\vp(s))x)\;ds &\leqslant
            \int_{J } \int_{-1}^1 f ( g_{(n+1)t} u(\vp(s) + re^{-\a(nt)}\dot{\vp}(s) 
            +O(e^{-(1+\g)\a(nt)} ) )x) \;dr\;ds.  
	\end{align*}
    Moreover, by definition of $g_t$ and $u(Y)$, we have
    \begin{equation*}
    	g_t u(Y) g_{-t} = u(e^{\a(t)}Y).
    \end{equation*}
    Thus, by our assumption that $n \geqslant 1/\g$, we get
    \begin{align*}
    	\int_J f (g_{(n+1)t} u(\vp(s))x)\;ds &\leqslant
            \int_{J } \int_{-1}^1 f (u(O(1)) g_{(n+1)t} u(\vp(s) + re^{-\a(nt)}\dot{\vp}(s) )x) \;dr\;ds \\
            &= \int_{J } \int_{-1}^1 f (u(O(1 )) g_{t} u(r\dot{\vp}(s))g_{nt} u(\vp(s))x) \;dr\;ds.
	\end{align*}
    Note that $u(O(1 ))$ belongs to a bounded neighborhood of identity independently of $t$ and $n$.
    Hence, by the log Lipschitz property of $f$, there exists a constant $C_1 > 1$ such that 
    for all $y\in X$,
    \[ f(u(O(1)) y) \leqslant C_1 f(y). \]
    This concludes the proof.
  \end{proof}

\subsection{Integral estimates and long excursions}
The goal of this section is to prove an upper bound on the measure of the set of trajectories with long excursions outside of fixed compact sets. We show that such a measure decays exponentially in the length of the excursion.
We remark that our proof of this fact is different from the proof of a similar step in~\cite[Proposition 5.1]{KKLM-SingSystems}. Our method allows us to handle curves which are in general not subgroups that are normalized by $g_t$.
The proof of~\cite{KKLM-SingSystems}, however, uses this point crucially.

	For $x\in X$, $M, t>0$ and natural numbers $m,n \in \N$, we define the following sets
	\begin{equation*} \label{defn: B_x(M,t; m+n)}
		B_x(M,t, m;n) = \set{ s\in [-1,1]: f(g_{mt}u(\vp(s)x) < M,
        	f(g_{(m+l)t}u(\vp(s))x) \geqslant M, \text{ for } 1 \leq l \leq n}.
	\end{equation*}

    For every $N\in \N$, let $\mc{P}_{N}$ denote the partition of the interval $[-1,1]$ into $N$ intervals of equal length.

    \begin{proposition} \label{propn: measure bound}
    There exists a constant $c_0 \geq 1$ such that for every $t>0$ with $e^{\a(t)} \in \N$, there exists $M_0 = M_0(t) >0$, so that for all $M>M_0$ the following holds.
     For all natural numbers $m \geq 1/\g$ and $n\geq 1$ and all $x\in X\backslash\set{f=\infty}$, one has that
     \begin{equation*}
       	\lvert B_x(M,t, m;n) \cap J_0 \rvert \leqslant c_0^n e^{-\b \a(nt)} |J_0|,
     \end{equation*}
     for every interval $J_0\in \mc{P}_{e^{\a(mt)} }$, where $\lvert\cdot\rvert$ denotes the Lebesgue measure on $[-1,1]$.
    \end{proposition}
    
    \begin{proof}
    	Let $t>0$ be fixed. Let $\tilde{c}$ and $\tilde{b} = \tilde{b}(t) >0$ be as in $(3)$ of Definition~\ref{defn: height functions}.
        Let $T = \tilde{b} e^{\b \a(t)}/\tilde{c}$.
        Then, for all $x\in X$ with $f(x) > T$, using~\eqref{eqn: CH}, we get
        \begin{equation*}
        	\frac{1}{2} \int_{-1}^1 f(g_t u(r\dot{\vp}(s)) x) \;dr 
            		\leqslant 2\tilde{c} e^{-\b \a(t)} f(x).
        \end{equation*}
    	Using $(2)$ of Definition~\ref{defn: height functions}, we can find $\tilde{C}_1 \geq 1$ such that for all $x\in X$ and all $s\in [-1,1]$, we have
        \begin{equation} \label{eqn: defn of tilde C_1}
        	\tilde{C}_1^{-1} f(x) \leqslant  f(u(\dot{\vp}(s)) x) \leqslant \tilde{C}_1 f(x).
        \end{equation}
        We define $c_0$ and $M_0$ as follows
        \begin{equation*}
        	c_0 = 4C_1\tilde{C}_1\tilde{c}, \qquad M_0 = \tilde{C}_1 T,
        \end{equation*}
        where $C_1$ denotes the constant in Lemma~\ref{lemma: approximate curve with tangent}.
        Suppose $M>M_0$.
    	 To simplify notation, for each $k \in \N$, we let 
        \begin{equation*}
        	B(M,k) := B_x(M,t, m;k).
        \end{equation*}   
        For purposes of induction, we also define $B(M,0)$ as follows
        \begin{equation*}
        	B(M,0) := \set{s\in [-1,1] : f(g_{mt}u(\vp(s))x) > T }.
        \end{equation*}
        Let us also write $\mc{P}_k$ to denote $\mc{P}_{e^{\a(kt)}}$ for simplicity.
        
        Suppose $J \in \mc{P}_{m+n-1}$ is such that $ J \cap B(M,n-1) \neq \emptyset$
        and let $s_0 \in J \cap B(M,n-1)$.
        Then, we have $f(g_{(m+n-1)t}u(\vp(s_0))x) > M$.
        Now, consider any $s\in J$. Writing $\vp(s) = \vp(s_0) + O_{\dot{\vp}}(|J|) $, we see that
        \[ f(g_{(m+n-1)t}u(\vp(s))x) > T. \]
        Indeed, this follows from~\eqref{eqn: defn of tilde C_1} and the fact that $M>\tilde{C}_1 T$.
        Therefore, by Lemma~\ref{lemma: approximate curve with tangent} and the choice of $T$, it follows that
        \begin{align} \label{eqn: first step removed}
        	\int_J f(g_{(m+n)t} u(\vp(s)) x) \;ds & \leqslant
            C_1 \int_J \int_{-1}^1 f( g_t u(r \dot{\vp}(s) )   g_{(m+n-1)t} u(\vp(s)) x) \;dr ds \nonumber \\
            &\leqslant 4C_1 \tilde{c} e^{-\b \a(t)} \int_J f( g_{(m+n-1)t} u(\vp(s)) x) \; ds.
        \end{align}
        
        Now, consider an interval $J_0 \in \mc{P}_m$ satisfying $J_0 \cap B(M,n) \neq \emptyset$.
        Then, since $B(M,n)$ is contained in  $B(M,n-1)$, we have that $J_0 \cap B(M,n-1) \neq \emptyset$.        
        Next, note that the following inclusion holds.
        \[ B(M,n-1) \cap J_0 \subseteq  \bigcup_{\substack{J \in \mc{P}_{m+n-1}\\
        		 J \cap B(M,n-1) \cap J_0 \neq \emptyset}} J.  \]
        In particular, by~\eqref{eqn: first step removed}, we get
        \begin{align} \label{eqn: estimate using cover}
        \int_{B(M,n-1)\cap J_0} f(g_{(m+n)t} u(\vp(s)) x) \;ds &\leqslant
        	\sum_{ \substack{J \in \mc{P}_{m+n-1} \\  J \cap B(M,n-1) \cap J_0 \neq \emptyset} } 
            \int_J f(g_{(m+n)t} u(\vp(s)) x) \;ds \nonumber \\
            &\leqslant 4C_1\tilde{c} e^{-\b \a(t)} 
            \sum_{ \substack{J \in \mc{P}_{m+n-1} \\  J \cap B(M,n-1) \cap J_0\neq \emptyset} } 
            \int_J f( g_{(m+n-1)t} u(\vp(s)) x) \; ds. 
        \end{align}
        Since $e^{\a(t)} \in \N$, for each $1\leq j\leq k$, the partition $\mc{P}_k$ is a refinement of $\mc{P}_{j}$.
        This implies the following inclusion.
        \begin{equation} \label{eqn: refining property}
        \bigcup_{\substack{J \in \mc{P}_{m+n-1} \\ J \cap B(M,n-1) \cap J_0\neq \emptyset}} J
         \subseteq
         \bigcup_{\substack{J \in \mc{P}_{m+n-2} \\  J \cap B(M,n-1) \cap J_0\neq \emptyset} } J.        
        \end{equation}
        Hence, the following inequality follows from~\eqref{eqn: estimate using cover},~\eqref{eqn: refining property}, and the fact that $f$ is non-negative:
        \begin{equation}\label{eqn: refinement trick}
        \int_{B(M,n-1)\cap J_0} f(g_{(m+n)t} u(\vp(s)) x) \;ds \leqslant 
             4C_1\tilde{c} e^{-\b \a(t)} 
            \sum_{ \substack{J \in \mc{P}_{m+n-2} \\  J \cap B(M,n-1) \cap J_0\neq \emptyset} } 
            \int_J f( g_{(m+n-1)t} u(\vp(s)) x) \; ds.
        \end{equation}

        Iterating~\eqref{eqn: refinement trick}, by induction, we obtain the following exponential decay integral estimate.
        \begin{align} \label{eqn: exp decay estimate}
        	\int_{B(M,n-1)\cap J_0} f(g_{(m+n)t} u(\vp(s)) x) \;ds &\leqslant
            (4C_1\tilde{c})^n e^{-\b \a(nt)} \sum_{ \substack{J \in \mc{P}_{m} \\  J \cap B(M,n-1) \cap J_0 \neq \emptyset} } 
            \int_J f( g_{mt} u(\vp(s)) x) \; ds \nonumber \\
            &= (4C_1\tilde{c})^n e^{-\b \a(nt)}\int_{J_0} f( g_{mt} u(\vp(s)) x) \; ds,
         \end{align}
         where on the second line, we used the following consequence of $\mc{P}_m$ being a partition.
         \begin{equation*}
         J \in \mc{P}_m, J\cap J_0 \neq \emptyset \Longrightarrow J = J_0.
         \end{equation*}
        
        Suppose  $s_0 \in J_0 \cap B(M,n-1)$.
        Then, by definition of the set $B(M,n-1)$, we have $f(g_{mt}u(\vp(s_0))x)$ is at most $M$.
        Thus, arguing as before, using~\eqref{eqn: defn of tilde C_1}, we obtain the following inequality for all $s\in J_0$,
        \begin{equation}\label{eq: first step in cpt}
            f(g_{mt}u(\vp(s))x) \leqslant \tilde{C}_1 M.
        \end{equation}  
        Combining this observation with~\eqref{eqn: exp decay estimate}, it follows that
        \begin{align}
        \int_{B(M,n-1)\cap J_0} f(g_{(m+n)t} u(\vp(s)) x) \;ds
            &\leqslant (4C_1\tilde{c})^n e^{-\b \a(nt)}  \tilde{C}_1 M |J_0|.
        \end{align}
        Hence, by Chebyshev's inequality, we obtain
        \begin{equation*}
        	\lvert B(M,n) \cap J_0 \rvert \leqslant c_0^n e^{-\b \a(nt)} |J_0|.
        \end{equation*}
        This completes the proof.
    \end{proof}
    
    The following corollary allows us to convert measure estimates into an estimate on covers.
    \begin{corollary} \label{cor: covering lemm}
    There exists a constant $C_2 \geqslant 1$, depending only on the height function $f$ and the curve $\vp$, such that the following holds.
    Suppose $M_0$ and $c_0$ are as in Proposition~\ref{propn: measure bound}.
    Then, for all $M>C_2 M_0$, $t>0$, $m,n\in \N$ with $m \geq 1/\g$ and $x\in X\backslash\set{f=\infty}$, the number of elements of the partition $\mc{P}_{e^{\a((m+n)t)}}$ needed to cover the set $B_x(M,t, m;n)\cap J_0$, for any $J_0 \in \mc{P}_{e^{\a(mt)}}$, is at most $   c_1^n e^{(1-\b)\a(nt)}$, where $c_1 = C_2 c_0$.
    \end{corollary}
    
    \begin{proof}
    	Using $(2)$ of Definition~\ref{defn: height functions}, one can find a constant $C_2 \geqslant 1$ so that the following holds.
    	Let $J \in \mc{P}_{e^{\a((m+n)t)}}$ be such that $J \cap B_x(M,t,m;n) \cap J_0 \neq \emptyset$. Then, for all $s\in J$ and all $1\leq l \leq n$,
        \begin{equation*}
            f(g_{mt}u(\vp(s))x) < C_2 M, \qquad 
            f(g_{(m+l)t}u(\vp(s))x) \geqslant C_2^{-1}M.
        \end{equation*}
        In particular, $J$ is contained in the set:
        \begin{equation*}
            B_x^{C_2}(M,t, m;n) = \set{ s: f(g_{mt}u(\vp(s)x) < C_2M,
        	f(g_{(m+l)t}u(\vp(s))x) \geqslant C_2^{-1}M, \text{ for } 1 \leq l \leq n}.
        \end{equation*}
        Moreover, since $\mc{P}_{e^{\a((m+n)t)}}$ is a refinement of $\mc{P}_{e^{\a(mt)}}$, it follows that
        \begin{equation}\label{eq: 0-1 for cor}
            J \subseteq B_x^{C_2}(M,t,m;n) \cap J_0.
        \end{equation} 
        
        The measure of the set $B_x^{C_2}(M,t, m;n)$ can be estimated as in the proof of Proposition~\ref{propn: measure bound}, where in the last step of the proof, we use the estimate $f(g_{mt}u(\vp(s))x) < C_2 M$ in place of that in~\eqref{eq: first step in cpt}.
        We, thus, obtain that
        \begin{equation}\label{eq: measure estimate for cor}
            |B_x^{C_2}(M,t,m;n) \cap J_0| \leqslant (C_2c_0)^n e^{-\b \a(nt)}|J_0|,
        \end{equation}
        where $c_0$ is the constant provided by Proposition~\ref{propn: measure bound}.
        The corollary thus follows upon combining~\eqref{eq: 0-1 for cor} and~\eqref{eq: measure estimate for cor}.
    \end{proof}

    \subsection{Integral estimates and coverings}

	For $x\in X$, $Q\subseteq X$, $ t,\d >0$ and $N \in \N$, we define the following sets
	\begin{equation} \label{defn: Z_x(M, N,t, delta)}
	Z_x(Q, N,t, \d) = \set{ s\in [-1,1]: \#\set{ 1\leq l \leq N: g_{lt}u(\vp(s))x \notin Q} > \d N}.
	\end{equation}
	To simplify notation, we denote the sets $Z_x(X_{\leq M}, N,t, \d)$ by $Z_x( M, N,t, \d)$ for all $M >0$.
    The following is the main covering result that will imply Theorem~\ref{thrm: Hdim and non-divergence}.
	\begin{proposition} \label{propn: main covering propn}
	There exists a constant $C_3 \geqslant 1$ such that the following holds.
    For all $t >0$ with $e^{\a(t)} \in \N$ and $x\in X\backslash \set{f=\infty}$, there exists $M_1 = M_1(t,x)>0$ so that for all $M>M_1$, $\d >0$ and $N\in \N$, the set $Z_x(M, N,t, \d)$ can be covered by at most 
    $ C_3^{N} e^{(1-\d\b) \a(Nt) } $ intervals of radius $e^{-\a(Nt)}$.
	\end{proposition}
    
    \begin{proof}
    Using $(2)$ of Definition~\ref{defn: height functions}, we have that
    \[ \tilde{M}_1 := \sup\limits_{\substack{s\in [-1,1], l\in [0, 1/\g] }} f(g_{lt}u(\vp(s))x) < \infty. \]
    Let $C_2 \geqslant 1$ be the constant in Corollary~\ref{cor: covering lemm} and let $M_0>0$ be as in Proposition~\ref{propn: measure bound}. Define $M_1$ as follows
    \begin{equation*}
    	M_1 := \max\set{ C_2 M_0, \tilde{M}_1 }.
    \end{equation*}
    
    Consider a set $\Phi \subseteq \set{1,\dots,N}$ containing at least $\d N$ elements.
    Define the following set of trajectories whose behavior is determined by $\Phi$:
    \begin{equation*}
    	Z(\Phi) = \set{s \in Z_x(M, N,t, \d): f(g_{lt}u(\vp(s))x) > M \textrm{ iff } l \in \Phi}.
    \end{equation*}
    Following~\cite{KKLM-SingSystems}, we decompose the set $\Phi$ into maximal connected intervals
    as follows:
    \begin{equation*}
    	\Phi = \bigsqcup_{i=1}^q B_i. 
    \end{equation*}
    Thus, we may write the set $\set{1,\dots,N}$ as disjoint union of maximal connected intervals in the following manner:
    \begin{equation*}
    	\set{1,\dots,N} = \bigsqcup_{i=1}^q B_i \sqcup \bigsqcup_{j=1}^p G_j.
    \end{equation*}
    Let $c_1 \geq 1$ be the constant in Corollary~\ref{cor: covering lemm}.
    We claim that $Z(\Phi)$ can be covered by at most $c_1^{N} e^{\a(Nt)-\b\a(|\Phi|t)}$ intervals of radius $e^{-\a(Nt)}$, where $|\Phi|$ denotes the cardinality of $\Phi$.
    Since the set $Z_x(M, N,t, \d)$ is a union of at most $2^N$ subsets of the form $Z(\Phi)$, the claim of the proposition follows by taking $C_3 = 2c_1$.

    Order the intervals $B_i$ and $G_j$ in the way they appear in the sequence $1 \leq\cdots\leq N$.
    For $1\leq r \leq p+q$, let $R_r$ denote the cardinality of the union of the first $r$ intervals in this sequence.
    In particular, $R_{p+q} = N$.
    We construct a cover by induction on $r$.
    In each step, we will show that if we write
    \[ \set{1,\dots,R_r} =  \bigsqcup_{i=1}^{r_1} B_i \sqcup \bigsqcup_{j=1}^{r_2} G_j, \]
     then the set $Z(\Phi)$ can be covered by 
     \[c_1^{R_r} e^{\a(t)\left(R_r-\b \sum_{i=1}^{r_1} |B_i|\right) }\]
     intervals of radius $e^{-\a(R_r t)}$ coming from the partition $\mc{P}_{e^{\a(R_r t)}}$.
    Note that by definition of $M_1$, we have $ 1 \in G_1 $.
    Hence, $R_1 = |G_1|$ and the first step of our induction is verified by taking all $e^{\a(R_1 t)}$ intervals of radius $e^{-\a(R_1 t)}$ which are needed to cover $[-1,1]$.
    
    Now, assume the claim holds for some $r < p+q$.
    Suppose that the $(r+1)$-st interval in the sequence of ordered intervals is of the form $G_j$ for some $1< j \leq p$.
    Let $J_0 \in \mc{P}_{e^{\a(R_r t)}}$ be an interval of radius $e^{-\a(R_rt)}$ in the cover constructed by the inductive hypothesis.
    Then, since $e^{\a(t)} \in \N$, $J_0$ contains $e^{(\a(R_{r+1}) -\a(R_r))t} = e^{\a(|G_j|t)}$ intervals of radius $e^{-\a(R_{r+1}t)}$.
    Thus, by taking all such intervals contained in each such $J_0$, we get a new cover of the desired cardinality in step $r+1$.
    
    Now, assume the $r+1$ interval in the sequence of ordered intervals is of the form $B_i$ for some $1\leq i \leq q$.
    We wish to apply Corollary~\ref{cor: covering lemm}.    
    By definition of $M_1$, we have that $ M > C_2 M_0$.
    Moreover, since $[1,1/\g]\cap\N$ is contained in $G_1$, we have that $R_r \geq 1/\g$.
    Thus, by Corollary~\ref{cor: covering lemm}, we can cover the set $B_x(M,t, R_r;|B_i|) \cap J_0$ by 
    \begin{align*}
       	c_1^{|B_i|} e^{(1-\b) \a(|B_i|t) }
     \end{align*}
     intervals of radius $e^{-\a((R_r+|B_i|)t)}$.
     Moreover, we have that
     \[ Z(\Phi) \subseteq B_x(M,t, R_r;|B_i|). \]
     Thus, the inductive step holds in this case as well by the inductive hypothesis on the number of the intervals $J_0 \in \mc{P}_{e^{\a(R_r t)}}$ needed to cover $Z(\Phi)$.
    \end{proof}

    \subsection{Proof of Theorem~\ref{thrm: Hdim and non-divergence}}
	
    Having established Proposition~\ref{propn: main covering propn}, the proof of Theorem~\ref{thrm: Hdim and non-divergence} follows the same lines as in~\cite{KKLM-SingSystems}.
    Let $x\in X$ and let $Z_x \subseteq [-1,1]$ denote the set of points $s$ for which the trajectory $g_t u(\vp(s))x$ diverges on average.
	To prove part $(1)$ of the theorem, we first note that for all compact sets $Q\subset X$ and for all $0<\d< 1$,
        \begin{align} \label{Z contained in liminf set}
        	Z_x \subseteq \liminf_{N\r\infty} Z_x(Q,N,t,\d)
        		= \bigcup_{N_0\geq 1} \bigcap_{N\geq N_0} Z_x(Q,N,t,\d),
        \end{align}
        where the sets $Z_x(Q,N,t,\d)$ were defined in~\eqref{defn: Z_x(M, N,t, delta)}.
        We wish to apply Proposition~\ref{propn: main covering propn} by taking $Q = X_{\leq M}$ for an appropriate $M$.
        
        Fix some $t>0$ and let $M_1 = M_1(t,x) >0$ be as in Proposition~\ref{propn: main covering propn}. 
        Suppose $M>M_1$ and $\d \in (0,1)$. Then, Proposition~\ref{propn: main covering propn} says that we can cover $Z_x(X_{\leq M} ,N,t,\d)$ by at most $ C_3^{N} e^{(1-\d\b) \a(Nt) } $ intervals of radius $e^{-\a(Nt)}$, where $C_3 \geq 1$ is a constant which is independent of $x, t$ and $N$.
        
         Then, we have
        \[
        	\overline{\dim}_{box}\left(\bigcap_{N\geq N_0} Z_x(Q,N,t,\d) \right)
            \leq \lim_{N\r\infty} \frac{N\log(C_3)+(1-\d\b) \a(Nt)}{ \a(Nt)} 
            = \frac{\log(C_3)+(1-\d\b) \a(t)}{ \a(t)}, 
        \]
        where for a set $A\subseteq [-1,1]$, $\overline{\dim}_{box}(A)$ denotes its upper box dimension.
        
        Since $Z_x$ is contained in countably many such sets by~\eqref{Z contained in liminf set} and since the upper box dimension dominates the Hausdorff dimension (which is stable under countable unions), we get that 
        \[ dim_H(Z_x) \leqslant \frac{\log(C_3)}{ \a(t)} + 1-\d\b, \]
        where $dim_H$ denotes the Hausdorff dimension.
        Taking the limit as $t\r\infty$ and $\d \r 1$, we obtain the desired dimension bound.
        
        Part $(2)$ of Theorem~\ref{thrm: Hdim and non-divergence} follows from Proposition~\ref{propn: main covering propn} and the Borel-Cantelli Lemma.
        More precisely, it follows from the statement of the Proposition that the set $Z_x(M,N,t,\d)$ has measure at most $C_3^N e^{-\d\b\a(Nt)}$.
        Choosing $t >0$ (and hence $M$) to be large enough, depending on $\d$ and $C_3$, we see that the measures of these sets are summable in $N$.

    \section{Bounded Orbits and Schmidt Games} 
\label{section: schmidt games}

	We describe a version of Schmidt's games played on intervals of the real line.
	These games were introduced in~\cite{KleinbockWeiss-SchmidtExpanding,KleinbockWeiss-SchmidtSLn+m} in the general setting of connected Lie groups building on earlier ideas of Schmidt~\cite{Schmidt-Games}.
    
    Fix a compact interval $I_0 \subset \R$ and a positive constant $\s>0$.
    For each $t>0$, consider the following contraction of $\R$:
    \[ \Phi_t (x) = e^{-\s t} x. \]
    Denote by $\mf{F} = \set{\Phi_t: t>0}$ this one-parameter semigroup of contractions.
    
    Now pick two real numbers $a,b >0$ and, following~\cite{KleinbockWeiss-SchmidtSLn+m,KleinbockWeiss-SchmidtExpanding}, we define a game,
played by two players Alice and Bob.
	First, Bob picks $t_0 >0$ and $x_1 \in \R$ so that the set $B_1 = \Phi_{t_0}(I_0) + x_1$ is contained in $I_0$.
	Then, Alice picks a translate $A_1$ of $\Phi_a(B_1)$ which
is contained in $B_1$, Bob picks a translate $B_2$ of $\Phi_b(A_1)$ which is contained in $A_1$, after that Alice picks a translate $A_2$ of $\Phi_a(B_2)$ which is contained in $B_2$, and so on.
	In other words, for $k \in \N$, we set
    \begin{equation} \label{eqn: t_k and s_k}
    	t_k = t_0 + (k-1)(a+b), \quad \text{and} \quad
        s_k = t_k + a.
    \end{equation}
    Thus, at the $k^{th}$ step of the game, Alice picks a translate $A_k$ of $\Phi_{s_k}(I_0)$ which is contained inside $B_k$. Then, Bob picks a translate $B_{k+1}$ of $\Phi_{t_{k+1}}(I_0)$ which is contained inside $A_k$.
    From compactness of $I_0$ and the definition of the sets $A_k$ and $B_k$, we see that the following intersections
    \begin{equation} \label{eqn: intersection}
    \bigcap_{k\geqslant 1} A_k = \bigcap_{k\geqslant 1} B_k
    \end{equation}
    are non-empty and consist of a single point. Note also that
    \begin{equation} \label{eqn: diameter of A_k and B_k}
    	\mathrm{diam}(A_k) = e^{-\s s_k} \mathrm{diam}(I_0), \quad
        \mathrm{diam}(B_k) = e^{-\s t_k} \mathrm{diam}(I_0),
    \end{equation}
    where the diameter of sets is with respect to the standard metric on $\R$.
    This game is referred to as the $\mathbf{(a,b)-}$\textbf{modified Schmidt game} on $I_0$.
    
    A subset $S \subseteq \R$ is said to be $\mathbf{(a,b)-}$\textbf{winning} if Alice can always pick her translates $A_k$ so that the point in the intersection~\eqref{eqn: intersection} always belongs to $S$, no matter how Bob picks his translates $B_k$.
    We say $S$ is $\mathbf{a}$-\textbf{winning} if it is $(a,b)$-winning for all $b>0$ and \textbf{winning} if it is $a$-winning for some $a$.
    
    \subsection{Admissible Curves and Induced Games}

	Suppose $G$ is a connected Lie group with Lie algebra $\mf{g}$ and $g_t$ is a $1$-parameter $\mrm{Ad}$-diagonalizable subgroup of $G$.
    Consider a $g_t$-admissible curve $\vp:I_0 \r \mf{g}$ as defined in~\ref{defn: admissible curves}, where $I_0$ is a compact interval in $\R$.
    Then, $g_t$ induces a Schmidt game on $I_0$ in the sense described above as follows.
    
    Suppose $\mf{g}_\a \subset \mf{g}$ is the eigenspace for the Adjoint action of $g_t$ which contains the image of $\vp$.
    The $\mf{F}_\a$-induced game on $I_0$ is given by the action of the one parameter semigroup $\mf{F}_\a = \set{\Phi_t: t>0}$ where for every $x\in \R$,
    \[ \Phi_t(x) = e^{-\a(t)}x = e^{-\a(1)t}x, \]
	and $\a(t)$ is the eigenvalue of $g_t$ corresponding to the eigenspace $\mf{g}_\a$ as in~\eqref{defn: eigenvalue}.

    The main result of this section states that the contraction hypothesis in addition to the following continuity property of the height function $f$ along unipotent orbits imply the winning property of bounded orbits.
    \begin{assumption} \label{assumption: bounded connected components}
    There exists $N\in \N$ such that for every $T, R>0$, there exists $M_1>0$ such that for all $x\in X$, $Y\in \mf{g}_\a$, $\norm{Y}\leq R$ and $M>M_1$, the following holds.
    \begin{equation} \label{eqn: bounded connected components}
    	\text{The set } \set{|s|\leqslant T: f(u(sY) x) > M} \text{ has at most } N \text{ connected components}.
        \tag{$\ast$}
    \end{equation}
    \end{assumption}

	The following is the main result of this section.
    \begin{theorem} \label{thrm: schmidt games}
    Let $X$ be a metric space equipped with a proper $G$-action. Suppose $g_t$ is an $Ad$-diagonalizable $1$-parameter subgroup of $G$ and $\vp:I_0\r \mf{g}$ is a $g_t$-admissible curve (Def.~\ref{defn: admissible curves}) satisfying the $\b$-contraction hypothesis (Def.~\ref{defn: height functions}) on $X$ for some $\b>0$.
    Assume further that the height function $f$ satisfies Assumption~\ref{assumption: bounded connected components}.
        Then, there exists $a_\ast>0$ such that for all $x\in X$ with $f(x)<\infty$, the set
        \begin{equation} \label{eqn: bndd orbits}
        \set{s\in I_0: \overline{\set{g_tu(\vp(s))x :t>0}} \text{ is compact in } X  } 
        \end{equation} 
is $a$-winning for the $\mf{F}_\a$-induced modified Schmidt game on $I_0$ for all $a > a_\ast$.    
    \end{theorem}
    
    \begin{corollary}[Corollary 3.4,~\cite{KleinbockWeiss-SchmidtSLn+m}]
    \label{cor: thickness of bounded orbits}
    Under the same hypotheses of Theorem~\ref{thrm: schmidt games}, the set in~\eqref{eqn: bndd orbits} is thick in $I_0$.
    \end{corollary}
    
    \begin{remark}
    The contraction hypothesis alone, without Assumption~\ref{assumption: bounded connected components}, can be used to show Corollary~\ref{cor: thickness of bounded orbits}.
    This can be done by a straightforward adaptation of the argument in~\cite{KleinbockWeiss-BoundedModuli}.
    \end{remark}
    
    \begin{proof}[Proof of Theorem~\ref{thrm: schmidt games}]
    Denote by $f$ the height function in the definition of the $\b$-contraction hypothesis.
    Suppose $\tilde{c}\geqslant 1$ is as in~\eqref{eqn: CH} and $C_1 \geqslant$ is the constant in the conclusion of Lemma~\ref{lemma: approximate curve with tangent}.
    
    Next, let $C_H$ denote the H\"{o}lder constant of $\dot{\vp}$.
    Let $\mc{O}$ denote a compact neighborhood of identity in $G$ containing the image under the exponential map of a ball of radius $C_H |I_0|$ around $0$ in $\mf{g}$.
    Denote by $C = C_\mc{O} \geq 1$ a constant so that~\eqref{defn: log lipschitz} holds.
    
    Let $N\in \N$ be as in Assumption~\ref{assumption: bounded connected components}.
    Choose $a_\ast$ to be sufficiently large so that
    \begin{equation} \label{defn: a star}
    	\a(a_\ast) \geqslant  \frac{ \log (40\tilde{c}C_1 C_\mc{O}^2 ) }{\b} + \log\left( 10 (N+1) \right).
    \end{equation}
    Fix some $a> a_\ast, b>0$ and $x\in X$.
    We show that there exists some $M\geqslant 1$ and a choice of subintervals $A_k$ for Alice so that for all $k\geqslant 1$ and all $s\in A_k$, we have
    \begin{equation} \label{Alice wins}
    	f( g_{t_{k+1}} u(\vp(s))x) \leqslant M,
    \end{equation}
    where $t_k$ is given by~\eqref{eqn: t_k and s_k}.
    Thus, by Definition~\ref{defn: height functions}, this shows that the point $s_0$ in the intersection~\eqref{eqn: intersection} will have that $g_t u(\vp(s_0))x$ is bounded in $X$ for all $t>0$.
    
    By Definition~\ref{defn: height functions}, there exists a constant $\tilde{b}>0$ depending on $a$ and $b$, so that the following holds for all $y\in X$ and all $s\in I_0$.
    \begin{equation} \label{eqn: a+b contraction}
    	\frac{1}{2} \int_{-1}^1 f(g_{a+b} u(r\dot{\vp}(s))y)\;dr \leqslant \tilde{c}e^{-\b\a(a+b)}  f(y) +\tilde{b}.
    \end{equation}
    
    Now, suppose that Bob chose some $t_0>0$ and a subinterval $B_1 \subset I_0$ to initialize the game.
    Let $\g >0$ be the H\"{o}lder exponent of the derivative of $\vp$ and define $M_0$ as follows:
    \begin{equation*}
    M_0 := \sup_{s\in I_0, 1\leqslant j \leqslant 1/\g +1} f(g_{j(a+b)+t_0} u(\vp(s))x).
    \end{equation*}
    By the properties of $f$ in Definition~\ref{defn: height functions} and the compactness of $I_0$, it follows that $M_0 $ is finite.
    
    Let $T = e^{\a(a+b)}|I_0|/2$ and $R = \sup_{s\in I_0} \norm{\dot{\vp}(s)}$.
    Let $M_1>0$ be as in Assumption~\ref{assumption: bounded connected components} applied with $T$ and $R$.
    Define $M$ by
    \begin{equation} \label{defn: M}
    	M = 40\tilde{b}C_1 C_\mc{O}^2  + M_0 + M_1 C_\mc{O},
    \end{equation}
    where $\a$ is such that $\mf{g}_\a$ is the eigenspace inside $\mf{g}$ containing the image of $\vp$.
    
    In the first $\lfloor 1/\g+1 \rfloor$ steps of the game, Alice may choose her intervals $A_k \subset B_k$ anyway she likes. By definition of $M_0$ and $M$,~\eqref{Alice wins} is satisfied for $1\leqslant k \leqslant 1/\g+1$.
    
    The rest of the proof consists of $2$ steps. First, we show that no matter how Bob chooses his sets $B_k$, the following integral estimate will always be satisfied for all $k\geqslant 1/\g + 1$:
    \begin{equation} \label{eqn: main game estimate}
    	\frac{1}{|B_k|}\int_{B_k} f(g_{t_{k+1}}u(\vp(s))x)\;ds \leqslant
        2\tilde{c}C_1 e^{-\b\a(a+b)} \frac{1}{|B_k|}\int_{B_k} f(g_{t_{k}}u(\vp(s))x)\;ds + 2\tilde{b}C_1.
    \end{equation}
    Then, we show that the estimate~\eqref{eqn: main game estimate} implies that Alice can choose her sets $A_k \subset B_k$ so that~\eqref{Alice wins} is satisfied, completing the proof.
    
    To show~\eqref{eqn: main game estimate}, let $k\geqslant 1/\g + 1$ and let $B_k \subset I_0$ be a subinterval of length $e^{- \a(t_k) } |I_0|$.
    By an argument identical to that of Lemma~\ref{lemma: approximate curve with tangent}, it follows that
    \begin{equation*}
    	\int_{B_k} f(g_{t_{k+1}}u(\vp(s))x)\;ds \leqslant 
    		C_1 \int_{B_k} \int_{-1}^1 f(g_{a+b} u(r \dot{\vp}(s))g_{t_k}u(\vp(s))x)\;dr ds.
    \end{equation*}
    Then, by~\eqref{eqn: a+b contraction}, we get
    \begin{equation*}
    	\int_{B_k} f(g_{t_{k+1}}u(\vp(s))x)\;ds \leqslant 
        2C_1 \tilde{c}e^{-\b\a(a+b)} \int_{B_k} f(g_{t_{k}}u(\vp(s))x)\;ds + 2C_1 \tilde{b} |B_k|.
    \end{equation*}
    This proves~\eqref{eqn: main game estimate}. We complete the proof by induction, noting that~\eqref{Alice wins} is satisfied for all $1\leqslant k \leqslant 1/\g +1$.
    Since $B_k \subset A_{k-1}$, by the induction hypothesis, we get that for all $s\in B_k$, $f(g_{t_k}u(\vp(s))x) \leqslant M$.
    Thus, the estimate in~\eqref{eqn: main game estimate} becomes
    \begin{equation*}
    \frac{1}{|B_k|}\int_{B_k} f(g_{t_{k+1}}u(\vp(s))x)\;ds \leqslant
        2\tilde{c}C_1 e^{-\b\a(a+b)} M + 2\tilde{b}C_1.
    \end{equation*}    
    By Chebyshev's inequality, the fact that $a>a_\ast$ chosen in~\eqref{defn: a star}, and the choice of $M$ in~\eqref{defn: M}, we obtain the following measure estimate:
    \begin{align} \label{eqn: measure}
    \left|\set{ s\in B_k:  f(g_{t_{k+1}} u(\vp(s))x) > M/C^2_\mc{O}} \right|
    &\leqslant
    \left[ 2\tilde{c}C_1 C^2_\mc{O} e^{-\b\a(a+b)} + \frac{ 2\tilde{b}C_1 C^2_\mc{O}}{M} \right] |B_k| \nonumber\\
    &\leqslant  |B_k|/10.
    \end{align}
    
	Let $s_0$ be the center of the interval $B_k$ and let $s\in B_k$ be any other point.
    Then, we have that
    \[ g_{t_{k+1}} u(\vp(s)) = u(O\left( e^{\a(t_{k+1}-(1+\g)t_k)} \right)) u(r \dot{\vp}(s_0)) 
    g_{t_{k+1}} u(\vp(s_0)), \]
	where $r = (s-s_0)e^{\a(t_{k+1})}$ and $\g$ is the H\"{o}lder exponent of $\dot{\vp}$.
    Since $k \geq 1/\g+1$, the element $u(O\left( e^{\a(t_{k+1}-(1+\g)t_k)} \right))$ belongs to our chosen bounded neighborhood $\mc{O}$ of identity which is independent of all the parameters.
    Hence, by the log Lipschitz property~\eqref{defn: log lipschitz} of $f$, we obtain
    \begin{equation}\label{eqn: affine transform}
     f(  u(r \dot{\vp}(s_0)) g_{t_{k+1}} u(\vp(s_0)) x) > M/C_\mc{O} 
     \Longrightarrow f(g_{t_{k+1}} u(\vp(s)) x) > M/C^2_\mc{O} ,\quad r = (s-s_0)e^{\a(t_{k+1})}.
    \end{equation}
    Moreover, since $|B_k| = e^{-\a(t_k)}|I_0|$, we have $|r| \leq e^{\a(a+b)}|I_0|/2 =T$.
    
    Thus, since $M/C_\mc{O}  > M_1$, by Assumption~\ref{assumption: bounded connected components}, the set
    \begin{equation*}
    \set{ |r|\leq T:  f(  u(r \dot{\vp}(s_0))  g_{t_{k+1}} u(\vp(s_0)) x) > M/C_\mc{O} }
    \end{equation*}
    has at most $N$ connected components.
    In particular, the complement of this set has at most $N+1$ connected components (intervals).
    
    Moreover, the measure estimate in~\eqref{eqn: measure}, combined with~\eqref{eqn: affine transform}, imply that 
    \begin{align} \label{eqn: horocycle measure}
      \left |\set{ |r|\leq T:  f(  u(r \dot{\vp}(s_0))  g_{t_{k+1}} u(\vp(s_0)) x) > M /C_\mc{O} } \right|
      &\leqslant   2T/10.
    \end{align}
    Denote by $Q$ the set on the left-hand side of~\eqref{eqn: horocycle measure}.
     Suppose that each connected component of $[-T,T]\backslash Q$ has length at most $2e^{-\a(a)}T$. 
     Then, since $[-T,T]\backslash Q$ has at most $N+1$ components, we get that 
     \[|[-T,T]\backslash Q| \leqslant 2(N+1)e^{-\a(a)}T < 2T/10,\]
     by the choice of $a$.
     This contradicts the measure estimate in~\eqref{eqn: horocycle measure}.
     
     It follows that we can find a subinterval $\tilde{A}_k$ of $[-T,T]$ of length $2e^{-\a(a)}T$ which is disjoint from the set in~\eqref{eqn: horocycle measure}.
     Let $A_k$ be defined as follows:
     \begin{equation*}
     	A_k = e^{-\a(t_{k+1})}\tilde{A}_k + s_0.
     \end{equation*}
     Then, $A_k$ is a subinterval of $B_k$ of length $e^{-\a(a)}|B_k|$.
     Moreover, applying the the log Lipschitz property of $f$ once more, we see that for all $s\in A_k$,
     \begin{equation*}
     f(  u(r \dot{\vp}(s_0)) g_{t_{k+1}} u(\vp(s_0)) x) \leqslant M/C_\mc{O} 
     \Longrightarrow f(g_{t_{k+1}} u(\vp(s)) x) \leqslant M ,\quad r = (s-s_0)e^{\a(t_{k+1})}.
     \end{equation*}
    This proves~\eqref{Alice wins} and concludes the proof.

    \end{proof}

    \section{The Contraction Hypothesis and Shrinking Curves}
\label{section: shrinking nondivergence}

	The purpose of this section is to demonstrate the link between the contraction hypothesis and the growth of orbits.
    In all the situations we consider, the height function $f$ which satisfies the contraction hypothesis
    also has the property that the ratio of $1+\log f(\cdot)$ and $1+d(\cdot,x_0)$ is uniformly bounded from above and below for any fixed base point $x_0 \in G/\G$, where $d(\cdot,\cdot)$ is the Riemannian metric  on $G/\G$.

    In fact, we establish the much stronger statement on the quantitative non-divergence of expanding translates of shrinking segments of admissible curves.
    In particular, Proposition~\ref{propn: non-divergence of shrinking curves} below implies that orbits with linear growth have measure $0$ using the Borel-Cantelli lemma along with Chebyshev's inequality.
    Throughout this section, we retain the same notation as in Section~\ref{section: abstract setup}.

	\begin{proposition} \label{propn: non-divergence of shrinking curves}
    	Let $G$ be a real Lie group and $X$ be a metric space equipped with a proper $G$-action. Suppose $g_t$ is an $\mrm{Ad}$-diagonalizable one-parameter subgroup of $G$ and $\vp$ is a $g_t$-admissible curve satisfying the $\b$-contraction hypothesis on $X$. Suppose $\d \in [0,\b)$ is fixed.
        Then, for all $x_0\in X$ with $f(x_0)<\infty$,
    	\begin{equation*}
        	\sup\limits_{\substack{t\geqslant 0,s_0\in[-1,1]\\ J_t + s_0 \subseteq [-1,1] }}
            \frac{1}{|J_t|} \int_{J_t+s_0} f(g_t u(\vp(s)) x_0) \;ds < \infty,
    	\end{equation*}
        where $J_t := [-e^{-\d \a(t)}, e^{-\d \a(t)}]$. Moreover, the supremum can be taken to be uniform over base points $x_0 \in \set{f \leqslant M}$ for any $M>0$.
	\end{proposition}
    
    \begin{proof}
    Let a choice of $\d \in [0,\b)$ be fixed.
    Suppose $s_0 \in [-1,1]$ and $n \geq 0$ is an integer.
    Fix $t>0$ so that~\eqref{eqn: CH} holds with constants $\tilde{c}$ and $\tilde{b}$.
   	  By Lemma~\ref{lemma: approximate curve with tangent}, we have    
      \begin{equation} \label{eqn: first step in induction}
          \int_{J_{nt}+s_0} f(g_{(n+1)t} u(\vp(s))x_0)\;ds \leqslant 
              C_1 \int_{J_{nt}+s_0} \int_{-1}^1 f(g_t u(r\dot{\vp}(s)) g_{nt} u(\vp(s))x_0) \;dr\;ds.
      \end{equation}
      Since $C_1$ and $\tilde{c}$ are independent of $t$, we may assume that $t>0$ is sufficiently large so that
      \begin{equation*}
      2C_1 \tilde{c} e^{-(\b-\d) \a(t)} <1.
      \end{equation*}
      Therefore, by~\eqref{eqn: first step in induction} and~\eqref{eqn: CH}, we get
      \begin{align*}
      	\int_{J_{nt}+s_0} f(g_{(n+1)t} u(\vp(s))x_0)\;ds 
            &\leqslant 2C_1 \tilde{c} e^{-\b \a(t)} \int_{J_{nt}+s_0}  f(g_{nt} u(\vp(s))x_0) \;ds + 2C_1 \tilde{b} |J_{nt}|.
      \end{align*}
      Next, for all $n\geq 1$, since $J_{nt} \subseteq J_{(n-1)t}$ and $f\geq 0$, we get 
      \begin{align*}
      	\int_{J_{nt}+s_0} f(g_{(n+1)t} u(\vp(s))x_0)\;ds 
            &\leqslant 2C_1 \tilde{c} e^{-\b \a(t)} \int_{J_{(n-1)t}+s_0}  f(g_{nt} u(\vp(s))x_0) \;ds + 2C_1 \tilde{b} |J_{nt}|.
      \end{align*}
      Moreover, since $|J_{(n-1)t}|/|J_{nt}| = e^{\d\a(t)}$, the above inequality implies
      \begin{align} \label{eqn: inductive step}
      	\frac{1}{|J_{nt}|} \int_{J_{nt}+s_0} f(g_{(n+1)t} &u(\vp(s))x_0)\;ds \nonumber\\
            &\leqslant 2C_1 \tilde{c} e^{-(\b-\d) \a(t)} 
            \frac{1}{|J_{(n-1)t}|}\int_{J_{(n-1)t}+s_0}  f(g_{nt} u(\vp(s))x_0) \;ds + 2C_1 \tilde{b}.
      \end{align}
      
      Define $M_0>0$ and $M$ by
      \begin{align*}
      	M_0 &=  \frac{1}{2} \int_{-1}^1 f(u(\vp(s))x_0) \;ds ,\\
        M &= \max\set{M_0,2C_1 \tilde{c} e^{-(\b-\d) \a(t)} M_0 + 2C_1 \tilde{b}, 
        	\frac{2C_1 \tilde{b}}{\left(1-2C_1 \tilde{c} e^{-(\b-\d) \a(t)}\right)}}.
      \end{align*}
      
      We claim that
      \begin{equation} \label{eqn: M as the sup}
      	\sup_{n\geqslant 0,s_0} \frac{1}{|J_{nt|}} \int_{J_{nt}+s_0} f(g_{(n+1)t}u(\vp(s)x_0)\;ds \leqslant M.
      \end{equation}
      We proceed by induction on $n$.
      When $n=0$, inequality~\eqref{eqn: inductive step}, the definition of $M_0$ and the fact that $M_0\leqslant M$ show that the integrand in~\eqref{eqn: M as the sup} is bounded above by $M$.
      Inequality~\eqref{eqn: inductive step} and the definition of $M$ finish the proof of the claim by induction.
      
      The conclusion of the proposition follows from the log-smoothness of $f$.
      Furthermore, we note that $M$ can be chosen to be uniform over the base point $x_0$ as it varies in sublevel sets of $f$ as evident from the definition of $M_0$.

    \end{proof}

    \section{Dynamics in Linear Representations}
\label{section: linear expansion}
	This section is dedicated to proving estimates on the average rate of expansion of vectors in linear representations of $\mrm{SL}(2,\R)$.
    The main result is Proposition~\ref{propn: expansion in linear representations}.
    In subsection~\ref{section: avoidance}, we prove an important fact regarding the orbit of a highest weight vector which will allow us to obtain precise average expansion rates in the sequel.
    
	\subsection{(C, $\mathbf{\a}$)-good functions}
	We recall the notion of $(C,\a)$-good functions introduced by Kleinbock and Margulis in~\cite{KleinbockMargulis} and used, in different form, in prior works of Dani, Margulis and Shah.
    \begin{definition} \label{defn: C alpha good functions}
    	A function $f: \R^m \r \R$ is $(C,\a)$-good on some subset $B\subset \R^m$ of finite Lebesgue measure
        if there exist constants $C, \a >0$ such that for any $\e>0$, one has
        \begin{align*}
        	\left|\set{x\in B: |f(x)| < \e } \right| 
            \leq C \left( \frac{\e}{ \sup_{x\in B} |f(x)| }  \right)^\a |B|,
        \end{align*}
        where, for a Borel set $A\subseteq \R^m$,  $|A|$ denotes its Lebesgue measure.
    \end{definition}
    
    The following lemma summarizes some basic properties of $(C,\a)$-good functions which will be useful for us. The proof follows directly from the definition.
    \begin{lemma} \label{lemma: properties of (C,alpha)-good functions}
     Let $C,\a >0$. Then,
     	\begin{enumerate}
     		\item If $f$ is a $(C,\a)$-good function on $B$, then so is $|f|$.
            \item If $f_1,\dots, f_n$ is a collection of $(C,\a)$-good function on $B$, then so is $\max_{k} |f_k|$.
     	\end{enumerate}
    \end{lemma}
    
    An important class of $(C,\a)$-good functions is polynomials. The exact exponent will be of importance to us and so we recall the following fact.
    \begin{proposition}[Proposition 3.2,~\cite{KleinbockMargulis}]
    \label{propn: polynomials are (C,alpha)-good}	
        For any $k \in \N$, any polynomial in $\R[x]$ of degree at most $k$ is $(2k(k+1)^{1/k},1/k)$-good on any interval in $\R$.
    \end{proposition}
    
    The following elementary lemma concerning polynomials will be useful for us.
    \begin{lemma} \label{lemma: lower sup estimate for polynomials}
    	For each $k\in \N$, there exists some $\rho >0$, such that any polynomial $p\in\R[x]$ of degree at most $k$ of the form $p(x) = \sum_{i=0}^k c_ix^i$ satisfies
        \begin{align*}
        	\sup_{x\in [-1,1]} |p(x)| \geq \rho \max_{0\leq i \leq k} |c_i|.
        \end{align*}
    \end{lemma}
    \begin{proof}
    	Let $k \in \N$ and suppose the lemma does not hold.
        Then, there exists a sequence of vectors $v_n \in \R^{k+1}$ with
        $\norm{v_n}_\infty = 1$ such that
        \begin{align} \label{eqn:sup tending to 0}
        	\sup_{x\in [-1,1]} \left| p_n(x) \right| < \frac{1}{n},
        \end{align}
        where for each $n$,
        \begin{equation*}
        	p_n(x) =  \sum_{0\leq i \leq k} v_n^{(i)}x^i.
        \end{equation*}
        By passing to a subsequence, we may assume that $v_n $ converges to a vector $v_0 \neq 0$.
        Thus, $p_n$ converges to $p_0$ on $[-1,1]$ in the uniform norm.
        But, then, by~\eqref{eqn:sup tending to 0}, we have $p_0 \equiv 0$ on $[-1,1]$.
        This necessarily implies that $v_0 = 0$ which is a contradiction.
    \end{proof}

    
    \subsection{Expansion in $\mrm{\mathbf{SL(2,R)}}$ Representations}

  Throughout this section, we fix a one-parameter $\mrm{Ad}$-diagonalizable subgroup of $G=\mrm{SL}(2,\R)$ which we denote by $g_t$.
 	Then, $\mf{g} = \mrm{Lie}(G)$ decomposes as a direct sum of eigenspaces of $\mrm{Ad}(g_t)$ as follows:
    \begin{equation}\label{eqn: SL_2 roots}
    	\mf{g} = \mf{g}_{-\a} \oplus \mf{g}_0 \oplus \mf{g}_\a,
    \end{equation}
    where $\a$ is a non-trivial character of the group $A=\set{g_t:t\in \R}$ such that $\a(g_t)>0$ for all $t>0$ and $\mf{g}_0$ consists of fixed vectors of $\mrm{Ad}(g_t)$.
    Let $H_0 \in \mf{g}_0$ be such that $g_t = \exp(t H_0)$. Let $X \in \mf{g}_\a \backslash\set{0} $ and let $u_s$ denote the following one-parameter horocyclic subgroup
    \begin{equation*}
    	u_s = \exp(s X).
    \end{equation*}

    Let $P$ denote the set of all characters of $A$. Then, $\a$ induces a partial order $\leqslant$ on $P$ as follows: $\l \leqslant \mu$ if and only if $\mu- \l$ is a positive multiple of $\a$.
    Given any irreducible representation $V$ of $G$, we can decompose $V$ into weight spaces for the $A$ action. The set of restricted weights of $V$ contains a unique maximal element for the partial order, called the highest weight.
    Denote the set of all the highest weights of $G$ by $P^+$, i.e. $P^+$ consists of characters of $A$ which occur as highest weights in some irreducible representation of $G$.
    From the representation theory of $\mrm{SL}(2,\R)$, we can identify $P^+$ with $\N\cup \set{0}$.
  
  The following is the main result of this section.
  \begin{proposition} \label{propn: expansion in linear representations}
	Suppose $V$ is a non-trivial representation of $G=\mrm{SL}(2,\R)$ and let $P^+(V)$ denote the set of highest weights appearing in the decomposition of $V$ into irreducible representations.
    Define
    \begin{equation*}
        \l := \max  P^+(V), \qquad \d_\l := 2\l(H_0)/\a(H_0),
    \end{equation*}
    where $\a$ is as in~\eqref{eqn: SL_2 roots}.
    Then, for all $\beta \in \left(0,1\right)$, there exists a constant $D= D(\b) \geq 1 $
    such that for all $t>0$ and all $w \in V$,
    \begin{equation}\label{eqn: CH in V}
    	\frac{1}{2} \int_{-1}^1 \norm{g_t u_s w}^{-\b /\d_\l  } \;ds \leqslant 
        		D e^{-\beta \a(H_0) t/2}\norm{\pi_\l(w)}^{-\beta/\d_\l},
    \end{equation}
    where $\pi_\l:V\r V$ denotes the $\mrm{SL}(2,\R)$-equivariant projection onto the direct sum of irreducible sub-representations of $V$ with highest weight $\l$.
  \end{proposition}
  
  \begin{proof}
    Suppose $w\in V$ and write $v= \pi_\l(w)$.
    Then, we have that $\norm{g_tu_s w} \geqslant \norm{g_tu_sv}$, for all $t$ and $s$.
    In particular, it suffices to prove~\eqref{eqn: CH in V} with $v$ in place of $w$ and we may assume that $\l$ is the only highest weight appearing in $V$.
    
    Since $\mrm{SL}(2,\R)$ is semisimple, $V$ decomposes into irreducible representations as follows:
    \[ V = V_1 \oplus \cdots \oplus V_r. \]
    For $1\leq i \leq r$, let $\pi_i: V \r V_i $ denote the associated projections and note that $u_s$ commutes with $\pi_i$ for all $i$.
    Note that all the $V_i$ have the same dimension since they have the same highest weight.
    Let $n \in \N$ be such that
    \[ \mathrm{dim}(V_i) = n + 1, \]
    for all $1\leq i\leq r$.
    From the the description of $\mrm{SL}(2,\R)$ representations, we get that
    \begin{equation} \label{n=delta_lamda }
    	n = \d_\l.
    \end{equation}

    Let $1\leq i\leq r$ be fixed.
    By the standard description of irreducible $\mrm{SL}(2,\R)$ representations, $V_i$ decomposes into $1$ dimensional eigenspaces for the action of $g_t$ as follows:
    \[ V_i = W_0^{(i)} \oplus W_1^{(i)} \oplus \cdots \oplus W_{n}^{(i)},  \]
    where we assume that $W_0^{(i)}$ denotes the highest weight subspace of $V_i$.
    In particular, for each $w\in W_0^{(i)}$,
    \begin{equation*}
    	g_t w = e^{\l(H_0)t} w.
    \end{equation*}

    Let $q_l: V_i \r W_l^{(i)}$ denote the associated projections.
    Let $\set{w_l^{(i)}:0\leq l\leq n}$ denote a basis of $V_i$ consisting of eigenvectors of $g_t$ and write
    \[ \pi_i(v) =  \sum_{l=0}^{n} c_{l}^{(i)} w_l^{(i)}. \]
    Note that for each $l$, we have that
    \begin{equation*}
    	u_s w_l^{(i)} =\sum_{k=0}^l \binom{l}{k} s^{l-k}w_k^{(i)}.
    \end{equation*}
    In particular, we get the following
    \begin{equation} \label{eqn:highest weight coefficient}
     q_0 ( \pi_i(u_sv)) = q_0 ( u_s\pi_i(v) ) = 
     	 \sum_{k=0}^{n}	c_k^{(i)}  s^{k} w_{0}^{(i)}.
    \end{equation}

    Denote by $\norm{\cdot}_\infty$ an $\ell^\infty$ norm on $V$ with respect to the basis chosen above for each irreducible representation.
    Note that all coordinates of $\pi_i(v)$ appear in the polynomial in~\eqref{eqn:highest weight coefficient}.
    In particular, this implies
    \begin{equation} \label{eqn:expansion of highest weight}
    	\norm{g_t u_s \pi_i(v) }_\infty \geqslant \norm{g_t q_0 ( \pi_i(u_sv))}_\infty 
        = e^{\l(H_0) t} \norm{q_0 ( \pi_i(u_sv))}_\infty.
    \end{equation}
    
    Denote by $V_\l$ the direct sum of the highest weight subspaces of $ V$. More precisely, let
    \begin{equation*}
    	V_\l = \bigoplus_{1\leq i \leq r} W_0^{(i)},
    \end{equation*}
    and let $\pi_+ : V \r V^+$ denote the associated projection.
    Hence, for all $w\in V$, by~\eqref{eqn:expansion of highest weight}, we have that
    \begin{equation} \label{eqn: expansion by g_t}
    	\norm{g_t w}_\infty \geqslant \norm{g_t \pi_+(w) }_\infty \geqslant e^{\l(H_0) t} \norm{\pi_+(w)}_\infty.
    \end{equation}
    
    The polynomials in~\eqref{eqn:highest weight coefficient} have degree at most $n=\d_\l$.
    Hence, by Lemma~\ref{lemma: properties of (C,alpha)-good functions} and Proposition~\ref{propn: polynomials are (C,alpha)-good},
    we see that $\norm{\pi_+(u_s v)}_\infty$ is $(C,\d_\l)$-good on $[-1,1]$ for 
    $C$ as in Proposition~\ref{propn: polynomials are (C,alpha)-good}.    
    Now, by~\eqref{eqn:highest weight coefficient} and Lemma~\ref{lemma: lower sup estimate for polynomials}, there exists some $\rho >0$ such that 
    \begin{equation*}
    \sup_{s\in [-1,1]} \norm{\pi_+(u_s v)}_\infty \geqslant \rho \norm{v}_\infty.
    \end{equation*}    
    Thus, by definition of $(C,\a)$-good functions, for any $\e>0$, we have
    \begin{equation} \label{eqn: (C,alpha)-good measure estimate}
    	\left|\set{ s\in [-1,1]: \norm{\pi_+(u_s v)}_\infty < \e \norm{v}_\infty } \right|
        \leqslant 2 C \left( \frac{\e}{\rho} \right)^{1/\d_\l}.
    \end{equation}
    Denote by $E(v,\e)$ the set on the left-hand side of inequality~\eqref{eqn: (C,alpha)-good measure estimate}.
    Let $\b \in (0,1)$.

    Without loss of generality, we may assume $\norm{v}_\infty = 1$.
    Then, for $n\in \N$,
    by~\eqref{eqn: expansion by g_t} and~\eqref{eqn: (C,alpha)-good measure estimate},
    we get
    \begin{align*}
    	\int_{E(v,2^{-n}\rho) \setminus E(v,2^{-(n+1)}\rho) } \norm{g_t u_sv}_\infty^{-\b/\d_\l}  \;ds 
        &\leqslant e^{-\b \l(H_0) t/\d_\l} \int_{E(v,2^{-n}\rho) \setminus E(v,2^{-(n+1)}\rho) }
        \norm{\pi_+( u_sv)}_\infty^{-\b/\d_\l}\;ds \\
        &\leqslant e^{-\b \a(H_0) t/2} 2^{\b(n+1)/\d_l} \rho^{-\b/\d_\l} 2 C 2^{- n/\d_\l} \\
        &= \rho^{-\b/\d_\l} 2^{1+\b/\d_l} C 2^{-(1-\b)n/\d_\l} e^{-\b \a(H_0) t/2}.
    \end{align*}
    Now, note that~\eqref{eqn: (C,alpha)-good measure estimate} implies that $|E(v,0)| = 0$.
    Hence, since 
    \[ [-1,1] = E(v,0) \sqcup \left( \bigsqcup_{n\geqslant 0} E(v,2^{-n}\rho) \setminus E(v,2^{-(n+1)}\rho) \right), \]
    we get that
    \[  \frac{1}{2} \int_{-1}^1 \norm{g_t u_s v}_\infty^{-\b/\d_\l} \;ds \leqslant 
    \frac{\rho^{-\b/\d_\l} 2^{\b/\d_l} C }{ 1 - 2^{(1-\b)/\d_\l}}e^{-\b \a(H_0) t/2}. \]
    Thus, the claim of the Proposition follows since all norms are equivalent.
  \end{proof}

    \subsection{Avoidance of Non-Extremal Subspaces}
\label{section: avoidance}

	The purpose of this section is to prove a useful property of the orbit of a highest weight vector under a semisimple group.
    This property will allow us to obtain precise expansion rates in the situations we are interested in.

	Suppose $G$ is a semisimple Lie group with Lie algebra $\mf{g}$, and $\mbf{S}$ is a maximal split torus in $G$ which we also identify with its Lie algebra.
    Denote by $\Delta \subset \mbf{S}^\ast$ the set of roots on which we fix an order and denote by $\Delta^+$ the subset of positive roots.
    Define the following subalgebras of $\mf{g}$
    \[ \mf{n}^+ = \bigoplus_{\a \in \Delta^+} \mf{g}_\a ,\qquad \mf{b} = \mf{g}_0 \oplus \mf{n}^+, \]
    where $\mf{g}_\a$ denotes the root space corresponding to $\a$.
    Denote by $N^+$ and $B$ the subgroups of $G$ whose Lie algebras are $\mf{n}^+$ and $\mf{b}$ respectively.

    We let $\mc{W}$ denote the Weyl group of $(G,\mbf{S},\Delta)$ and recall that $\mc{W}$ acts naturally on $\mbf{S}^\ast$.
    The Bruhat decomposition of $G$~\cite[Section 3, Theorem 1]{Bourbaki-4-6} implies
    \begin{equation} \label{eqn: general Bruhat}
    G = \bigcup_{\mrm{w} \in \mc{W}} B \mrm{w} B.
    \end{equation}

    Given a representation $V$ of $G$ and a linear functional $\mu\in \mbf{S}^\ast$, we denote by $V^\mu$ the weight subspace of $V$ with weight $\mu$.
    
    \begin{proposition} \label{propn: avoidance}
    Suppose $V$ is an irreducible representation of $G$ with highest weight $\l$. 
    Then, for all $0\neq v \in  V^{ \l} $,
    \begin{equation*}
    G\cdot v \quad \bigcap \bigoplus_{ \mu \in \mbf{S}^\ast\backslash \mc{W}\cdot\l  } V^\mu = \emptyset.
    \end{equation*}
    \end{proposition}
    
	\begin{proof}
	Let $0\neq v \in  V^{ \l} $ and $g\in G$.
    Denote by $\pi :V \r \bigoplus_{\mrm{w} \in \mc{W} } V^{\mrm{w} \cdot \l}$ the projection parallel to the weight spaces of $\mbf{S}$.
    It suffices to show that $\pi(gv) \neq 0$.

    Using the Bruhat decomposition~\eqref{eqn: general Bruhat}, we can write
    \[ g = b_1 \mrm{w} b_2,  \]
    for some $b_1,b_2 \in B$ and $\mrm{w}\in \mc{W}$.
    The group $B$ stabilizes the line $\R\cdot v$.
    In particular, we have that $gv \in b_1 \mrm{w} V^{\l} \subseteq b_1 V^{\mrm{w}\cdot \l} $.
    
    We can further decompose $b_1$ as follows.
    \[ b_1 = n^+m, \]
    where $n^+\in N^+$ and $m \in C_G(\mbf{S})$ commutes with $\mbf{S}$.
    In particular, $m$ preserves the eigenspaces of $\mbf{S}$ and thus we have 
    \begin{equation}\label{eqn: orbit hits extremal weight}
     gv \in b_1 V^{\mrm{w}\cdot \l} = n^+ V^{\mrm{w}\cdot \l} .
    \end{equation}

    Let $Y \in \mf{n}^+$ be such that $n^+ = \exp(Y)$.
    Denote by $\rho : G \r \mrm{GL}(V)$ the representation of $G$ on $V$ and let $d\rho: \mf{g} \r \mf{gl}(V)$ denote its derivative.
    Then, since $Y$ is nilpotent, so is $d\rho(Y)$.
    In particular, $\rho(n^+) = \exp(d\rho(Y))$ is a polynomial in $d\rho(Y)$ of the form
        \begin{equation} \label{eqn: n+ is nilpotent}
        	\rho(n^+) = \mrm{I} + d\rho(Y) + \cdots + \frac{d\rho(Y)^k}{k!},
        \end{equation}
        for some $k\in \N$, where $\mrm{I}$ is the identity map. From the standard representation theory of semisimple Lie groups, we have
        \begin{equation*}
        	d\rho(\mf{g}_\a) V^\mu \subseteq V^{\a + \mu},
        \end{equation*}
        for any root $\a \in \Delta$ and any weight $\mu \in \mbf{S}^\ast$.
        Thus, for each $1\leq j \leq k$, we see that
        \[ d\rho(Y)^j V^{\mrm{w}\cdot \l} \subseteq V^{\k}, \qquad \k = \mrm{w}\cdot \l + \sum_{\a \in \Delta^+} k_\a \a, \]
        for some non-negative integers $k_\a$, at least one of which is non-zero, and, in particular, $V^\k \cap V^{\mrm{w}\cdot \l} = \set{0}$.
        Hence, in view of~\eqref{eqn: n+ is nilpotent}, for all $w\in V^{\mrm{w}\cdot \l}$,  we have
        \[\pi^{\mrm{w}\cdot \l}( \rho(n^+)w) = w,\]
        where $\pi^{\mrm{w}\cdot \l}:V \r V^{\mrm{w}\cdot \l} $ denotes the projection parallel to the eigenspaces of $\mbf{S}$.
        Combined with~\eqref{eqn: orbit hits extremal weight}, this shows that $\pi(gv) \neq 0$ as desired.
	\end{proof}

    \section{The Contraction Hypothesis in Homogeneous Spaces of Rank One}
	\label{section: height function rank 1}
	
    Throughout this section, $G$ is a  simple Lie group of real rank $1$ and $\G$ is a lattice in $G$. We let $X = G/\G$.
    The goal of this section is to construct a height function on $X$ and show that it satisfies the strong $\b$-contraction hypothesis for admissible curves.
    The main result of this section, Theorem~\ref{thrm: CH in rank 1}, combined with those of Sections~\ref{section: abstract setup},~\ref{section: schmidt games} and~\ref{section: shrinking nondivergence} complete the proof of Theorem~\ref{thrm: rank 1 DOA}.

    \subsection{Construction of a Height Function}
	Following~\cite{EskinMargulis-RandomWalks} and~\cite{BQ-RandomWalkRecurrence},
    we construct a proper function $\tilde{\a} : G/\G \r \R_+$ which will allow us to control recurrence of trajectories to compact sets.
    
    By the work of Garland and Raghunathan in~\cite{GarlandRaghunathan}, there exist finitely many $\G$-conjugacy classes of maximal unipotent subgroups $\set{U_i:1\leq i \leq p}$ of $G$ such that $U_i \cap \G$ is a lattice in $U_i$.
    Moreover, for any sequence $g_n \in G$ such that $g_n \G$ tends to infinity in $G/\G$, after passing to a subsequence, for each $n$, there exists $\g_n \in \G$ and $i$ such that 
    \[ g_n \g_n u (g_n\g_n)^{-1} \xrightarrow{n\r\infty} e, \]
    for all $u\in U_i$.
    In addition, $\g_n$ and $i$ are determined uniquely for all $n$ sufficiently large.
    
    Given any faithful irreducible normed representation $V$ of $G$, for each $i$, we fix a non-zero vector $v_i$ which is fixed by $U_i$.
    By the Iwasawa decomposition, for any $i$ and any sequence $g_n$ in $G$, one has that $g_n v_i \r 0$ if and only if $g_n u g_n^{-1} \r e$ for all $u \in U_i$.
    Moreover, the $\G$ orbit of the identity coset in $G/U_i$ is discrete.
    In particular, the orbit $\G\cdot v_i$ is discrete (and hence closed) for each $i$.

    Thus, the function $\tilde{\a}: G/\G \r \R_+$ defined by
    \begin{equation} \label{defn: height function}
    \tilde{\a}(g\G) := \max_{w \in \bigcup_{i=1}^p g\G\cdot v_i } \norm{w}^{-1}
    \end{equation}
    is proper.
    The following Lemma provides us with other properties of the function $\tilde{\a}$.
    \begin{lemma}
    \label{lemma: height function properties}
    Suppose $\tilde{\a}$ is as in~\eqref{defn: height function}. Then,
    	\begin{enumerate}
    	\item Given a bounded neighborhood $\mc{O}$ of identity in $G$, there exists a constant $C_\mc{O}>1$, such that for all $g\in \mc{O}$ and all $x\in X$,
        	\[ C_\mc{O}^{-1} \tilde{\a}(x) \leq \tilde{\a}(gx) \leq C_\mc{O} \tilde{\a}(x). \]
        \item For all $M>0$, the set $\overline{ \tilde{\a}^{-1}([0,M])}$ is compact.
        \item{(cf.~\cite{GarlandRaghunathan})} There exists a constant $\e_1 >0$ such that for all $x = g\G \in X$, there exists at most one vector $v \in \bigcup_i g\G \cdot v_i$ satisfying $\norm{v} \leq \e_1$.
    	\end{enumerate}
    \end{lemma}

     \subsection{Rank One and Linear Expansion}

	We retain the same notation as in the previous section.
    Suppose $g_t$ is a one-parameter subgroup of $G$ which is $\mrm{Ad}$-diagonalizable over $\R$.
    Since $G$ has real rank equal to $1$, we can decompose the Lie algebra $\mf{g}$ of $G$ into eigenspaces for the adjoint action of $g_t$ as follows
    \begin{equation} \label{eqn: decomposition of Lie algebra}
    	\mf{g} = \mf{g}_{-2\a} \oplus \mf{g}_{-\a} \oplus \mf{g}_0 \oplus \mf{g}_\a \oplus \mf{g}_{2\a}.
    \end{equation}
	Then, we can find $H_0 \in \mf{g}_0$ so that
    \begin{equation}
    	g_t = \exp(tH_0),
    \end{equation}
    for all $t>0$.

    The following lemma is the form in which we use Proposition~\ref{propn: expansion in linear representations}.
    The key point of the lemma is that vectors expand at a maximal rate.
    \begin{lemma} \label{lemma: expansion of vectors in rank 1}
    Suppose $V$ is an irreducible real representation of $G$ with highest weight $\l$ and $\mu \in \set{\a,2\a}$ is such that $\mf{g}_\mu \neq 0$. Let $\d_\l = 2\l(H_0)/\mu(H_0)$ and suppose $0\neq v \in V$ is a highest weight vector.
    Then, for all $\b \in (0,1)$, there exists $\tilde{c} >0$ such that for all $g\in G$, $Z\in \mf{g}_\mu\backslash\set{0}$, and all $t>0$, the following holds
    	\begin{equation*}
    		\frac{1}{2} \int_{-1}^1 \norm{g_t u_s gv }^{-\b/\d_\l} \;ds \leqslant \tilde{c} e^{-\b\mu(H_0)/2} \norm{gv}^{-\b/\d_\l},
     	\end{equation*}
        where $u_s = \exp(sZ)$.
    \end{lemma}
    
    \begin{proof}
    	Let $v\in V$ be a highest weight vector.
    	Suppose $\mu$ and $0\neq Z\in \mf{g}_\mu$ are given and let $u_s = \exp(sZ)$.
        Since $u_s$ is normalized by $g_t$ and $G$ has rank $1$, the Jacobson-Morozov theorem implies that we can find $Z^{-}\in \mf{g}_{-\mu}$ so that $[Z,Z^-] = H_0$.
        In particular, the sub-algebra $\mf{h}$ generated by $Z$ and $Z^{-}$ is isomorphic to $\mf{sl}_2(\R)$.
        Denote by $H$ the corresponding subgroup of $G$.
        
        Note that since $H_0\in \mf{h}$, $\l$ can be regarded as a weight for $H$ in its induced representation on $V$.
        In particular, $V$ decomposes as a direct sum
        \[  V = V_\l \oplus V_0, \]
        where $V_\l$ is a direct sum of irreducible representations of $H$ with highest weight $\l$ and $V_0$ is an $H$-invariant complement.
        Hence, $v\in V_\l$.
        Denote by $\pi_\l: V \r V_\l$ the $H$-equivariant projection.
        
        Note that $\norm{g_t u_s gv} \geqslant \norm{g_tu_s\pi_\l(gv)}$ for all $t$ and $s$.
        Hence, by Proposition~\ref{propn: expansion in linear representations}, we get
        \begin{equation}\label{eqn: bound using projection}
        \frac{1}{2} \int_{-1}^1 \norm{g_t u_s gv }^{-\b/\d_\l} \;ds 
        \leqslant \frac{1}{2} \int_{-1}^1 \norm{g_t u_s \pi_\l(gv) }^{-\b/\d_\l} \;ds
        \leqslant c e^{-\b\mu(H_0)/2} \norm{\pi_\l(gv)}^{-\b/\d_\l},
        \end{equation}
        for some constant $c\geqslant 1$.
        
        For a weight $\mu$, denote by $V^\mu$ the corresponding weight space.
        Since $G$ has rank $1$, its Weyl group contains one non-trivial element sending $\l$ to $-\l$.
        Thus, by Proposition~\ref{propn: avoidance}, since $V^{-\l}\oplus V^{\l} \subseteq V_\l$, we get that
        \begin{equation} \label{eqn: G orbit avoids lower order space}
        	G\cdot v \cap V_0 = \emptyset.
        \end{equation}
        
        Since the stabilizer of the line $\R\cdot v$ is a parabolic subgroup $P$ and $G = KP$ for a compact group $K$, it follows from~\eqref{eqn: G orbit avoids lower order space} that $G\cdot v$ projects to a compact subset of the projective space $P(V)$ which is disjoint from the closed image of $V_0$ in $P(V)$.
        In particular, there exists $\e'>0$ such that for all $g\in G$,
        \begin{equation*}
        	\norm{\pi_\l(gv)} \geqslant \e' \norm{gv}.
        \end{equation*}
        Combining this estimate with~\eqref{eqn: bound using projection}, we obtain the desired conclusion with $\tilde{c} = c \e_1^{-\b/\d_\l}$.

    \end{proof}

    \subsection{The Main Integral Estimate}
    
	The height function $\tilde{\a}$ constructed in the previous sections satisfies the following integral estimate. 
	\begin{proposition} \label{propn: main integral estimate}
      Suppose $\l$ is the highest weight for $G$ in $V$ and $\mu \in \set{\a,2\a}$ is such that $\mf{g}_\mu \neq 0$.
      Define the following exponent
      \begin{equation*}
          \d_\l = 2\l(H_0)/\mu(H_0).
      \end{equation*}
		Then, for every $\b \in (0,1)$, there exists $\tilde{c} \geq 1$ such that the following holds:
        for all $t >0$, there exists $b = b(t) >0$ such that for all $x \in X$ and all $ Z \in \mf{g}_\mu$ with $\norm{Z}=1$,
        \begin{equation*}
        	\frac{1}{2} \int_{-1}^1 \tilde{\a}^{\b/\d_\l}(g_t \exp(rZ) x) \;dr 
            		\leqslant \tilde{c} e^{-\b \mu(H_0) t/2} \tilde{\a}^{\b/\d_\l}(x) + b.
        \end{equation*}
	\end{proposition}
    
    \begin{proof}
    	Let $t>0$ be fixed and define
        \[ \omega := \sup_{\substack{ r\in [-1,1]\\ Z\in\mf{g}_\l, \norm{Z}=1  }} \max\set{\norm{g_t \exp(rZ)}, \norm{(g_t \exp(rZ))^{-1}} }. \]
    	Now, fix some $Z \in \mf{g}_\l $ with $\norm{Z} =1$.
        For simplicity, we use the following notation
        \[ u_r := \exp(rZ). \]
        Then, for all $r \in [-1,1]$ and all $x\in X$, we have
        \begin{equation} \label{eqn: Lipschitz property}
        	\omega^{-1} \tilde{\a}(x)\leqslant \tilde{\a}(g_t u_r x) 
            	\leqslant \omega \tilde{\a}(x),
        \end{equation}
        where $\norm{\cdot}$ denotes the operator norm of the action of $G$ on $V$.
        Let $\e_1$ be as in $(3)$ of Lemma~\ref{lemma: height function properties}.
        Suppose $x\in X$ is such that $\tilde{\a}(x) \leq \omega/\e_1$.
        Then, by~\eqref{eqn: Lipschitz property}, for any $\b>0$, we have that
        \begin{equation} \label{eqn: bounded case}
        \frac{1}{2} \int_{-1}^1 \tilde{\a}^{\b/\d_\l}(g_t u_r x) \;dr \leqslant 
        	(\omega^2 \e_1^{-1})^{\b/\d_\l}.
        \end{equation}
        
        Now, suppose $x\in X$ is such that $\tilde{\a}(x) \geq \omega/\e_1$ and write $x=g\G$ for some $g\in G$.
        Then, by $(3)$ of Lemma~\ref{lemma: height function properties}, there exists a unique vector $v_0 \in \bigcup_i g\G \cdot v_i$ satisfying $\tilde{\a}(x) = \norm{v_0}^{-1}$.
        Moreover, by~\eqref{eqn: Lipschitz property}, we have that $\tilde{\a}(g_t u_r x) \geq 1/\e_1$ for all $r \in [-1,1]$.
        And, by definition of $\omega$, for all $r\in [-1,1]$, $\norm{g_tu_r v_0} \leq \e_1$.
        Thus, applying $(3)$ of Lemma~\ref{lemma: height function properties} once more, we see that $g_t u_r v_0$ is the unique vector in $\bigcup_i g_tu_r g\G \cdot v_i $ satisfying 
        \begin{equation*}
        	\tilde{\a}(g_tu_r x) = \norm{g_tu_rv_0}^{-1},
        \end{equation*}
        for all $r\in [-1,1]$.
        Moreover, since all the (minimal) parabolic subgroups of $G$ are conjugate, we see that the vectors $v_i$ all belong to the $G$-orbit of a highest weight vector $\tilde{v}$.
        
        Thus, we may apply Lemma~\ref{lemma: expansion of vectors in rank 1} as follows.
        Fix some $\b\in (0,1)$ and let $\tilde{c} \geqslant 1$ be the constant in the conclusion of the lemma.
        \begin{equation*}
        	\frac{1}{2} \int_{-1}^{1} \tilde{\a}^{\b/\d_\l}(g_t u_r x) \;dr
            	= \frac{1}{2} \int_{-1}^{1} \norm{g_tu_rv_0}^{-\b/\d_\l} \;dr
                \leqslant \tilde{c} e^{\frac{-\b \mu(H_0) t}{2}} \norm{v_0}^{-\b/\d_\l} 
                = \tilde{c} e^{\frac{-\b \mu(H_0) t}{2}} \tilde{\a}^{\b/\d_\l}( x).
        \end{equation*}
        Combining this estimate with~\eqref{eqn: bounded case}, we obtain the desired estimate.

    \end{proof}
    
    In order to obtain the winning property for bounded orbits, we need to show that the height function $\tilde{\a}$ satisfies Assumption~\ref{assumption: bounded connected components}.
    This is the content of the following lemma.
    Its proof is a combination of $(3)$ of Lemma~\ref{lemma: height function properties} and the fact that polynomial maps have finitely many zeros.
    \begin{lemma}\label{lemma: bounded conn. comp. rank 1}
    There exists $N\in \N$, depending only on the dimension of $G$, such that for every $T, R>0$, there exists $M_0>0$ such that for all $x\in G/\G$, $Y\in \mf{g}_\a\oplus \mf{g}_{2\a}$ with $\norm{Y}\leq R$ and $M\geq M_0$, the following holds.
    \begin{equation}\label{eqn: cusp set rank 1}
    	\text{The set } \set{|s|\leqslant T: \tilde{\a}(u(sY) x) > M} \text{ has at most } N \text{ connected components}.
    \end{equation}
    \end{lemma}
    
    \begin{proof}
    Let $T,R >0$, $Y\in \mf{g}_\a\oplus \mf{g}_{2\a}$ with $\norm{Y}\leq R$ and let $u_s = u(sY)$.
    Fix some $x =g\G \in X$. Let $\e_1 >0$ be the constant in $(3)$ of Lemma~\ref{lemma: height function properties}.
    Define $M_0$ as follows.
    \begin{equation*}
    	M_0 = \e_1^{-1} \sup \set{ \norm{ u(s Z)}:  Z\in \mf{g}_\a\oplus \mf{g}_{2\a}, \norm{Z}\leq R , |s| \leq T  }.
    \end{equation*}
    Let $M\geq M_0$.
    If $\tilde{\a}(u_s x) \leqslant M$ for all $|s|\leq T$, then the set in~\eqref{eqn: cusp set rank 1} is empty and the claim follows.
    On the other hand, if $\tilde{\a}(u_{s_0} x) > M$ for some $|s_0|\leq T$, then, by definition of $M$, we see that $\tilde{\a}(u_s x) > \e_1^{-1}$ for all $|s| \leq T$.
    In particular, by $(3)$ of Lemma~\ref{lemma: height function properties}, there exists a unique vector $w \in \bigcup_i g\G \cdot v_i$ such that
    \[ \tilde{\a}(u_s x) = \norm{u_s w}^{-1}, \text{ for all } |s|\leq T. \]
    Note that for any vector $w\in V$, since $u_s$ is a unipotent transformation, the map $s\mapsto \norm{u_s w}^2$ is a polynomial of degree at most $N$, where $N$ depends only on the dimension of $V$.
    Thus, since polynomials have finitely many zeros, for any $\epsilon >0$, the set $\set{|s|\leq T: \norm{u_s w} <\epsilon}$ has a number of connected components uniformly bounded above only in terms of $N$.
    This concludes the proof.  
    
    \end{proof}
    
    Given a $g_t$-admissible curve $\vp$ (Def.~\ref{defn: admissible curves}), applying Proposition~\ref{propn: main integral estimate} to the derivative $\dot{\vp}$ yields the following.
    \begin{theorem} \label{thrm: CH in rank 1}
    Suppose $\vp$ is a non-constant $g_t$-admissible curve. Then, $\vp$ satisfies the $\b$-contraction hypothesis (Def.~\ref{defn: height functions}) for all $\b \in (0,1/2)$ with a height function satisfying Assumption~\ref{assumption: bounded connected components}.
    \end{theorem}

    \section{Height Functions and Reduction Theory}
\label{section: height fun higher rank}

	The purpose of this section is to construct a height function on arithmetic homogeneous spaces and establish its main properties.
    This construction will be used in Section~\ref{section: CH higher rank} to verify the $\b$-contraction hypothesis in the setting of Theorem~\ref{thrm: products of sl2}.
    The height function we use here was introduced in~\cite{EskinMargulis-RandomWalks}.
    It generalizes the construction for $\mrm{SL}(n,\R)/\mrm{SL}(n,\Z)$ introduced in~\cite{EskinMargulisMozes} and builds on ideas which were used for the problem of quantitative recurrence of unipotent flows in~\cite{DaniMargulis}.
    However, we follow the approach of~\cite{KleinbockWeiss-SchmidtExpanding} which replaces the method of systems of integral inequalities with the notion of $W$-active Lie algebras.

    Throughout this section, we assume $G$ is a semisimple algebraic Lie group defined over $\Q$ with Lie algebra $\mf{g}$ such that the real rank of $G$ is at least $2$.
    We fix a lattice $\G\subset G(\Q)$.
    In particular, the rational structure on $\mf{g}$ is $\mrm{Ad}(\G)$-invariant.
    We let $\mf{g}_\Z$ denote an integer lattice of $\mf{g}$ with respect to this $\Q$-structure.

    Suppose $\mbf{S}$ is a maximal $\Q$-split torus in $G$. We identify $\mbf{S}$ with its Lie algebra and denote by $\mbf{S}^\ast$ its linear dual.
    Let $\mc{C} \subseteq \mbf{S}$ be a closed Weyl chamber and fix an order on the roots of $\mbf{S}$ making $\mc{C}$ positive.
    Denote by $\Pi =\set{ \a_1,\dots, \a_r} \in \mbf{S}^\ast$ a set of simple positive roots.
    We assume that $G/\G$ is not compact. In particular, $r = \mrm{rank}_\Q G \geq 1$.
    Let $\Delta^+$ denote the set of positive $\Q$-roots.
    For each root $\b$, denote by $\mf{g}_\b$ the corresponding root space.
    The reader is referred to~\cite[Section 14]{Borel-LinearAlgebraicGroups} for standard facts regarding root systems over $\Q$.
    
    For each $1\leqslant k \leqslant r$, let $\mrm{P}_k$ be the maximal standard parabolic $\Q$-subgroup obtained from $\Pi \backslash\set{\a_k}$.
    Then, each $\mrm{P}_k$ is defined over $\Q$.
    We note that every maximal parabolic $\Q$-subgroup of $G$ is conjugate over $G(\Q)$ to $\mrm{P}_k$ for some $k$.
    
    We fix a maximal compact subgroup $K$ of $G$ which is fixed by a Cartan involution leaving $\mbf{S}$ invariant.
    In particular, $G=K\mrm{P}_k$ for all $k$.

    \subsection{Siegel Sets and Reduction Theory}
    A subset $\Omega \subset G$ is said to be a \emph{fundamental set} for $\G$ if the following hold.
    \begin{enumerate}
    	\item $G = \Omega\G$, and
        \item The set of elements $\g \in \G$ such that $\Omega\g \cap \Omega \neq \emptyset$ is finite.
    \end{enumerate}

    Let $\mrm{P}_0 =\cap_k \mrm{P}_k $ be the standard minimal parabolic subgroup associated with $\mbf{S}$ and $\Pi$ and let $\mrm{U}_0$ be the unipotent radical of $\mrm{P}_0$.    
    Denote by $\mrm{M}_0 \subset \mrm{P}_0$ the identity component of the maximal $\Q$-anisotropic subgroup of $Z_G(\mbf{S})$.
    
    A Siegel set $\mf{S}$ (relative to $K, \mrm{P}_0$ and $\mbf{S}$)  is a set of the form $\mf{S} = K \mbf{S}_t  W$, where $W$ is a compact subset of $\mrm{M}_0\mrm{U}_0$, $t\geq 0$, and
    \begin{equation} \label{Siegel set}
    	\mbf{S}_t = \set{s \in \mbf{S}:  \a_k(s) \leq t, k=1,\dots,r  }.
    \end{equation}
    The following classical result, due to Borel and Harish-Chandra, shows that Siegel sets give rise to fundamental sets for $\G$.
    \begin{proposition}[Theorem 15.5, Proposition 15.6,\cite{Borel-FrenchBook}]
    \label{propn: siegel sets are fundamental}
    The space $\G\backslash G(\Q) / \mrm{P}_0(\Q)$ has finitely many double cosets.
    Given a finite set of representatives $F \subset G(\Q)$, there exists a Siegel set $\mf{S}$ such that $\Omega = \mf{S}F$ is a fundamental set for $\G$.
    \end{proposition}
    Through the remainder of this section, we fix a Siegel set $\mf{S}$ and a finite set $F\subset G(\Q)$ as in Proposition~\ref{propn: siegel sets are fundamental}.
    We denote by $F^{-1}$ the set of inverses of the elements of $F$.

\subsection{The functions $\tilde{\a}_k$}
    Denote by $\mrm{U}_k$ the unipotent radical of $\mrm{P}_k$ and let $d_k = \mathrm{dim} \mrm{U}_k$.
    Then, each $\mrm{U}_k$ is defined over $\Q$.
    In particular, $\mrm{U}_k \G$ is closed in $G$ and $\mrm{U}_k/\mrm{U}_k\cap \G$ is compact.
    
    Let $\mf{u}_k$ be the Lie algebra of $\mrm{U}_k$ and let $u_{k_1},\dots , u_{k_{d_k}} \in \mf{u}_k\cap \mf{g}_\Z$ be an integral basis for $\mf{u}_k$.
    Define $\mathbf{p}_{\mf{u}_k}$ as follows
    \begin{equation} \label{p_u}
    	\mathbf{p}_{\mf{u}_k} = u_{i_1}\wedge \cdots \wedge u_{i_{d_k}}
        \in \bigwedge^{d_k} \mf{g}.
    \end{equation}
    Note that the stabilizer of the line $\mathrm{span}(\mf{p}_{\mf{u}_k})$ is $\mrm{P}_k$.
    For each $1\leq k\leq r$, consider the following vector space
    \begin{equation} \label{defn: representations V_k}
    V_k = \mathrm{span}\left( \bigwedge^{d_k}(\mrm{Ad}(G)) \mathbf{p}_{\mf{u}_k} \right).
    \end{equation}
    Then, the representation of $G$ on each $V_k$ is irreducible.
    Indeed, the vector $\mf{p}_{\mf{u}_k}$ is fixed by $\oplus_{\b \in \Delta^+} \mf{g}_\b$ and is thus a highest weight vector and so $V_k = \mrm{span}(\bigwedge^{d_k} \mrm{ad}(\mf{g})\cdot \mf{p}_{\mf{u}_k})$ is irreducible.
    Moreover, since $\mf{p}_{\mf{u}_k} \in \bigwedge^{d_k} \mf{g}_\Z$, $F$ is a finite subset of $G(\Q)$, and $\mrm{Ad}(\G)(\mf{g}_\Z) \subseteq \mf{g}_\Z $, we see that $\G F^{-1} \cdot \mf{p}_{\mf{u}_k}$ is discrete since it is contained in $\bigwedge^{d_k} \mf{g}_{\frac{1}{N}\Z}$, for some $N\in \N$ depending on $F$.

    We use $\norm{\cdot}$ to denote a $K$-invariant norm on $V_k$, where $K$ is our fixed maximal compact subgroup.
    Define $\tilde{\a}_k: G \r \R_+$ as follows
    \begin{equation} \label{defn: b_k}
    \tilde{\a}_k(g) = \max \set{ \norm{g \g f^{-1}  \cdot \mf{p}_{\mf{u}_k} }^{-1} : \g\in \G, f\in F}.
    \end{equation}
    Note that the functions $\tilde{\a}_k$ are $\G$-invariant and can be regarded as functions on $G/\G$.
    In particular, we define a function $f:G/\G \r \R^+$ as follows
    \begin{equation} \label{defn: height higher rank}
    	f(x) = \max_{1\leq k\leq r}  \tilde{\a}_k^{1/d_k}(g),
    \end{equation}
    for $x = g\G \in G/\G$.    
    The following proposition shows that $f$ encodes divergence in $G/\G$.
    \begin{proposition} \label{alpha_k is proper}
    A subset $L \subseteq G/\G$ is bounded if and only if
    \[ \max_{1\leq k\leq r} \sup_{l\in L} \tilde{\a}_k(l) <\infty. \]
    \end{proposition}
    \begin{proof}
    This result is well-known and is present in several places in the literature. See for example Steps $1$ and $2$ in the proof of Proposition 4.1 in~\cite{KleinbockWeiss-SchmidtExpanding}. We include a proof for completeness.
    The direction $``\Rightarrow"$ follows from the discreteness of the sets $\G F^{-1} \cdot \mf{p}_{\mf{u}_k}$.
    Conversely, suppose $x_n $ is an unbounded sequence in $G/\G$ and let $g_n \in \mf{S} F$ be a representative of $x_n$ in the fundamental set for $\G$.
    Hence, we can write
    \[g_n = k_n w_n s_n  f_n, \]
    with $k_n\in K$, $w_n\in W$, $s_n\in \mbf{S}_t$ and $f_n\in F$ such that, possibly after passing to a subsequence, there is some $1\leq j \leq r$ satisfying
    \[ \a_j(s_n) \r -\infty. \]
    By the $K$-invariance of $\norm{\cdot}$ and compactness of $W$, we get that
    \[ \tilde{\a}_j(x_n) \geqslant \norm{k_nw_n s_nf_n f_n^{-1} \cdot\mf{p}_{\mf{u}_j}  }^{-1} \gg \norm{s_n \cdot\mf{p}_{\mf{u}_j}  }^{-1}.\]
    Now, observe that $\mf{p}_{\mf{u}_j} $ is a weight vector for $\mbf{S}$ with weight $\chi_j$of the form
    \[  \chi_j = \sum n_\b \b, \]
    where the sum is taken over all positive roots $\b$ which have $\a_j$ in their expansion in terms of simple roots and $n_\b$ denotes the dimension of the root space corresponding to $\b$.
    Finally, note that since $s_n\in \mbf{S}_t$, the values $\b(s_n)$ are bounded above for all positive roots $\b$.
    In particular, $s_n \cdot\mf{p}_{\mf{u}_j} = e^{\chi_j(s_n)}\mf{p}_{\mf{u}_j}$ and 
    $\chi_j(s_n) \r -\infty$
    which concludes the proof.
    \end{proof}

    \subsection{W-Active Lie Algebras and The Contraction Hypothesis in $G/\G$}

    We recall here several facts concerning unipotent radicals of parabolic subgroups which will be useful for us.
    The first is the following observation due to Tomanov and Weiss~\cite{TomanovWeiss}.
    \begin{lemma}[Proposition 3.3 in~\cite{TomanovWeiss}]
    \label{lemma: short integral vectors are unipotent}
    There exists a compact neighborhood $W$ of $0$ in $\mf{g}$ such that for all $g\in G$, the Lie algebra generated by $\mrm{Ad}(g)(\mf{g}_\Z)\cap W$ is unipotent.
    \end{lemma}

    Next, we record the following classical facts regarding intersections of parabolic groups.
    
    \begin{lemma} \label{lem: intersections of parabolics}
    Suppose that $P_0$ is a minimal parabolic subgroup and $P$ is a parabolic subgroup of $G$. Then, the following hold.
    \begin{enumerate} [(i)]
      \item If $P_0$ contains the unipotent radical of $P$, then $P_0\subseteq P$.
      \item{\cite[Proposition 14.22(iii)]{Borel-LinearAlgebraicGroups}} If $Q$ is conjugate to $P$ and $Q$ contains the unipotent radical of $P$, then $Q=P$.
    \end{enumerate}
    \end{lemma}
    \begin{proof}
    	Item $(i)$ in fact follows from $(ii)$. Let $P_0' \subset P$ be a minimal parabolic subgroup (over $\R$) containing the unipotent radical of $P$. Then, since all minimal parabolic subgroups are conjugate~\cite[Proposition 21.12]{Borel-LinearAlgebraicGroups}, there exists $g\in G$ such that $gP_0'g^{-1} = P_0 \subseteq gPg^{-1}$. In particular, we can apply $(ii)$ with $Q=gPg^{-1}$ to get that $P = Q$ and hence the claim follows.
    \end{proof}

    Following~\cite{KleinbockWeiss-SchmidtExpanding}, we make the following key definition.
    \begin{definition} 
    Given a neighborhood $W \subset \mf{g}$ of $0$ and $g\in G$, we say a Lie sub-algebra $\mf{u}$ is $\mbf{W}$-\textbf{active} for $g$ if 
    \begin{equation}\label{defn: W active}
    	\mrm{Ad}(g)(\mf{u}) \subseteq \mrm{span} \left( \mrm{Ad}(g)(\mf{g}_\Z) \cap W \right) .
    \end{equation} 
    
    \end{definition}

    The following is a key result obtained in~\cite{KleinbockWeiss-SchmidtExpanding}.
    \begin{proposition} [Proposition 4.1 in~\cite{KleinbockWeiss-SchmidtExpanding}]
    \label{propn: shortest vectors are active}
    For every compact neighborhood $W$ of $0$ in $\mf{g}$ and every $\w >0$, there exists $M>0$ such that for all $x=g\G\in G/\G$ with $f(x) > M$ and all $k$, the set
    \[ \Psi_k(g) = \set{ v \in \G F^{-1} \cdot \mf{p}_{\mf{u}_k} : \norm{g\cdot v}^{-1/d_k} \geqslant f(x)/\w } \]
    consists of $W$-active elements for $g$.
    \end{proposition}
    
    The above facts will be used in the form of the following corollary.
    \begin{corollary} \label{cor: shortest are collinear}
    Suppose $W$ is a compact neighborhood of $0$ in $\mf{g}$ as in the conclusion of Lemma~\ref{lemma: short integral vectors are unipotent}. Then, for every $\w >0$, there exists $M>0$ such that for all $x=g\G\in G/\G$ with $f(x) > M$ and all $k$, the span of the set
    \[ \Psi_k(g) = \set{ v \in \G F^{-1} \cdot \mf{p}_{\mf{u}_k} : \norm{g\cdot v}^{-1/d_k} \geqslant f(x)/\w } \]
    has dimension at most $1$.
    \end{corollary}
    
    \begin{proof}
    
    By Proposition~\ref{propn: shortest vectors are active}, let $M$ be chosen so that for each $k$, the set $\Psi_k(g)$ consists of $W$-active elements.
    For $i=1,2$, let $v_i = \g_i f_i^{-1}\cdot\mf{p}_{\mf{u}_k} \in \Psi_k(g)$.
    
    By Lemma~\ref{lemma: short integral vectors are unipotent}, the Lie algebra generated by $\mrm{Ad}(g \g_1 f_1^{-1}) \mf{u}_k$ and $\mrm{Ad}(g \g_2 f_2^{-1}) \mf{u}_k$ is unipotent.
    In particular, both $\mrm{Ad}(g \g_1 f_1^{-1}) (\mrm{U}_k)$ and $\mrm{Ad}(g \g_2 f_2^{-1}) (\mrm{U}_k)$ are contained in the same minimal parabolic subgroup $P_0$.
    By $(i)$ of Lemma~\ref{lem: intersections of parabolics}, for $i=1,2$, since $\mrm{Ad}(g \g_i f_i^{-1}) (\mrm{U}_k)$ is the unipotent radical of $\mrm{Ad}(g \g_i f_i^{-1})(\mrm{P}_k)$, it follows that $P_0 \subseteq \mrm{Ad}(g \g_i f_i^{-1})(\mrm{P}_k)$.
    
    In particular, $\mrm{Ad}(g \g_2 f_2^{-1})(\mrm{P}_k)$ contains the unipotent radical of $\mrm{Ad}(g \g_1 f_1^{-1})(\mrm{P}_k)$.
    By $(ii)$ of Lemma~\ref{lem: intersections of parabolics}, $\mrm{Ad}(g \g_2 f_2^{-1})(\mrm{P}_k) = \mrm{Ad}(g \g_1 f_1^{-1})(\mrm{P}_k)$. Hence, since $\mrm{P}_k$ is its own normalizer, we get 
    \[ f_2 \g_2^{-1}   \g_1 f_1^{-1}\in \mrm{P}_k. \]
    In particular, since $\mrm{P}_k$ normalizes $\mf{u}_k$,
    \[  v_1 = c v_2, \]
    for some $c \neq 0$ (in fact, $c\in \frac{1}{N}\Z$ for some $N\in \N$ depending on $F$).
    \end{proof}

    
    \section{The Contraction Hypothesis in Arithmetic Homogeneous Spaces}
\label{section: CH higher rank}

	In this section, we establish the contraction hypothesis for certain curves on arithmetic homogeneous spaces using the height function constructed in the previous section.
    The main result is Theorem~\ref{thrm: CH for maximal representations}.
    Combined with the results in Sections~\ref{section: abstract setup},~\ref{section: schmidt games}, and~\ref{section: shrinking nondivergence}, we obtain, for a wide class of curves on arithmetic homogeneous spaces, an explicit bound on the dimension of divergent on average orbits, thickness of the set of bounded orbits, and quantitative non-divergence of translates of shrinking curve segments.
    We retain the same notation as in the previous section.

\subsection{Deformations of Maximal Representations and Linear Expansion}

    We introduce the notion of deformations of a maximal representation of an $\mf{sl}_2$-triple to abstract the exact properties we require from the derivative of our curves which imply that they satisfy the $\b$-contraction hypothesis.

    \begin{definition} \label{defn: deformations}
    Given a bounded interval $B\subset \R$ and an $\mf{sl}_2$-triple $(Y,h,X)$, we say a map $\rho: \mf{sl}(2,\R) \times B \r \mf{g}$ is a \textbf{deformation of a maximal representation} if the following conditions hold.
    \begin{enumerate}
    	\item $\rho$ is continuous and for each $s\in B$, $\rho_s := \rho |_{\mf{sl}(2,\R) \times \set{s}}$ is a faithful Lie algebra homomorphism. In particular, $\rho_s(X)\neq 0$ for all $s\in B$.
        \item $H_\rho := \rho_s(h) $ belongs to (closure of) the positive Weyl chamber $\mc{C}\subset \mbf{S}$ and is independent of $s$.
        \item For each $s\in B$, $\rho_s(X) \in \bigoplus_{\b\in \Delta^+}\mf{g}_\b$ and $\rho_s(\mf{sl}(2,\R))$ is not contained in any conjugate of a proper parabolic $\Q$-subalgebra of $\mf{g}$.
    \end{enumerate}
    
    \end{definition}

    In the examples we study, the curves $\vp$ satisfy $\dot{\vp}(s) = \rho_s(X)$ for some such $\rho$.
    In the remainder of this section, we fix $\rho: \mf{sl}(2,\R) \times B \r \mf{g}$ a deformation of a maximal representation.

    The simple roots $\Pi$ induce a partial order on $\mbf{S}^\ast$ in the following natural way.
    \begin{equation*}
    	\mu \leqslant \nu \Leftrightarrow \nu -\mu = \sum_{\Delta^+}k_\a \a \text{ for some }  k_\a \in \N\cup\set{0}.
    \end{equation*}
    In particular, given any irreducible representation $V$ of $G$, defined over $\Q$, the set of $\Q$-weights of $\mbf{S}$ admits a maximal element which we call the highest $\Q$-weight.

\subsection{Maximal Representations and Linear Expansion}
    The following lemma is a direct analogue of Lemma~\ref{lemma: expansion of vectors in rank 1} in the setting of Lie groups of real rank equal to $1$.

    \begin{lemma} \label{lemma: expansion of vectors in higher rank}
    Suppose $V$ is an irreducible representation of $G$ defined over $\Q$ with highest $\Q$-weight $\l$.
    Let $\d_\l = \l(H_\rho)$ and suppose $0\neq v \in V(\Q)$ is a highest weight vector.
    Then, for all $\b \in (0,1)$, there exists $\tilde{c} >0$ and $0<\b'\leqslant \b$ such that for all $g\in G$, $s\in B$, and all $t>0$, 
    	\begin{equation*}
    		\frac{1}{2} \int_{-1}^1 \norm{g_t u(r\rho_s(X)) gv }^{-\b/\d_\l} \;dr \leqslant \tilde{c} e^{-\b' t} \norm{gv}^{-\b/\d_\l},
     	\end{equation*}
        where $g_t = \exp(tH_\rho)$ and $u(r\rho_s(X)) = \exp(r\rho_s(X)) $. Moreover, if $\mrm{rank}_\Q G = 1$, we can take $\b'=\b$. 
    \end{lemma}
    
    \begin{proof}

    The proof of Lemma~\ref{lemma: expansion of vectors in higher rank} in the case $\mrm{rank}_\Q G = 1$ is identical to that of Lemma~\ref{lemma: expansion of vectors in rank 1}.
    Indeed, the key ingredients in the proof of Lemma~\ref{lemma: expansion of vectors in rank 1} are Proposition~\ref{propn: avoidance} and the fact that the only non-trivial Weyl group element sends the highest weight $\l$ to $-\l$.

    In the higher rank case, fix some $s\in B$ and let $\mf{h} =\left({\mf{sl}(2,\R) \times \set{s}}\right)$.
    Then, we can decompose $V = V_1 \oplus V^\mf{h}$, where $\mf{h}$ acts trivially on $V^\mf{h}$ and $V_1$ is the $\mf{h}$-invariant complement of $V^\mf{h}$ and contains no trivial sub-representations.
    Let $\pi_1$ denote the $\mf{h}$-equivariant projection onto $V_1$.
    
    Note that the stabilizer of $\R\cdot v$ is a parabolic $\Q$-subgroup of $G$.
    Thus, since $\rho$ is a maximal representation, we have that
    \[  G\cdot v \cap V^\mf{h} = \emptyset. \]
    In particular, arguing as in the proof of Lemma~\ref{lemma: expansion of vectors in rank 1}, this implies that there exists some $\e>0$ such that for all $g\in G$,
    \begin{equation}\label{eqn: projectively isolated}
    \norm{\pi_1(gv)} \geqslant \e \norm{gv}. 
    \end{equation}  
    Since $\rho$ is continuous in $s$ and $B$ is compact, we note further that $\e$ may be chosen uniformly over $s\in B$ since the spaces $V^{\mf{h}}$ vary continuously.

    Denote by $P^+(V)$ the set of highest weights for $\mf{h}$ appearing in the decomposition of $V$ into irreducible representations.
    For each $\mu\in P^+(V)$, we let $\d_\mu = \mu(H_\rho)$ and $V_\mu$ be the direct sum of irreducible sub-representations of $V$ with highest weight $\mu$.
    Let $\pi_\mu: V \r V_\mu$ denote the associated projection.
    Inequality~\eqref{eqn: projectively isolated} implies that there exists $\mu \in P^+(V)\backslash \set{0}$ such that
    \begin{equation} \label{eqn: one weight matters}
    \norm{\pi_\mu(gv) } \asymp  \norm{\pi_1(gv)} \geqslant \e \norm{gv}.   
    \end{equation}

    Define $\b'$ as follows
    \begin{equation} \label{eqn: exponent in higher rank}
    	\b' = \frac{\b}{\d_\l} \min\set{ \d_\mu : \mu \in P^+(V)\backslash \set{0} }.
    \end{equation}    
    The Lemma now follows immediately from Proposition~\ref{propn: expansion in linear representations} applied to the projection of $gv$ onto $V_\mu$ with $0\neq \mu \in P^+(V)$ and satisfying~\eqref{eqn: one weight matters}.

    \end{proof}

    \begin{remark}
    It is natural to ask whether the constant of proportionality between $\b'$ and $\b$ in~\eqref{eqn: exponent in higher rank} is optimal.
    When $\mrm{rank}_\Q(G) >1$, the Weyl group typically contains more than one non-trivial element.
    This fact played a key role in the (real and rational) rank $1$ cases in showing that $\b'=\b$, allowing us to obtain the fastest possible contraction rate.
    In particular, it is not clear whether it is possible to modify the argument in the proof of Lemma~\ref{lemma: expansion of vectors in higher rank} to show that in the case $\mrm{rank}_\Q(G) >1$, the $G$-orbit of a highest weight vector avoids sub-representations of $V$ with non-maximal, non-zero highest weights for $\mf{h}$.
    More precisely, it is not clear whether the analogue of equation~\eqref{eqn: G orbit avoids lower order space} holds in the setting of Lemma~\ref{lemma: expansion of vectors in higher rank} when $\mrm{rank}_\Q(G) >1$.
    
    \end{remark}

 \subsection{The Main Integral Estimate}
 
  For each $1\leqslant k \leqslant r$, let $\chi_k \in \mbf{S}^\ast$ denote the highest $\Q$-weight with respect to $\Pi$ in the representation $V_k$ defined in~\eqref{defn: representations V_k}.
    Then, since $\rho_s$ is a maximal representation, for each $1\leq k\leq r$,
    \[ \d_k := \chi_k(H_\rho) \neq 0.\]
 	Indeed, otherwise, if $\d_k = 0$, this implies that $\mrm{Ad}\left(\rho\left(\mf{sl}(2,\R)\times\set{s}\right)\right)$ annihilates $\mf{p}_{\mf{u}_k}$. But, since $\mrm{P}_k$ is the normalizer of $\mrm{U}_k$, this implies that $\rho\left(\mf{sl}(2,\R)\times\set{s}\right)$ is contained in the Lie algebra of $\mrm{P}_k$, contradicting the fact that $\rho_s$ is maximal.
    
 The following proposition is the main result of this section.
 \begin{proposition} \label{propn: main integral estimate in higher rank}
    For all $0 < \b< \min_k  d_k/\d_k$, there exists $c_0 \geqslant 1$ and $0<\b'\leqslant \min_k \b \d_k/d_k$, depending on $\b$, so that the following holds.
    For every $t>0$, there exists a positive constant $b$ such that for all $x\in G/\G$ and all $s\in B$,
    \begin{equation*}
    	\frac{1}{2} \int_{-1}^1 f^\b (g_t u(r\rho_s(X)) x)\;dr
        \leqslant c_0 e^{-\b' t} f^\b(x) + b,
    \end{equation*}
    where $g_t = \exp(tH_\rho)$ and $u(r\rho_s(X)) = \exp(r\rho_s(X)) $. Moreover, if $\mrm{rank}_\Q G = 1$, we can take $\b'=\b\d_1/d_1$. 
    \end{proposition}
	
 	\begin{proof}

    Let $W$ be a compact neighborhood of $0$ for which Lemma~\ref{lemma: short integral vectors are unipotent} holds.
    Fix some $t>0$ and define $\w$ as follows.
    \begin{equation*}
    	\w = \sup_{|r| \leqslant 1, s\in B}  \max\set{ \norm{g_t u(r\rho_s(X))}, \norm{(g_t u(r\rho_s(X)))^{-1}} }.
    \end{equation*}
    Here $\norm{\cdot}$ refers to the operator norm for the $G$ action on $\bigoplus_k \bigwedge^k \mf{g}$.    
    Then, for all $s\in B$ and $r \in [-1,1]$ and all $x\in G/\G$, we have
    \begin{equation}\label{eqn: f is log lip}
    	\w^{-1/d_0} f(x) \leqslant f(g_t u(r\rho_s(X)) x) \leqslant \w^{1/d_0} f(x),
    \end{equation}
    where 
    \begin{equation} \label{defn: d_0}
    d_0 = \min_{1\leq k \leq r} d_k.  
    \end{equation} 
    Let $M>0$ be as in Corollary~\ref{cor: shortest are collinear} applied to our chosen $W$ and with $\w^{2/d_0}$ in place of $\w$.
    Suppose that $x_0 \in G/\G$ is such that $f(x) \leqslant M$.
    Fix $\b\in (0,1)$.
    Then, we have
    \begin{equation} \label{eqn: estimate for bounded x}
    	\frac{1}{2} \int_{-1}^1 f^\b(g_t u(r\rho_s(X)) x_0) \;dr \leqslant b,
    \end{equation}
    for $b = \w^{\b/d_0} M^\b$.
    
    Now, suppose $f(x_0) >M$ and write $x_0 = g\G$ for some $g\in G$. For each $1\leq k\leq r$, consider the sets
    \begin{equation*}
    \Psi_k(g) = \set{v \in  \G F^{-1} \cdot \mf{p}_{\mf{u}_k} : \norm{g\cdot v}^{-1/d_k} \geqslant f(x_0)/\w^{2/d_0} }. 
    \end{equation*}
    We claim that for all $s\in B$ and $r \in [-1,1]$, one has
    \begin{equation*}
    	f\left(g_t u(r\rho_s(X))x_0\right) =\max
        \set{ \norm{g_t u(r\rho_s(X))g\cdot v}^{-1/d_k} : v\in \Psi_k(g), 1\leq k\leq r}.
    \end{equation*}
    
    Indeed, suppose $v\in  \G F^{-1} \cdot \mf{p}_{\mf{u}_k}$ satisfies $f\left(g_t u(r\rho_s(X))x_0\right) = \norm{g_t u(r\rho_s(X))g\cdot v}^{-1/d_k}$ for some $k$ and some $(s,r) \in B\times [-1,1]$.
    Then, by definition of $\w$ and~\eqref{eqn: f is log lip}, we obtain
    \[
    	\w^{1/d_k} \norm{g\cdot v}^{-1/d_k}\geqslant \norm{g_t u(r\rho_s(X))g\cdot v}^{-1/d_k} = 
        f\left(g_t u(r\rho_s(X))x_0\right) \geqslant f(x_0) / \w^{1/d_0}.
    \]
    Hence, $v\in \Psi_k(g)$ as desired.
    Say a vector $v\in \G F^{-1} \cdot \mf{p}_{\mf{u}_k}$ is primitive if $v$ has minimal norm in $\R\cdot v \cap \G F^{-1} \cdot \mf{p}_{\mf{u}_k}$.
    Next, we note that Corollary~\ref{cor: shortest are collinear} implies that for each $k$, the set $\Psi_k(g)$ contains at most one primitive vector up to a sign.
    Denote by $\Psi_k^0(g)$ the following set.
    \[ \Psi_k^0(g) =  \set{ v\in \Psi_k(g): v \text{ is primitive}  }. \]

    In order to apply Lemma~\ref{lemma: expansion of vectors in higher rank}, let $\d_k = \chi_k(H_\rho)$ and $\g_k = \b \d_k/d_k$.
    The choice of $\b$ implies that $0<\g_k<1$ for all $1\leq k\leq r$.
    Thus, by Lemma~\ref{lemma: expansion of vectors in higher rank} applied with $\g_k$ in place of $\b$, there exists $0<\b' \leqslant \min_k \g_k$ such that the following inequalities hold:
    \begin{align} \label{eqn: estimate for unbounded x}
    \frac{1}{2} \int_{-1}^1 f^\b(g_t u(r\rho_s(X)) x_0) \;dr &\leqslant
    \sum_{ \substack{ v\in \Psi_k^0(g) \\ 1\leq k\leq r} }
    \frac{1}{2} \int_{-1}^1  \norm{g_t u(r\rho_s(X))g\cdot v}^{-\b/d_k} \nonumber \\
    &\leqslant \tilde{c} e^{-\b' t} \sum_{ \substack{ v\in \Psi_k^0(g) \\ 1\leq k\leq \mrm{rank}_\Q G} }
    \norm{g\cdot v}^{-\b/d_k} \leqslant 2r \tilde{c} e^{-\b' t} f^\b(x_0),
    \end{align}
    where $\tilde{c}$ is as in the conclusion of Lemma~\ref{lemma: expansion of vectors in higher rank}.
    Combining~\eqref{eqn: estimate for bounded x} and~\eqref{eqn: estimate for unbounded x} completes the proof.
    \end{proof}

    \begin{remark}
    An analogue of Proposition~\ref{propn: main integral estimate in higher rank} was obtained in~\cite[Section 3.2]{EskinMargulis-RandomWalks} in the context of random walks on homogeneous spaces.
    It was assumed in~\cite{EskinMargulis-RandomWalks} that the Zariski closure of the group generated by the support of the measure generating the random walk is a semisimple group which is not contained in any proper parabolic $\Q$-subgroup of $G$.
    This assumption is replaced here with the notion of a deformation of a maximal representation.
    Lemma~\ref{lemma: expansion of vectors in higher rank} acts as a substitute for the positivity of the top Lyapunov exponent in the context of random walks.
    In the case when the rational rank of $G$ is equal to $1$, we also observe that we can obtain a precise contraction rate which allows us to obtain a sharp dimension upper bound for divergent on average orbits.
    \end{remark}

    In the following lemma, we show that the height function $f$ satisfies Assumption~\ref{assumption: bounded connected components}.
    Its proof is very similar to the analogous Lemma~\ref{lemma: bounded conn. comp. rank 1} in rank $1$.
    \begin{lemma}\label{lemma: bounded conn. comp. higher rank}
    There exists $N\in \N$, depending only on $G$, such that for every $T, R>0$, there exists $M_0>0$ such that for all $x\in G/\G$, $Y\in \bigoplus_{\b\in \Delta^+}\mf{g}_\b$ with $\norm{Y}\leq R$ and $M_1\geq M_0$, the following holds.
    \begin{equation}\label{eqn: cusp set higher rank}
    	\text{The set } \set{|s|\leqslant T: f(u(sY) x) > M_1} \text{ has at most } N \text{ connected components}.
    \end{equation}
    \end{lemma}
    
    \begin{proof}
    Let $T,R >0$, $Y\in \bigoplus_{\b\in \Delta^+}\mf{g}_\b$ with $\norm{Y}\leq R$ and let $u_s = u(sY)$.
    Fix some $x =g\G \in X$ and define $\w$ as follows:
    \begin{equation*}
    	\w = \sup \set{ \norm{ u(s Z)}:  Z\in \bigoplus_{\b\in \Delta^+}\mf{g}_\b, \norm{Z}\leq R , |s| \leq T  },
    \end{equation*}
    where $\norm{\cdot}$ refers to the operator norm on $\bigoplus_k \bigwedge^k \mf{g}$.
    Arguing as in the proof of Proposition~\ref{propn: main integral estimate in higher rank}, let $W$ be a compact neighborhood of $0$ for which Lemma~\ref{lemma: short integral vectors are unipotent} holds.
    Let $M>0$ be as in Corollary~\ref{cor: shortest are collinear} applied to $W$ and $\w^{2/d_0}$, where $d_0$ is defined in~\eqref{defn: d_0}.
    Now, define $M_0$ as follows.
    \begin{equation*}
    	M_0 = \w^{2/d_0} M.
    \end{equation*}

    Let $M_1 \geq M_0$.
    If $f(u_s x) \leqslant M_1$ for all $|s|\leq T$, then the set in~\eqref{eqn: cusp set higher rank} is empty and the claim follows.
    On the other hand, if $f(u_{s_0} x) > M_1$ for some $s_0$ with $|s_0|\leq T$, then, by definition of $M_1$ and $\w$, we see that $f(u_s x) > M$ for all $|s| \leq T$.
    For each $1\leq k \leq r = \mrm{rank}_\Q(G)$, define the following sets.
    \begin{equation*}
    \Psi_k^0(g) = \set{v \in  \G F^{-1} \cdot \mf{p}_{\mf{u}_k} : \norm{g\cdot v}^{-1/d_k} \geqslant f(x)/\w^{2/d_0}, v \text{ is primitive} }. 
    \end{equation*}
    By an argument identical to that in the proof of Proposition~\ref{propn: main integral estimate in higher rank}, it follows that
    \begin{equation*}
    	f\left(u_s x\right) =\max
        \set{ \norm{ u_sg\cdot v}^{-1/d_k} : v\in \Psi_k^0(g), 1\leq k\leq r},
    \end{equation*}
    for all $|s| \leq T$ and the sets $\Psi_k^0(g)$ contain at most one vector up to a sign for each $k$.
    In particular, for each $|s|\leq T$, $f(u_sx)$ is a maximum over functions of the form $\norm{u_s w}^{-1/d_k}$ for at most $2r$ vectors $w$.

    Finally, for any vector $w\in V=\bigoplus_k \bigwedge^k \mf{g}$, the map $s\mapsto \norm{u_s w}^2$ is a polynomial of degree at most $d$, where $d$ depends only on the dimension of $V$.
    Thus, since polynomials have finitely many zeros, for any $\epsilon >0$, the set $\set{|s|\leq T: \norm{u_s w} <\epsilon}$ has a number of connected components uniformly bounded above only in terms of $d$.
    Moreover, each connected component of the set $\set{s: f(u_sx)>M_1}$ is a union of connected components of sets of the form $\set{s: \norm{u_sg\cdot v} <\epsilon }$ for an appropriate $\epsilon >0$.
    The claim now follows by taking $N = 2r d$.  
    
    \end{proof}
    
    Proposition~\ref{propn: main integral estimate in higher rank} establishes the main contraction property of the function $f$ while the other properties in Definition~\ref{defn: height functions} follow easily from the definition and Proposition~\ref{alpha_k is proper}. Thus, we have established the following.
    \begin{theorem}\label{thrm: CH for maximal representations}
    Suppose $G$ is a semisimple algebraic real Lie group defined over $\Q$ with Lie algebra $\mf{g}$ and $\G$ is a lattice in $G$.
    Let $\rho: \mf{sl}(2,\R) \times B \r \mf{g}$ be a deformation of a maximal representation (Def.~\ref{defn: deformations}) and let $g_t = \exp(t H_\rho)$.
    Suppose $\vp: B \r \mf{g}$ is a differentiable curve satisfying $\dot{\vp}(s) = \rho_s(X)$ for each $s\in B$.
    Then, there exists $0<\b_0<1$ such that $\vp$ satisfies the $\b$-contraction hypothesis for all $\b \in (0,\b_0)$ with a height function satisfying Assumption~\ref{assumption: bounded connected components}.
    Moreover, if $\mrm{rank}_\Q(G) = 1$, then $\b_0 = 1/2$.
    \end{theorem}

	\begin{remark}
	An explicit estimate for $\b_0$ is given in~\eqref{eqn: exponent in higher rank} when $\mrm{rank}_\Q(G)>1$.
	\end{remark}
    
    \subsection{Proof of Theorem~\ref{thrm: rank 1 DOA}} \label{section: proof of rank 1 thrm}
    
    In light of Lemma~\ref{lemma: stability of CH}, it suffices to prove the result when $\G$ is an irreducible, non-uniform lattice in $G$.
    If $\mrm{rank}_\R G = 1$, i.e. $G$ is a simple real Lie group of real rank $1$ and finite center, then Theorem~\ref{thrm: rank 1 DOA} follows from Theorem~\ref{thrm: CH in rank 1} which establishes the $\b$-contraction hypothesis for all $\b \in (0,1/2)$ with a height function satisfying Assumption~\ref{assumption: bounded connected components}.
    One can thus apply Theorems~\ref{thrm: Hdim and non-divergence} and~\ref{thrm: schmidt games}, and Proposition~\ref{propn: non-divergence of shrinking curves} to conclude.
    
    When $\mrm{rank}_\R G > 1$, we wish to apply Theorem~\ref{thrm: CH for maximal representations} in place of Theorem~\ref{thrm: CH in rank 1}. 
    Thanks to Margulis' arithmeticity theorem, $\G$ is arithmetic, i.e. $\G$ is commensurable with $G_\Z$ in some $\Q$-structure on $G$. It follows from~\cite[5.5.12]{WitteMorris} that $\G$ arises via a restriction of scalars construction\footnote{In fact, the complexifications of each simple factor of $G$ must be isogenous, but we do not need this fact.}. The reader is referred to~\cite[Section 5.5]{WitteMorris} for more details.
    In particular, since $G$ is a product of simple Lie groups of real rank $1$, we necessarily have that $\mrm{rank}_\Q \G \leq 1$.
    Since we are assuming that $G/\G$ is not compact, we thus have that $\mrm{rank}_\Q \G = 1$.
    
    It remains to show that the curves considered in Theorem~\ref{thrm: rank 1 DOA} arise as deformations of a maximal representation.
    To this end, we only need to show that $g_t$ and $u(\dot{\vp}(s))$ are part of a maximal $\mrm{SL}(2)$-triple for every $s\in B$.
    
    For each $1\leq i \leq k$, write $g_t^{(i)} = \exp (t H_i)$ for $H_i \in \mf{g}_i $.    
    Using the fact that each simple factor of $G$ is a rank $1$ group, it follows from the Jacobson-Morozov Lemma that for each $1\leq i\leq k$, $H_i$ and $\dot{\vp}_i(s)$ can be completed to an $\mf{sl}(2)$-triple $\mf{h}_i = \langle Y_i(s), H_i, \dot{\vp}_i(s)\rangle$.
    One can then check that $\mf{h} = \langle \oplus_{i=1}^k Y_i(s), \oplus_{i=1}^k H_i, \dot{\vp}(s)\rangle$ is the desired $\mf{sl}(2)$-triple.
    
    The maximality of $\mf{h}$ follows from the fact that the only proper parabolic $\Q$-subgroups in $G$ are minimal and have an abelian Levi part in this case. In particular, $\mf{h}$ cannot be contained in any proper parabolic $\Q$-subalgebra of $\mf{g}$ as desired.


	\subsection{Examples of Maximal Representations}
    The goal of this subsection is to produce more examples of deformations of maximal representations.
    In Section~\ref{section: sl2 products}, we discuss the case $G$ is a product of $\mrm{SL}(2)$'s.

	Observe that if a reductive subgroup $H <G$ is contained in some proper parabolic $\R$-subgroup $P<G$, then $H$ must be contained inside a Levi subgroup $L<P$.
    The centralizer $Z_P(L)$ of $L$ in $P$ is a non-trivial $\R$-split torus and is, thus, non-compact.
    This proves the following simple criterion for checking whether an $\mf{sl}_2$-triple is maximal in the sense of Definition~\ref{defn: deformations}.
    
    \begin{lemma} \label{lemma: criterion for maximality}
    If the centralizer $Z_G(H)$ of a reductive real Lie subgroup $H<G$ is compact, then $H$ is not contained in any proper parabolic $\R$-subgroup of $G$.
    \end{lemma}

	Note that if $Z_G(H)$ is compact, then $Z_{G\times G}(\Delta(H))$ is also compact, where $\Delta(H)$ denotes the diagonal embedding of $H$ inside $G\times G$.	We can use Lemma~\ref{lemma: criterion for maximality} to construct other examples as follows.
    
    \begin{example}
      Let $G= \mrm{SO}(p,2)$ with $p\geq 1$. Let $H $ be a $\Q$-subgroup isomorphic to $\mrm{SO}(1,2)$.
      Then, $Z_G(H) \cong \mrm{SO}(p-1)$ is compact.
      Let $A$ denote a $\Q$-split torus inside $H$.
      Suppose $B\subset \R$ is an interval and let
      \[ z :B \r Z_G(A) \]
      be an arbitrary continuous map. Then, one can check that the map $\rho: \mf{sl}(2,\R) \times B \r \mrm{Lie}(G)$
      defined by setting
      \[ \rho \left({\mf{sl}(2,\R) \times \set{s}} \right)= \mrm{Ad}(z(s))\left( \mrm{Lie}(H) \right) \]
      is indeed a deformation of a maximal representation.
    \end{example}

    \section{Specializing to Products of SL(2)}
\label{section: sl2 products}

	In this section, we specialize the results of the previous sections to the case $G = \mrm{SL}(2,\R)^r\times \mrm{SL}(2,\C)^s$, in order to complete the proof of Theorem~\ref{thrm: products of sl2}.
    Moreover, we consider curves in this setting which do not fit within the notion of maximal representations as defined in~\ref{defn: deformations}.
    The main result of this section is Theorem~\ref{thrm: CH for sl2 products}.
    
    Suppose $\G$ is a lattice in $G$.
    Then, up to finite index, and thanks to Lemma~\ref{lemma: stability of CH}, we may assume $\G = \G_1 \times \cdots \times \G_l$, where each $\G_i$ is irreducible in a sub-product of $G$.
    In light of Lemma~\ref{lemma: stability of CH}, it suffices to establish the contraction hypothesis in each irreducible factor and thus we may assume $\G$ is irreducible.
    If $r+s = 1$, then $G$ has real rank $1$ and this result was established in Section~\ref{section: height function rank 1}.
    Thus, we may also assume that $r+s > 1$ and in particular that $\mrm{rank}_\R(G) >1$.
    Define the following elements of $G$:
    
    \begin{equation} \label{eqn: standard flow}
    g_t = \left( \begin{pmatrix} e^t & 0 \\ 0 & e^{-t} \end{pmatrix} \right)_{1\leqslant i \leqslant r+s},
    \qquad  
    u(\mathbf{x})= \left( \begin{pmatrix} 1& \mathbf{x}_i \\ 0 & 1 \end{pmatrix} \right)_{1\leqslant i \leqslant r+s}.
    \end{equation}  
    
    By Margulis' arithmeticity theorem, there exists a rational structure on $G$ so that $\G$ is commensurable with $G(\Z)$.
    In this section, we assume that the $\Q$-rank of $G$ is equal to $1$ so that $G/\G$ is not compact.
    Without loss of generality and to simplify notation, we will assume that $g_t$ is $\Q$-split.
    Hence, the group $\mrm{U}=\set{u(\mbf{x}): \mbf{x} \in \R^r\times \C^s}$ is the unipotent radical of the minimal parabolic group $\mrm{P}_0$ associated with the $\Q$-torus $\mbf{S} = \set{g_t: t\in \R}$.
    The group $\mrm{P}_0$ has the following form.
    \[ \mrm{P}_0 = \set{ \left(\begin{pmatrix} * & * \\ 0 & *   \end{pmatrix} \right)  }. \]

    For each $i$, let $G_i$ denote the $i^{th}$ factor of $G$.
    Let $\mf{g} = \oplus_{i=1}^{r+s} \mf{g}_i$ denote the Lie algebra of $G$, where $\mf{g}_i$ is the Lie algebra of $G_i$.
    For $1\leq i\leq r+s$, we let $H_i \in \mf{g}_i$ denote the following element.
    \[ H_i = \begin{pmatrix} 1 & 0 \\ 0 & -1 \end{pmatrix}. \]

     Recall that any irreducible representation $V$ of $G$ is isomorphic to a tensor product $\bigotimes_i W_i$, where each $W_i$ is an irreducible representation of $G_i$.
    In particular, if $\l \in \mbf{S}^\ast$ is a highest weight for $G$ in $V$, then 
    \begin{equation} \label{eqn: weight of G and weights of G_i}
     \l = \sum_i \l_i,
    \end{equation}
    where each $\l_i \in (\R\cdot H_i)^\ast$ is a highest weight for $G_i$ in $W_i$.
    Given any such representation $V$ with highest weight $\l$ and $0\neq \mbf{x} = (\mbf{x}_i) \in  \R^r\times \C^s$, we define the following exponents:
    \begin{equation} \label{eqn: sl2 prods exponents}
    	\d_i = \l_i(H_i), \qquad \d_\mbf{x} = \sum_{i:\mbf{x}_i\neq 0} \d_i, \qquad \zeta_\mbf{x} =  \sum_{i:\mbf{x}_i= 0} \d_i.
    \end{equation}

    The following Lemma acts as a substitute for Lemma~\ref{lemma: expansion of vectors in higher rank} in this setting.

    \begin{lemma} \label{lemma: sl2 products non-maximal}     
    Suppose $V $ is a non-trivial irreducible representation for $G$ and $0\neq v\in V$ is a highest weight vector.
    Then, for all $0<\b<1$ and $0\neq \mbf{x}\in \R^r\times \C^s$, there exists $c\geq 1$ such that for all $t>0$ and all $g\in G$, the following holds:
    \begin{equation*}
    	\frac{1}{2} \int_{-1}^1 \norm{g_t u(r \mbf{x}) \cdot gv}^{-\b/\d_\mbf{x}} \;dr \leqslant c  e^{-\b' t} \norm{gv}^{-\b/\d_\mbf{x}},
    \end{equation*}
    where $\b'$ is given by
    \begin{equation} \label{eqn: exponent for non-maximal curves}
    	\b' = \b \left[ 1 -  \frac{\zeta_\mbf{x}}{\d_\mbf{x}}  \right].
    \end{equation}
    Moreover, the constant $c$ can be chosen uniformly as $\mbf{x}$ varies in a fixed compact set.
    \end{lemma}
    
    \begin{proof}
    Let $0\neq \mbf{x} = (\mbf{x}_i) \in  \R^r\times \C^s$ be given and define $\mbf{y} = (\mbf{y}_i) \in \R^r\times \C^s$ by 
    \begin{equation*}
    	\mbf{y}_i =\begin{cases}
    		1/\mbf{x}_i, & \mbf{x}_i \neq 0\\
            0, & \textrm{otherwise.}
    	\end{cases} 
    \end{equation*}
    Denote by $\mf{h} = \mf{h}(\mbf{x})$ the subalgebra of $\mf{g}$ generated by
    \[ \underline{\mbf{x}}= \left( \begin{pmatrix} 0& \mathbf{x}_i \\ 0 & 0 \end{pmatrix} \right)_{1\leqslant i \leqslant r+s}, 
    \qquad \underline{\mbf{y}}=  \left( \begin{pmatrix} 0& 0 \\ \mathbf{y}_i & 0 \end{pmatrix} \right)_{1\leqslant i \leqslant r+s}. \]
    Thus, in particular, $\mf{h} \cong \mf{sl}(2,\R)$ with the following distinguished positive diagonalizable element.
    \[ \underline{\mbf{h}} := [\underline{\mbf{x}},\underline{\mbf{y}}] = \sum_{i: \mbf{x}_i \neq 0 } H_i.\]
    Consider the following elements of $G$.
    \[ a_t = \exp\left( t \underline{\mbf{h}} \right), \qquad 
    b_t = g_t a_{-t} = \exp\left( t \sum_{i: \mbf{x}_i = 0}  H_i\right).  \]
    Since the smallest eigenvalue of $b_t$ in $V$ is $\exp \left(t \zeta_\mbf{x} \right)$, we get that
    \begin{equation}\label{eqn: remove b_t}
    \frac{1}{2} \int_{-1}^1 \norm{g_t u(r \mbf{x}) \cdot gv}^{-\b/\d_\mbf{x}} \;dr \leqslant 
    e^{\b t \zeta_\mbf{x}/\d_\mbf{x}}  \frac{1}{2} \int_{-1}^1 \norm{a_t u(r \mbf{x}) \cdot gv}^{-\b/\d_\mbf{x}} \;dr. 
    \end{equation}
    
        Denote by $P^+(V)$ the set of highest weights for $\mf{h}$ appearing in the decomposition of $V$ into irreducible representations and denote by $\chi$ the maximal element in $P^+(V)$.
    For each $\mu\in P^+(V)$, we let $V_\mu$ be the direct sum of irreducible sub-representations of $V$ with highest weight $\mu$.
    Let $\pi_\mu: V \r V_\mu$ denote the associated projection.
    
    Note that $\d_\mbf{x}$ is the largest eigenvalue of $\underline{\mbf{h}}$ in $V$.
    In particular, we can apply Proposition~\ref{propn: expansion in linear representations} to get the following estimate.
    \begin{equation} \label{eqn: estimate with a_t}
    \frac{1}{2} \int_{-1}^1 \norm{a_t u(r \mbf{x}) \cdot gv}^{-\b/\d_\mbf{x}} \;dr \leqslant
    c_1 e^{-\b t } \norm{ \pi_\chi(gv) }^{-\b/\d_\mbf{x}},
    \end{equation}
    for some constant $c_1 \geq 1$ depending only on $\b$.
    From~\eqref{eqn: remove b_t} and~\eqref{eqn: estimate with a_t}, to conclude the proof, it remains to show the existence of a constant $\e>0$ so that for all $g\in G$,
    \begin{equation} \label{eqn: non-trivial component in V_chi}
    \norm{\pi_\chi(gv)} \geqslant \e \norm{gv}.
    \end{equation}
    
    To do so, we wish to apply Proposition~\ref{propn: avoidance}.
    For a weight $\eta \in \mbf{S}^\ast$, we denote by $V^\eta$ the weight space for $\mbf{S}$ with weight $\eta$.
    Note that $V^{-\l}\oplus V^{\l} \subseteq V_\chi$, where $\l \in \mbf{S}^\ast$ denote the highest weight for $G$ in $V$.
    This follows from~\eqref{eqn: weight of G and weights of G_i} and the definition of $\underline{\mbf{h}}$.    
    In particular, by Proposition~\ref{propn: avoidance}, we get that
    \begin{equation} \label{eqn: G orbit avoids slow subspaces}
     G\cdot v \quad \bigcap \bigoplus_{\mu \in P^{+}(V)\backslash \set{ \chi}} V_\mu = \emptyset.
    \end{equation}

    Now, observe that no conjugate of $\mf{h}$ is contained in the Lie algebra of the group $\mrm{P}_0$ since the Levi part of $\mrm{P}_0$ is abelian while $\mf{h}$ is semisimple.
	Moreover, the group $\mrm{P}_0$ stabilizes the line $\R \cdot v$.
    Arguing as in the proof of Lemma~\ref{lemma: expansion of vectors in rank 1}, we see that the image of $G\cdot v$ is compact in projective space and disjoint from the closed image of $\bigoplus_{\mu \in P^{+}(V)\backslash \set{ \chi}} V_\mu$.
    Thus,~\eqref{eqn: non-trivial component in V_chi} follows.

    \end{proof}

    \begin{remark}
    Consider the case $V = V_1$ in Lemma~\ref{lemma: sl2 products non-maximal}, where $V_1$ is the representation defined in~\ref{defn: representations V_k} and used to define the height function on $G/\G$.
    Then, we have
    \begin{equation*}
    	\d_i = \begin{cases}
    	2, & \text{if } G_i =  \mrm{SL}(2,\R). \\ 4, & \text{if } G_i = \mrm{SL}(2,\C).
    	\end{cases}
    \end{equation*}
    and $\d_V = 2r+4s$.
    In particular, the exponent $\b'$ defined in~\eqref{eqn: exponent for non-maximal curves} is positive if     
    $\mbf{x} = (\mbf{x}_i) \in \R^r \times \C^s$ is such that
    \begin{equation*}
    	\# \set{1\leqslant i \leqslant r: \mbf{x}_i \neq  0 } 
   			+ 2 \cdot \# \set{r< i \leqslant r+s: \mbf{x}_i \neq  0 } >  \frac{r+2s}{2}.  
    \end{equation*}
    \end{remark}
    
    Given $0\neq \mbf{x}\in \R^r\times \C^s$, define a height function $f_\mbf{x}:G/\G \r \R_+$ by
    \begin{equation} \label{eqn: height sl2 products}
    f_\mbf{x} (x_0) = \tilde{\a}_1^{1/\d_\mbf{x}}(x_0),
    \end{equation}
    where $\tilde{\a}_1$ is defined in~\eqref{defn: b_k}.
    
    Thus, we can apply Lemma~\ref{lemma: sl2 products non-maximal} in place of Lemma~\ref{lemma: expansion of vectors in higher rank} and obtain the following direct analogue of Proposition~\ref{propn: main integral estimate in higher rank}.
    \begin{proposition} \label{propn: main integral estimate in sl2 products}
    For all $0 < \b< 1$, there exists $c_0 \geqslant 1$, depending on $\b$, so that the following holds.
    For every $t>0$, there exists a positive constant $b$ such that for all $x_0\in G/\G$ and all $0\neq \mbf{x}\in \R^r\times \C^s$,
    \begin{equation*}
    	\frac{1}{2} \int_{-1}^1 f_\mbf{x}^\b (g_t u(r\mbf{x}) x_0)\;dr
        \leqslant c_0 e^{-\b' t} f_{\mbf{x}}^\b(x_0) + b,
    \end{equation*}
    where $\b'$ is given by~\eqref{eqn: exponent for non-maximal curves}. 
    \end{proposition}
    
    \begin{remark}
    The proof of Proposition~\ref{propn: main integral estimate in higher rank} in the $\Q$-rank $1$ case used the function $f = \tilde{\a}_1^{1/d_1}$, where $d_1 =r+2s$ is the dimension of the group $\mrm{U}$.
    However, this different exponent does not change the main properties of $f$. In particular, the key ingredient, Corollary~\ref{cor: shortest are collinear} still holds for our definition of $f_\mbf{x}$.
    \end{remark}

    Given a non-constant differentiable map $\vp = (\vp_i): B \r  \mrm{Lie}(U^+(g_1)) \cong \R^{r}\oplus \C^{s}$ such that $\dot{\vp}_i$ is either identically $0$ or does not vanish on $B$ for $1\leq i \leq r+s$, we observe that $\d_{\dot{\vp}(s)}$ (eqn.~\eqref{eqn: sl2 prods exponents}) is independent of $s$.
    Thus, by taking our height function to be $f^\b_{\dot{\vp}(s)}$ for any $s$, Proposition~\ref{propn: main integral estimate in sl2 products} along with the results of Section~\ref{section: height fun higher rank} and~\ref{section: CH higher rank} imply the following.
    
    \begin{theorem}\label{thrm: CH for sl2 products}
    Suppose $G = \mrm{SL}(2,\R)^r \times \mrm{SL}(2,\C)^s$, $\G$ is an irreducible lattice in $G$ and $g_t$ is a split $1$-parameter subgroup.
    Let $\vp = (\vp_i): B \r  \mrm{Lie}(U^+(g_1)) \cong \R^{r}\oplus \C^{s}$ be a non-constant $C^{1+\e}$-map for some $\e>0$ such that $\dot{\vp}_i$ is either identically $0$ or does not vanish on $B$ for $1\leq i \leq r+s$.
    Define $\b_\vp$ as follows:
    \begin{equation*}
    	\b_\vp := \frac{1}{2} \left[  1-   \frac{ \# \set{1\leqslant i \leqslant r: \dot{\vp}_i \equiv  0 } 
   			+ 2 \cdot \# \set{r< i \leqslant r+s: \dot{\vp}_i \equiv  0 }}
            {\# \set{1\leqslant i \leqslant r: \dot{\vp}_i \not\equiv  0 } 
   			+ 2 \cdot \# \set{r< i \leqslant r+s: \dot{\vp}_i \not\equiv  0 }} \right].
    \end{equation*}
    If $\b_\vp >0$, then $\vp$ is a $g_t$-admissible curve satisfying the $\b$-contraction hypothesis for the $G$ action on $G/\G$ for all $0<\b < \b_\vp$.
    Moreover, the $\b$-contraction hypothesis holds with a height function satisfying Assumption~\ref{assumption: bounded connected components}.
    \end{theorem}

    \subsection{Proof of Theorem~\ref{thrm: products of sl2}} \label{section: proof of sl2 products}
    
    When $\G$ is irreducible, the result follows by combining Theorem~\ref{thrm: CH for sl2 products} with Theorems~\ref{thrm: Hdim and non-divergence} and~\ref{thrm: schmidt games}, and Proposition~\ref{propn: non-divergence of shrinking curves}. In particular, the dimension of divergent on average orbits is at most $1-\b_\vp$, where $\b_\vp$ is as in Theorem~\ref{thrm: CH for sl2 products}. Note that this upper bound is less than $1$ if and only if $\b_\vp >0$.

    \subsection{Non-maximal Curves on Products of SO(d,1)} \label{section: products of so(n,1)}
    
    The methods of this section can be used with minor modifications to obtain an analogous result to Theorem~\ref{thrm: products of sl2} when $G$ is a product of of copies of $\mrm{SO}(n,1)$.
    
    \begin{theorem} \label{thrm: products of so(n,1)}
    	Suppose $G = G_1 \times \cdots \times G_k$ is such that for each $1\leq i\leq k$, $G_i \cong \mrm{SO}(d_i,1)$ for some $d_i \geq 2$. Let $\G$ be an irreducible lattice in $G$. For each $1\leq i\leq k$, let $g_t^{(i)}$ be a $1$-parameter subgroup of $G_i$ which is $\mrm{Ad}$-diagonalizable over $\R$, and suppose $ \vp_i: B \r \mrm{Lie}(U^+(g_1^{(i)})) $ is a $C^{1+\e}$-map for some $\e >0$.
        Assume that for each $i$, $\dot{\vp}_i$ is either non-vanishing or vanishes identically on $B$.
    Let $g_t = (g_t^{(i)})_{1\leq i\leq k}$ and $\vp = \oplus_{i=1}^k \vp_i$. Assume that $g_t$ is split and $\vp$ is $g_t$-admissible and non-constant.
	Then, for every $x_0\in X=G/\G$, the Hausdorff dimension of the set of points $s\in B $ for which the forward orbit $(g_tu(\vp(s))x_0)_{t\geqslant 0}$ is divergent on average is at most
    \[ \frac{1}{2} + \frac{1}{2}  \frac{\displaystyle{\sum_{1\leq i \leq k, \dot{\vp}_i\equiv 0} (d_i-1) } }{\displaystyle{\sum_{1\leq i \leq k, \dot{\vp}_i \not\equiv 0} (d_i-1) }   }.  \]
        Moreover, if the above quantity is strictly less than $1$, then parts $(ii)-(iv)$ of Theorem~\ref{thrm: rank 1 DOA} also hold in this setting.
	\end{theorem}

    \section{The Contraction Hypothesis for \rm{SL}(2,R) Actions}
\label{section: height linear forms}

	In this section, we construct a family of functions that will allow us to control recurrence to compact sets in $\mrm{SL}(d,\R)/\mrm{SL}(d,\Z)$.
    This construction was introduced in~\cite{EskinMargulisMozes} and generalized in~\cite{BQ-RandomWalkRecurrence}. Here, we follow the approach of ~\cite{BQ-RandomWalkRecurrence}.
    The main result of this section, Theorem~\ref{thrm: CH for sl2 actions}, establishes the contraction hypothesis in the context of $\mrm{SL}(2,\R)$ actions completing the proof of Theorem~\ref{thrm: sl2 actions}.
    
    We recall the set up and notation of Theorem~\ref{thrm: sl2 actions}.
    Let $L$ be a semisimple algebraic Lie group defined over $\Q$ and let $\G$ be an arithmetic lattice in $L$.
    We let $\rho: \mrm{SL}(2,\R)\r L$ be a non-trivial homomorphism and let $G$ denote the image of $\rho$.
    Let $g_t$ and $u_s$ be as in the statement of Theorem~\ref{thrm: sl2 actions}.
    
    The aim of this section is to show that the ``curve'' $u_s$ satisfies the $\b$-contraction hypothesis for the $G$-action on $L/\G$.
    In light of Lemma~\ref{lemma: stability of CH}, we have the freedom of replacing $\G$ by a commensurable lattice without loss of generality.
    
    In particular, we may regard $L$ as a subgroup of $S= \mrm{SL}(d,\R)$ for some $d \geq 1$ so that $\G=L\cap\L$, for $\L  = \mrm{SL}(d,\Z)$.
    Since $L$ is defined over $\Q$, we have that $L\L$ is closed in $S$ and the homogeneous space $X = L/\G \cong L/L\cap\L$ can be regarded a closed subspace of $X' = S/\L$.
    As a result, the contraction hypothesis for the $G$-action on $L/\G$ will follow from that of the $G$-action on $S/\L$.
    Therefore, without loss of generality, we will assume through the remainder of this section that
    \[ L = \mrm{SL}(d,\R), \qquad \G = \mrm{SL}(d,\Z),\qquad X = L/\G. \]

    Using the results in~\cite{BQ-RandomWalkRecurrence}, Shi showed in~\cite{Shi} the $\b$-contraction hypothesis for the $G$ action on $X$ for some $\b>0$.
    We reproduce the proof in this section with some modifications to obtain a more precise range for the exponent $\b$.

    \subsection{The Contraction Hypothesis in Vector Spaces}
\label{section: CH in vector spaces higher rank}

	As before, we first construct functions in linear representations encoding divergence in $X$ and then convert our linear estimates into a height function on $X$.
    The relevant representation in this case is $\bigoplus_i \bigwedge^i \R^{d}$.

    Let $ H_0,Z \in \mrm{Lie}(G) \cong \mf{sl}(2,\R)$ be such that
    \begin{equation} \label{log g_t}
    	g_t = \exp( t H_0), \qquad u_s  \exp(sZ).
     \end{equation}
     In particular, by definition of $g_t$ and $u_s$, we have
     \[ [H_0,Z] = 2Z.  \]

    Denote by $P^+$ the set of all possible highest weights appearing in linear representations of $G$.
    From the representation theory of $\mrm{SL}(2,\R)$, the set $P^+$ of highest weights can be identified with $\N \cup \set{0}$.
    Given an arbitrary finite dimensional representation $V$ of $G$ and $\l\in P^+$, we use $V^\l$ to denote the direct sum of all irreducible subrepresentations of $V$ whose highest weight is $\l$.
    We denote by $\pi_\l:V \r V^\l$ the associated $G$-equivariant projection.

    Following~\cite{BQ-RandomWalkRecurrence}, we define two sets of exponents. 
	For $i \in \set{1,\dots,d-1}$ and $\l \in P^{+}$, define
    \begin{equation} 
    \d_i = i(d-i), \qquad \d_\l = \l(H_0).
    \end{equation}
    
    In particular, we have
    \begin{equation*}
    	\l(H_0) = 0 \Leftrightarrow \l = 0.
    \end{equation*}
    
    For every $\epsilon >0$ and $0<i <d$, we define a function $\vp_{\epsilon}$ on $\bigwedge^i \R^{d}$ as follows.
    For $v \in \bigwedge^i \R^{d}$, let
    \begin{align} \label{defn: phi}
    \vp_{\epsilon}(v) = \begin{cases}
    	 \min\limits_{\l \in P^+ \backslash\set{0} } \epsilon^{\frac{\d_i}{\d_\l} } \norm{ \pi_\l(v)}^{-1/\d_\l} 
         				& \text{if } \norm{\pi_0(v)} < \epsilon^{\d_i}, \\
         0 & \text{otherwise.}
    \end{cases}
    \end{align}
    Here, we use the convention $1/0 = \infty$.

    The following Lemma is the form in which we use Proposition~\ref{propn: expansion in linear representations} in our setting.
    \begin{lemma} \label{lemma: contraction hypothesis for phi}
    For every $\b \in (0, 1)$, there exists $D \geqslant 1$ such that for all $t,\epsilon >0$ and all $v \in \bigwedge^i \R^{d} $ with $0< i<d$,
    \begin{equation*}
    	\frac{1}{2}\int_{-1}^1 \vp_{\epsilon}^\b(g_t u_s v) \;ds \leqslant D e^{-\b   t} \vp_{\epsilon}^\b(v).
    \end{equation*}
    \end{lemma}
    
    \begin{proof}
    	
    	First, we note that for all $g\in G$, $\pi_0(gv) = g\pi_0(v) = \pi_0(v)$.
        In particular, if $\vp_\epsilon(v) = 0$, then
        $\norm{\pi_0(v)} = \norm{\pi_0(g_t u_s v) }\geqslant \epsilon^{\d_i}$ for all $s$ and $t$
        and the statement follows in this case. Hence, we may assume $\vp_\epsilon (v) \neq 0$.

        Moreover, since the integral of the minimum of finitely many functions is bounded by the minimum of their integrals, it suffices to prove
        \begin{equation} \label{contraction in each weight}
        \frac{1}{2} \int_{-1}^1 \norm{  \pi_\l(g_t u_s v)}^{-\b/\d_\l} \;ds = 
        	\frac{1}{2} \int_{-1}^1 \norm{ g_t u_s \pi_\l(v)}^{-\b/\d_\l} \;ds 
            	\leqslant D e^{- \b  t} \norm{\pi_\l(v)}^{-\b/\d_\l},
        \end{equation}
        for each $\l \in P^{+}\backslash\set{0}$ and for some constant $D$ depending only on $\b$.

        If $\pi_\l(v) = 0$ for some $\l$, then the right-hand side of~\eqref{contraction in each weight} is $\infty$ and the claim is proved in this case.        
        Now, suppose that $\pi_\l(v) \neq 0$ for some $\l \in P^+\backslash \set{0}$.

        Denote by $V =\left( \bigwedge^i \R^{d} \right)_\l$, i.e. the image of $\bigwedge^i\R^{d}$ under the projection $\pi_\l$.
        From the representation theory of $\mrm{SL}(2,\R)$, we see that the dimension of an irreducible representation with weight $\l$ is equal to $\d_\l + 1$.
        In particular, choosing $\b\in (0,1)$ allows us to apply
        Proposition~\ref{propn: expansion in linear representations} to get that
        \begin{equation*} 
        	\frac{1}{2} \int_{-1}^1 \norm{ g_t u_s \pi_\l(v)}^{-\b/\d_\l} \;ds
            	\leqslant D e^{- \b t } \norm{\pi_\l(v)}^{-\b/\d_\l}.
        \end{equation*}
        This proves~\eqref{contraction in each weight} and completes the proof.
    \end{proof}


    \subsection{The Contraction Hypothesis on $X$}

	The space $X = S/\L$ may be identified with the space of unimodular lattices in $\R^{d}$ via the map
    $g \mrm{SL}(d,\Z) \mapsto g\Z^{d}$.
    For $x \in X$, let $P(x)$ denote the set of all primitive subgroups of the lattice $x$.
    Recall that a subgroup of a lattice in $\R^{d}$ is primitive if its $\Z$ basis can be completed to a basis of $\R^{d}$ as an $\R$-vector space.
    
    We say a monomial $v_1 \wedge \cdots \wedge v_i \in \bigwedge^i\R^{d} $ is $x$-integral if the abelian subgroup of $\R^{d}$ generated by
    $\set{v_1,\dots,v_i}$ belongs to $P(x)$.
    
    Now, we define the function $f_\epsilon: X \r [0,\infty]$ by 
    \begin{align} \label{defn: alpha}
    	f_\epsilon(x) = \max \vp_{\epsilon}(v),
    \end{align}
    where the maximum is taken over all non-zero $x$-integral monomials $v\in \bigwedge^i\R^{d}$ and all $0<i<d$.

    \begin{remark}
    The function $f_\epsilon$ can assume the value $\infty$. However, for any $x\in X$, one can choose $\epsilon$ to be small enough so that $f_\epsilon(x) <\infty$. In fact, one can choose such $\epsilon$ uniformly for compact subsets of $X$ by Mahler's criterion. 
    \end{remark}
    
    Following the same lines as Proposition 5.3 in~\cite{BQ-RandomWalkRecurrence}, we obtain the following result.
    \begin{proposition} \label{propn: main integral estimate for sl2 actions}
    For all $\b\in (0,1)$, there exists $c_0 \geqslant 1$ such that
    for all $t>0$, there exist constants $\e_0, b>0$, depending on $t$,
    satisfying
    \begin{equation*}
    	\frac{1}{2} \int_{-1}^1 f_{\e_0}^\b(g_tu_s x)\;ds \leqslant c_0 e^{-\b t} f_{\e_0}^\b(x) + b,
    \end{equation*}
    for all $x\in X$.
    \end{proposition}

    \begin{proof}
    	Suppose $\b \in (0,1)$ and let $D\geqslant 1$ be the constant in the conclusion of Lemma~\ref{lemma: contraction hypothesis for phi}.
        For a compact set $Q\subset G$, define
          \begin{equation} \label{eqn: omega Q}
          \omega(Q) = \sup_{g\in Q} \max\set{ \norm{g}, \norm{ g^{-1}}  },
          \end{equation}
        where $\norm{\cdot}$ is the operator norm induced by the euclidean norm on $V = \bigoplus_{i=1}^{d-1} \bigwedge^i \R^{d}$.
        
        Fix some $t >0$ and denote by $\omega$ the following constant.
        \begin{equation} \label{eqn: omega}
        	\w = \w \left( \set{g_t u_s: s\in [-1,1] } \right).
        \end{equation}

        Let $\e_0>0$ be a constant to be determined later.
        Note that the exponents $\d_\l$ in the definition of $\vp_{\e_0}$ satisfy 
        \begin{equation*}
        	\d_\l \geqslant 1/\a(H_0),
        \end{equation*}
        for all $0\neq \l \in P^{+}$.        
        Thus, by definition of $\vp_{\e_0}$, for all $s \in [-1,1]$ and all $v\in V$,
        \begin{equation} \label{log smooth}
             \omega^{-\a(H_0)} \vp_{\e_0}(v) \leqslant \vp_{\e_0}(g_tu_sv) 
             	\leqslant \omega^{\a(H_0)} \vp_{\e_0}(v).
        \end{equation}
        
        It is shown in~\cite[Claim 5.9]{BQ-RandomWalkRecurrence} that given a compact subset $Q$ of $G$, there exists constants $C_1 \geqslant 1$ and $\e_0 >0$, depending on $Q$, such that whenever $f_{\e_0}(x) > C_1$, the set of $x$-integral monomials $v$, satisfying
        \begin{equation} \label{eqn: defn of psi}
        	\vp_{\e_0}(v) \geqslant f_{\e_0}(x)/\omega(Q)^{2\a(H_0)},
        \end{equation}
        contains at most one vector up to a sign in each of $\bigwedge^i \R^{d}$ with $0<i<d$.
        In particular, we may apply this result to the compact set $Q=\set{g_t u_s: s\in [-1,1] }$.
        
        Suppose $x\in X$ satisfies $f_{\e_0}(x) >C_1$ and let $\Psi$ denote the set of $x$-integral monomials satisfying~\eqref{eqn: defn of psi}.
        Then, Lemma~\ref{lemma: contraction hypothesis for phi} implies
        \begin{equation*}
        \frac{1}{2} \int_{-1}^1 f_{\e_0}^\b(g_tu_s x)\;ds
        \leqslant \sum_{v\in \Psi} \frac{1}{2} \int_{-1}^1 \vp_{\e_0}^\b(g_t u_s v)\;ds 
        \leqslant 4d D e^{-\b t} f_{\e_0}^\b(x).
        \end{equation*}
        
        Finally, if $f_{\e_0}(x) < C_1$ for some $x\in X$, then~\eqref{log smooth} implies that
        \begin{equation*}
        \frac{1}{2} \int_{-1}^1 f_{\e_0}^\b(g_tu_s x)\;ds
        \leqslant \omega^{\a(H_0)} C_1.
        \end{equation*}
        Thus, the statement of the Proposition follows by taking $c_0 = 4dD$ and $b = \omega^{\a(H_0)} C_1$.
    \end{proof}

    Proposition~\ref{propn: main integral estimate} establishes that the functions $f_\e^\b$ satisfy the main property in the $\b$-contraction hypothesis (Def.~\ref{defn: height functions}) for the $G=\mrm{SL}(2,\R)$ action on homogeneous spaces $X$.
    Properties $(1)$, $(2)$ and $(4)$ follow from Mahler's compactness criterion and the lower semi-continuity of $f_{\e_0}$.
    Finally, Assumption~\ref{assumption: bounded connected components} follows from the following lemma.
    \begin{lemma}\label{lemma: bounded conn. comp. sl2 actions}
    There exists $N\in \N$, depending only on $G$ and $\G$, such that for every $T>0$, there exists $M_0>0$ such that for all $x\in G/\G$, and $M_1\geq M_0$, the following holds.
    \begin{equation}\label{eqn: cusp set sl2 actions}
    	\text{The set } \set{|s|\leqslant T: f(u_s x) > M_1} \text{ has at most } N \text{ connected components}.
    \end{equation}
    \end{lemma}
    \begin{proof}
    The proof is identical to that of Lemma~\ref{lemma: bounded conn. comp. higher rank} and relies on the bounded cardinality of a set of vectors analogous to the set $\Psi$ in the proof of Proposition~\ref{propn: main integral estimate for sl2 actions}.
    \end{proof}
    
    Thus, we have established the following statement.
    \begin{theorem} \label{thrm: CH for sl2 actions}
    Let $B \subset \R$ be an interval and suppose $L$ is a semisimple algebraic Lie group defined over $\Q$, $\G$ an arithmetic lattice in $L$, and $\rho: \mrm{SL}(2,\R) \r L $ a non-trivial representation. Let
          \[ g_t = \rho \left( \begin{pmatrix}e^t &0\\ 0 & e^{-t}    \end{pmatrix}  \right), \qquad u(\vp(s)) = \rho  \left( \begin{pmatrix}1 &s\\ 0 & 1    \end{pmatrix}  \right), s\in B. \]
    Then, $\vp(s)$ is a $g_t$-admissible curve satisfying the $\b$-contraction hypothesis for the action of $G=\rho(\mrm{SL}(2,\R))$ on $L/\G$ for all $\b\in (0,1/2)$, with a height function satisfying Assumption~\ref{assumption: bounded connected components}.
    \end{theorem}
    

    \subsection{Proof of Theorem~\ref{thrm: sl2 actions}} \label{section: proof of sl2 actions}
    
    The result follows by combining Theorem~\ref{thrm: CH for sl2 actions} with Theorems~\ref{thrm: Hdim and non-divergence} and~\ref{thrm: schmidt games}, and Proposition~\ref{propn: non-divergence of shrinking curves}.

    \section{Conclusions and Open Problems} \label{section: conclusion}

	The results of this article leave open several natural questions, which we now discuss.

	\subsection{Lower Bounds}

    It is known (cf.~\cite{KadyrovPohl}) that when $G = \mrm{SL}(2,\R)$, then the upper bound obtained in Theorem~\ref{thrm: rank 1 DOA} on the dimension of divergent on average orbits coincides with the lower bound.
    This fact can be used to deduce a lower bound on the dimension of divergent on average orbits in a special case of Theorem~\ref{thrm: rank 1 DOA} as follows.

\begin{proof}[Proof of Corollary~\ref{cor: lower bound cor}]
	By Theorem~\ref{thrm: rank 1 DOA}, we only need to establish the lower bound.
	First, we observe\footnote{Note that the converse is not true in general.} that if the orbit $(g_t^{(1)} u(\vp_1(s)) x_0)_{t\geqslant 0}$ is divergent on average in $\mrm{SL}(2,\R)/\G_1$, then the orbit $(g_t u(\vp(s)) x_0)_{t\geqslant 0}$ is divergent on average in $G/\G$.
    This follows from the fact that every compact subset $\mc{K} \subset G/\G$ is contained in a set of the form $\mc{K}_1 \times \mc{K}_2 $, where $\mc{K}_1 \subset \mrm{SL}(2,\R)/\G_1$ and $\mc{K}_2 \subset G'/\G'$ are compact sets.
    
	The assumption that $\vp_1$ is non-constant implies that $\vp_1(B)$ is a compact non-trivial interval. It follows from~\cite[Theorem 1.3]{KadyrovPohl} that the set of points $r\in \vp_1(B)$ for which the orbit $(g_t^{(1)} u(r) x_0)_{t\geqslant 0}$ is divergent on average has Hausdorff dimension $1/2$. Since $\vp_1$ is Lipschitz, and Lipschitz maps do not increase Hausdorff dimension, we obtain the desired lower bound.
\end{proof}

	Corollary~\ref{cor: lower bound cor} leaves open the question of whether $1/2$ is in fact a lower bound on the dimension of divergent orbits (not divergent on average) when $\G$ is reducible.
However, this corollary motivates the following natural question.
\begin{question} \label{lower bound question}
	In the settings of Theorems~\ref{thrm: rank 1 DOA},~\ref{thrm: products of sl2}, and~\ref{thrm: sl2 actions}: If $G/\G$ is not compact, is the Hausdorff dimension of the divergent on average orbits of $g_t$ equal to the minimum of $1$ and the upper bounds obtained in \textit{loc. cit.}?
\end{question}

    \subsubsection{A Counter Example}
    As noted in the introduction, Theorem~\ref{thrm: products of sl2} gives a meaningful upper bound on the dimension of divergent on average orbits only when
    \begin{equation}\label{max condition}
    	\# \set{1\leqslant i \leqslant r_k: (\dot{\vp}_k)_i \not\equiv  0 } 
   			+ 2 \cdot \# \set{r_k< i \leqslant r_k+s_k: (\dot{\vp}_k)_i \not\equiv  0 } > \frac{r_k+2s_k}{2},
    \end{equation}
    for all $1\leq k \leq l$. We now show that this condition cannot be relaxed in general.
    
    Suppose $G = \mrm{SL}(2,\R) \times \mrm{SL}(2,\C)$ and $\G$ is an irreducible, non-uniform lattice in $G$. One may construct such a lattice using the Galois embedding of $\mrm{SL}(2,\mc{O}_K)$ into $G$, where $\mc{O}_K$ is the ring of integers in $K = \Q(\sqrt[3]{2})$ for example.
    Let $\vp:B \r\R\times \C$ be given by $\vp(s) = (s,0)$ and let $g_t$ be as in~\eqref{eqn: g_t and u(x)}.
    Let $x_0 = g\G$ for $g\in G$ the Weyl ``element'' given by
    \[ g = \left(\begin{pmatrix}  0 & 1 \\ -1 & 0 \end{pmatrix}, \begin{pmatrix} 0 & 1 \\ -1 & 0 \end{pmatrix}\right). \]
	One can then check directly using the definition of the proper function $f$ in~\eqref{defn: height higher rank} that $f(g_t u(\vp(s)) x_0) $ tends to $\infty$ as $t\r\infty$ for every $s\in B$.
    Roughly the amount of expansion provided by the first factor (an eigenvalue of $2$) is negated by the contraction in the second factor (an eigenvalue of $-4$).
	
    This, however, leaves open the question of whether the dimension of divergent on average orbits is strictly less than $1$ in the critical case when the $2$ sides of inequality~\eqref{max condition} are equal.

    \subsection{More General Diagonal Flows}
    
    The notion of a deformation of a maximal representation a priori restricts our results to curves whose tangents form an $\mf{sl}(2,\R)$-triple with the diagonal flow in question.
    However, the proof of Theorem~\ref{thrm: products of sl2} (particularly Lemma~\ref{lemma: sl2 products non-maximal}) shows that our methods apply to a more general class of curves which we now describe. 
    
    Suppose $\rho$ is a deformation of a maximal representation and let $g_t = \exp(t H_\rho)$ and $\vp(s) = \rho_s(X)$.
    For each $k$, let $V_k^+$ denote the $g_t$-expanding subspace of the vector space $V_k$ defined in~\ref{defn: representations V_k}.    
    Suppose $A$ is a maximal $\R$-split torus containing $g_t$, which we identify with its Lie algebra.
    Denote by $A^+(H_\rho)$ the cone inside $A$ of semisimple elements $H'$ which have positive eigenvalues on $V_k^+$ for each $k$.
    Note that $H_\rho \in A^+(H_\rho)$ by definition.
    Given any $H'\in A^+(H_\rho)$ such that $\vp$ is $\exp(t H')$-admissible, one can check that the proofs of the integral estimates obtained in Section~\ref{section: CH higher rank} go through with $\exp(t H')$ in place of $g_t$, with a contraction rate depending the eigenvalues of $H'$.
    However, it is not clear whether the upper bounds on the dimension of divergent orbits for the flows obtained in this manner are sharp.

    

     \section*{Acknowledgements}
 		I would like to thank my advisor, Nimish Shah, for numerous valuable discussions. I would also like to thank Jon Chaika for introducing me to the types of problems studied in this article.
 		Finally, I would like to thank the anonymous referee for a careful reading of the article and for several suggestions and corrections on the previous version.

\bibliography{bibliography}{}

\newcommand{\etalchar}[1]{$^{#1}$}
\providecommand{\bysame}{\leavevmode\hbox to3em{\hrulefill}\thinspace}
\providecommand{\MR}{\relax\ifhmode\unskip\space\fi MR }
\providecommand{\MRhref}[2]{%
  \href{http://www.ams.org/mathscinet-getitem?mr=#1}{#2}
}
\providecommand{\href}[2]{#2}
\begin{thebibliography}{ABRdS18}

\bibitem[AAE{\etalchar{+}}17]{AAEKMU}
H.~{Al-Saqban}, P.~{Apisa}, A.~{Erchenko}, O.~{Khalil}, S.~{Mirzadeh}, and
  C.~{Uyanik}, \emph{{Exceptional directions for the Teichm\"uller geodesic
  flow and Hausdorff dimension}}, Journal of the European Math. Soc. (to
  appear) (2017).

\bibitem[ABRdS18]{Aka-ExtremalCriterion}
Menny Aka, Emmanuel Breuillard, Lior Rosenzweig, and Nicolas de~Saxc\'e,
  \emph{Diophantine approximation on matrices and {L}ie groups}, Geom. Funct.
  Anal. \textbf{28} (2018), no.~1, 1--57. \MR{3777412}

\bibitem[ABV18]{AnBeresnevichVelani-WinningPlanarCurves}
Jinpeng An, Victor Beresnevich, and Sanju Velani, \emph{Badly approximable
  points on planar curves and winning}, Adv. Math. \textbf{324} (2018),
  148--202. \MR{3733884}

\bibitem[Ara94]{Aravinda}
C.~S. Aravinda, \emph{Bounded geodesics and {H}ausdorff dimension}, Math. Proc.
  Cambridge Philos. Soc. \textbf{116} (1994), no.~3, 505--511. \MR{1291756}

\bibitem[Ber15]{Beresnevich-BadCurves}
Victor Beresnevich, \emph{Badly approximable points on manifolds}, Invent.
  Math. \textbf{202} (2015), no.~3, 1199--1240. \MR{3425389}

\bibitem[BKM15]{BeresnevichKleinbock-WeakNonPlanar}
Victor Beresnevich, Dmitry Kleinbock, and Gregory Margulis, \emph{Non-planarity
  and metric {D}iophantine approximation for systems of linear forms}, J.
  Th\'eor. Nombres Bordeaux \textbf{27} (2015), no.~1, 1--31. \MR{3346961}

\bibitem[BL05]{BugeaudLaurent}
Yann Bugeaud and Michel Laurent, \emph{Exponents of {D}iophantine approximation
  and {S}turmian continued fractions}, Ann. Inst. Fourier (Grenoble)
  \textbf{55} (2005), no.~3, 773--804. \MR{2149403}

\bibitem[Bor69]{Borel-FrenchBook}
Armand Borel, \emph{Introduction aux groupes arithm\'etiques}, Publications de
  l'Institut de Math\'ematique de l'Universit\'e de Strasbourg, XV.
  Actualit\'es Scientifiques et Industrielles, No. 1341, Hermann, Paris, 1969.
  \MR{0244260}

\bibitem[Bor91]{Borel-LinearAlgebraicGroups}
\bysame, \emph{Linear algebraic groups}, second ed., Graduate Texts in
  Mathematics, vol. 126, Springer-Verlag, New York, 1991. \MR{1102012}

\bibitem[Bou02]{Bourbaki-4-6}
Nicolas Bourbaki, \emph{Lie groups and {L}ie algebras. {C}hapters 4--6},
  Elements of Mathematics (Berlin), Springer-Verlag, Berlin, 2002, Translated
  from the 1968 French original by Andrew Pressley. \MR{1890629}

\bibitem[BPV11]{BadziahinEtal-SchmidtConjecture}
Dzmitry Badziahin, Andrew Pollington, and Sanju Velani, \emph{On a problem in
  simultaneous {D}iophantine approximation: {S}chmidt's conjecture}, Ann. of
  Math. (2) \textbf{174} (2011), no.~3, 1837--1883. \MR{2846492}

\bibitem[BQ11]{BQ-RandomWalkRecurrence}
Yves Benoist and Jean-Francois Quint, \emph{Random walks on finite volume
  homogeneous spaces}, Inventiones mathematicae \textbf{187} (2011), no.~1,
  37--59.

\bibitem[CC16]{CheungChevallier}
Yitwah Cheung and Nicolas Chevallier, \emph{Hausdorff dimension of singular
  vectors}, Duke Math. J. \textbf{165} (2016), no.~12, 2273--2329. \MR{3544282}

\bibitem[CCM13]{ChaikaCheungMasur-WinningModuli}
Jonathan Chaika, Yitwah Cheung, and Howard Masur, \emph{Winning games for
  bounded geodesics in moduli spaces of quadratic differentials}, J. Mod. Dyn.
  \textbf{7} (2013), no.~3, 395--427. \MR{3296560}

\bibitem[Che07]{Cheung-Sl2products}
Yitwah Cheung, \emph{Hausdorff dimension of the set of points on divergent
  trajectories of a homogeneous flow on a product space}, Ergodic Theory Dynam.
  Systems \textbf{27} (2007), no.~1, 65--85. \MR{2297087}

\bibitem[Che11]{Cheung-SingularPairs}
\bysame, \emph{Hausdorff dimension of the set of singular pairs}, Ann. of Math.
  (2) \textbf{173} (2011), no.~1, 127--167. \MR{2753601}

\bibitem[Dan85]{Dani-Divergent}
S.~G. Dani, \emph{Divergent trajectories of flows on homogeneous spaces and
  {D}iophantine approximation}, J. Reine Angew. Math. \textbf{359} (1985),
  55--89. \MR{794799}

\bibitem[Dan86]{Dani-Bounded}
\bysame, \emph{Bounded orbits of flows on homogeneous spaces}, Comment. Math.
  Helv. \textbf{61} (1986), no.~4, 636--660. \MR{870710}

\bibitem[Dan89]{Dani-Bounded2}
\bysame, \emph{On badly approximable numbers, {S}chmidt games and bounded
  orbits of flows}, Number theory and dynamical systems ({Y}ork, 1987), London
  Math. Soc. Lecture Note Ser., vol. 134, Cambridge Univ. Press, Cambridge,
  1989, pp.~69--86. \MR{1043706}

\bibitem[DM91]{DaniMargulis}
S.~G. Dani and G.~A. Margulis, \emph{Asymptotic behaviour of trajectories of
  unipotent flows on homogeneous spaces}, Proc. Indian Acad. Sci. Math. Sci.
  \textbf{101} (1991), no.~1, 1--17. \MR{1101994}

\bibitem[EGL16]{Einsiedler-CurvesNumfields}
Manfred Einsiedler, Anish Ghosh, and Beverly Lytle, \emph{Badly approximable
  vectors, {$C^1$} curves and number fields}, Ergodic Theory Dynam. Systems
  \textbf{36} (2016), no.~6, 1851--1864. \MR{3530469}

\bibitem[EM01]{EskinMasur-HeightStrata}
Alex Eskin and Howard Masur, \emph{Asymptotic formulas on flat surfaces},
  Ergodic Theory Dynam. Systems \textbf{21} (2001), no.~2, 443--478.
  \MR{1827113}

\bibitem[EM04]{EskinMargulis-RandomWalks}
Alex Eskin and Gregory Margulis, \emph{Recurrence properties of random walks on
  finite volume homogeneous manifolds}, Random Walks and Geometry: Proceedings
  of a Workshop at the Erwin Schr{\"o}dinger Institute, Vienna, June 18-July
  13, 2001, Walter de Gruyter, 2004, p.~431.

\bibitem[EMM98]{EskinMargulisMozes}
Alex Eskin, Gregory Margulis, and Shahar Mozes, \emph{Upper bounds and
  asymptotics in a quantitative version of the oppenheim conjecture}, Annals of
  Mathematics \textbf{147} (1998), no.~1, 93--141.

\bibitem[EMM15]{EMM}
Alex Eskin, Maryam Mirzakhani, and Amir Mohammadi, \emph{Isolation,
  equidistribution, and orbit closures for the {SL}(2, r) action on moduli
  space}, Annals of Mathematics (2015), 673--721.

\bibitem[GR70]{GarlandRaghunathan}
H.~Garland and M.~S. Raghunathan, \emph{Fundamental domains for lattices in
  ({R}-)rank {$1$} semisimple {L}ie groups}, Ann. of Math. (2) \textbf{92}
  (1970), 279--326. \MR{0267041}

\bibitem[KKLM17]{KKLM-SingSystems}
S.~Kadyrov, D.~Kleinbock, E.~Lindenstrauss, and G.~A. Margulis, \emph{Singular
  systems of linear forms and non-escape of mass in the space of lattices},
  Journal d'Analyse Math{\'e}matique \textbf{133} (2017), no.~1, 253--277.

\bibitem[KLW04]{KleinbockLindenstraussWeiss}
Dmitry Kleinbock, Elon Lindenstrauss, and Barak Weiss, \emph{On fractal
  measures and {D}iophantine approximation}, Selecta Math. (N.S.) \textbf{10}
  (2004), no.~4, 479--523. \MR{2134453}

\bibitem[KM98]{KleinbockMargulis}
D.~Y. Kleinbock and G.~A. Margulis, \emph{Flows on homogeneous spaces and
  {D}iophantine approximation on manifolds}, Ann. of Math. (2) \textbf{148}
  (1998), no.~1, 339--360. \MR{1652916}

\bibitem[KM15]{KleinbockMerrill}
Dmitry Kleinbock and Keith Merrill, \emph{Rational approximation on spheres},
  Israel J. Math. \textbf{209} (2015), no.~1, 293--322. \MR{3430242}

\bibitem[KMW10]{KleinbockMargulisWang}
Dmitry Kleinbock, Gregory Margulis, and Junbo Wang, \emph{Metric {D}iophantine
  approximation for systems of linear forms via dynamics}, Int. J. Number
  Theory \textbf{6} (2010), no.~5, 1139--1168. \MR{2679461}

\bibitem[KP17]{KadyrovPohl}
S.~Kadyrov and A.~Pohl, \emph{Amount of failure of upper-semicontinuity of
  entropy in non-compact rank-one situations, and {H}ausdorff dimension},
  Ergodic Theory Dynam. Systems \textbf{37} (2017), no.~2, 539--563.
  \MR{3614037}

\bibitem[KW04]{KleinbockWeiss-BoundedModuli}
Dmitry Kleinbock and Barak Weiss, \emph{Bounded geodesics in moduli space},
  Int. Math. Res. Not. (2004), no.~30, 1551--1560. \MR{2049831}

\bibitem[KW10]{KleinbockWeiss-SchmidtSLn+m}
\bysame, \emph{Modified {S}chmidt games and {D}iophantine approximation with
  weights}, Adv. Math. \textbf{223} (2010), no.~4, 1276--1298. \MR{2581371}

\bibitem[KW13]{KleinbockWeiss-SchmidtExpanding}
\bysame, \emph{Modified {S}chmidt games and a conjecture of {M}argulis}, J.
  Mod. Dyn. \textbf{7} (2013), no.~3, 429--460. \MR{3296561}

\bibitem[Mas92]{Masur-NUE}
Howard Masur, \emph{Hausdorff dimension of the set of nonergodic foliations of
  a quadratic differential}, Duke Math. J. \textbf{66} (1992), no.~3, 387--442.
  \MR{1167101}

\bibitem[Mor15]{WitteMorris}
Dave~Witte Morris, \emph{Introduction to arithmetic groups}, Deductive Press,
  [place of publication not identified], 2015. \MR{3307755}

\bibitem[Sch66]{Schmidt-Games}
Wolfgang~M. Schmidt, \emph{On badly approximable numbers and certain games},
  Trans. Amer. Math. Soc. \textbf{123} (1966), 178--199. \MR{0195595}

\bibitem[Sch69]{Schmidt-BadLinearSystems}
\bysame, \emph{Badly approximable systems of linear forms}, J. Number Theory
  \textbf{1} (1969), 139--154. \MR{0248090}

\bibitem[Sha09a]{Shah-Duke2}
Nimish~A. Shah, \emph{Asymptotic evolution of smooth curves under geodesic flow
  on hyperbolic manifolds}, Duke Math. J. \textbf{148} (2009), no.~2, 281--304.
  \MR{2524497}

\bibitem[Sha09b]{Shah-Inventiones}
\bysame, \emph{Equidistribution of expanding translates of curves and
  {D}irichlet's theorem on {D}iophantine approximation}, Invent. Math.
  \textbf{177} (2009), no.~3, 509--532. \MR{2534098}

\bibitem[Sha09c]{Shah-Duke1}
\bysame, \emph{Limiting distributions of curves under geodesic flow on
  hyperbolic manifolds}, Duke Math. J. \textbf{148} (2009), no.~2, 251--279.
  \MR{2524496}

\bibitem[Sha10]{Shah-JAMS}
\bysame, \emph{Expanding translates of curves and {D}irichlet-{M}inkowski
  theorem on linear forms}, J. Amer. Math. Soc. \textbf{23} (2010), no.~2,
  563--589. \MR{2601043}

\bibitem[{Shi}14]{Shi}
R.~{Shi}, \emph{{Pointwise equidistribution for one parameter diagonalizable
  group action on homogeneous space}}, ArXiv e-prints (2014).

\bibitem[TW03]{TomanovWeiss}
George Tomanov and Barak Weiss, \emph{Closed orbits for actions of maximal tori
  on homogeneous spaces}, Duke Math. J. \textbf{119} (2003), no.~2, 367--392.
  \MR{1997950}

\bibitem[Ubi17]{Ubis}
Adri\'an Ubis, \emph{Effective equidistribution of translates of large
  submanifolds in semisimple homogeneous spaces}, Int. Math. Res. Not. IMRN
  (2017), no.~18, 5629--5666. \MR{3704742}

\bibitem[Wei04]{Weiss-Divergent}
B.~Weiss, \emph{Divergent trajectories on noncompact parameter spaces}, Geom.
  Funct. Anal. \textbf{14} (2004), no.~1, 94--149. \MR{2053601}

\end{thebibliography}
\bibliographystyle{amsalpha}

\end{document}